\documentclass[final,3p]{elsarticle}

\usepackage[colorlinks=true,breaklinks=true,pdftex]{hyperref}

\usepackage[latin1]{inputenc}
\usepackage{amssymb}
\usepackage{graphicx}
\usepackage{amsmath}
\usepackage{amsthm}
\usepackage{mathrsfs}
\usepackage{placeins}
\usepackage{tikz}
\usetikzlibrary{arrows.meta}
\usepackage{pgfplots}
\usepackage{comment}
\usepackage{mathtools}
\usepackage{listings}
\usepackage{xcolor}
\usepackage{float}
\usepackage[T1]{fontenc}
\usepackage{caption}
\usepackage{subcaption}
\usepackage{enumitem}
\journal{CAGD}

\biboptions{authoryear}

\newtheorem{theorem}{Theorem}
\newtheorem{lemma}[theorem]{Lemma}

\newcommand{\IDOF}{\mathbb{I}_{{\mathrm{DOF}}}}
\newcommand{\NDOF}{N_{{\mathrm{DOF}}}}
\newcommand{\PDOF}{P_{{\mathrm{DOF}}}}

\pgfplotsset{compat=1.18}

\begin{document}

\begin{frontmatter}

\title{Constructing $C^1$ limit surfaces from unstructured splines via averaging and refinement}

\author[first]{Syeda Hijab Zahra\corref{cor}}
\ead{syedahijab.zahra@oeaw.ac.at}
\author[first]{Thomas Takacs}
\ead{thomas.takacs@ricam.oeaw.ac.at}
\cortext[cor]{Corresponding author}
\address[first]{Johann Radon Institute for Computational and Applied Mathematics (RICAM), Austrian Academy of Sciences, Altenberger Str. 69, 4040 Linz, Austria}

\begin{abstract}
In this paper we present a construction for unstructured splines over quadrilateral meshes by iterative averaging and refinement. We represent the spline as a multi-patch B-spline, where the degrees of freedom are those B-spline coefficients on the quadrilateral patches that are not associated with interior edges and vertices of the mesh, i.e., their corresponding Greville points lie inside the patches. In every averaging step, we replace the remaining B-spline coefficients associated with interior edges and vertices by suitable averages of neighboring degrees of freedom. In the refinement step we apply regular splits to all patches by knot insertion. This process results in a subdivision scheme that, for degree $p=2$, is similar to the almost-$C^1$ spline construction from~\cite{takacs2023almost} and behaves similar to Doo--Sabin subdivision, cf.~\cite{doo1978behaviour}, and that can be defined for arbitrary degrees and regularities inside the patches. We derive two families of spline constructions, based on simple and coplanar averaging, respectively, and analyze their spectral properties when interpreted as subdivision schemes. Using this interpretation, we show that they are $C^1$ in the limit. Moreover, the coplanar averaging scheme produces splines that are $C^1$ at all vertices for every level of refinement, whereas the simple averaging is $C^1$ only in the limit. For both constructions, we have control over the subdominant eigenvalue, which has multiplicity two and can range between $\frac{1}{4}$ and $1$, with $\frac{1}{2}$ often being the desired option. The resulting basis functions form a partition of unity. Moreover, they form a non-negative partition of unity for suitably selected averaging parameters. 
\end{abstract}


\end{frontmatter}

\section{Introduction}

To construct splines over unstructured meshes, that is, piecewise polynomials with prescribed smoothness, is a challenging task. In geometric modeling there are two main concepts,  using constructions based on tensor-product B-spline or NURBS patches as well as subdivision surfaces. While CAD systems usually represent a geometry as a collection of B-spline or NURBS patches, cf.~\cite{piegl2012nurbs}, computer animation tools are based on subdivision surfaces. Standard quadrilateral-based subdivision schemes reproduce uniform tensor-product B-splines in regular regions of the mesh. Doo--Sabin subdivision, cf.~\cite{doo1978behaviour}, reproduces biquadratic splines and Catmull--Clark subdivision, cf.~\cite{catmull1978recursively}, bicubic splines. Both schemes generate an infinite sequence of nested spline rings around extraordinary vertices of the mesh, that is, vertices of a valence different from four. The limit surface is then $C^1$-smooth at the extraordinary vertex, cf.~\cite{peters2008subdivision} for a general analysis. Extensions that reproduce B-splines of arbitrary degree are, e.g., possible through the schemes developed in~\cite{stam2001subdivision}. 
Subdivision surface constructions can be combined with locally refinable spline constructions that are usually defined on axis-aligned grids. An example of locally refinable splines over unstructured meshes are T-NURCCs, which combine T-splines with Catmull--Clark subdivision, cf.~\cite{sederberg2003t}.

Subdivision schemes fail to generate high order approximations near extraordinary vertices, cf.~\cite{dietz2023subdivision,takacs2025approximation}. This issue can be circumvented by tuning the shrinking of the mesh during refinement near the extraordinary vertices, as developed in~\cite{wei2021tuned, ma2019subdivision}.
It is also possible to replace the infinite sequence of rings by finite caps, e.g., by using geometric continuity as in~\cite{marsala2022g1}, or by breaking the smoothness near the extraordinary vertices. An example for such an approach are almost-$C^1$ splines that were developed in~\cite{takacs2023almost}, based on~\cite{toshniwal2022quadratic}. The almost-$C^1$ construction has the same approximation properties as Doo--Sabin subdivision but does not require additional care near the extraordinary vertices, since on each level the surface is partitioned into a finite set of polynomial pieces.

Our construction is based on an important observation related to subdivision theory. That is, any tensor product B-spline space $\mathcal{S}_{p,r,\ell} \times \mathcal{S}_{p,r,\ell}$ of degree $p$ and interior regularity $r<p$ on level $\ell$ defines a subdivision process through repeated knot insertion from level $\ell$ to $\ell+1$. These refinement rules generate spline surfaces of degree $p$ with $C^r$-smoothness in both parametric directions. In regular mesh regions, the subdivision process exactly reproduces classical B-splines of arbitrary degree and regularity. Across patch interfaces, we perform averaging rules that ensure $C^1$-smoothness everywhere except near extraordinary vertices. An example of such a spline surface is shown in Figure~\ref{fig:example}.

\begin{figure}[h!]
\centering
\begin{subfigure}[b]{0.32\textwidth}\centering
\includegraphics[width=0.85\textwidth]{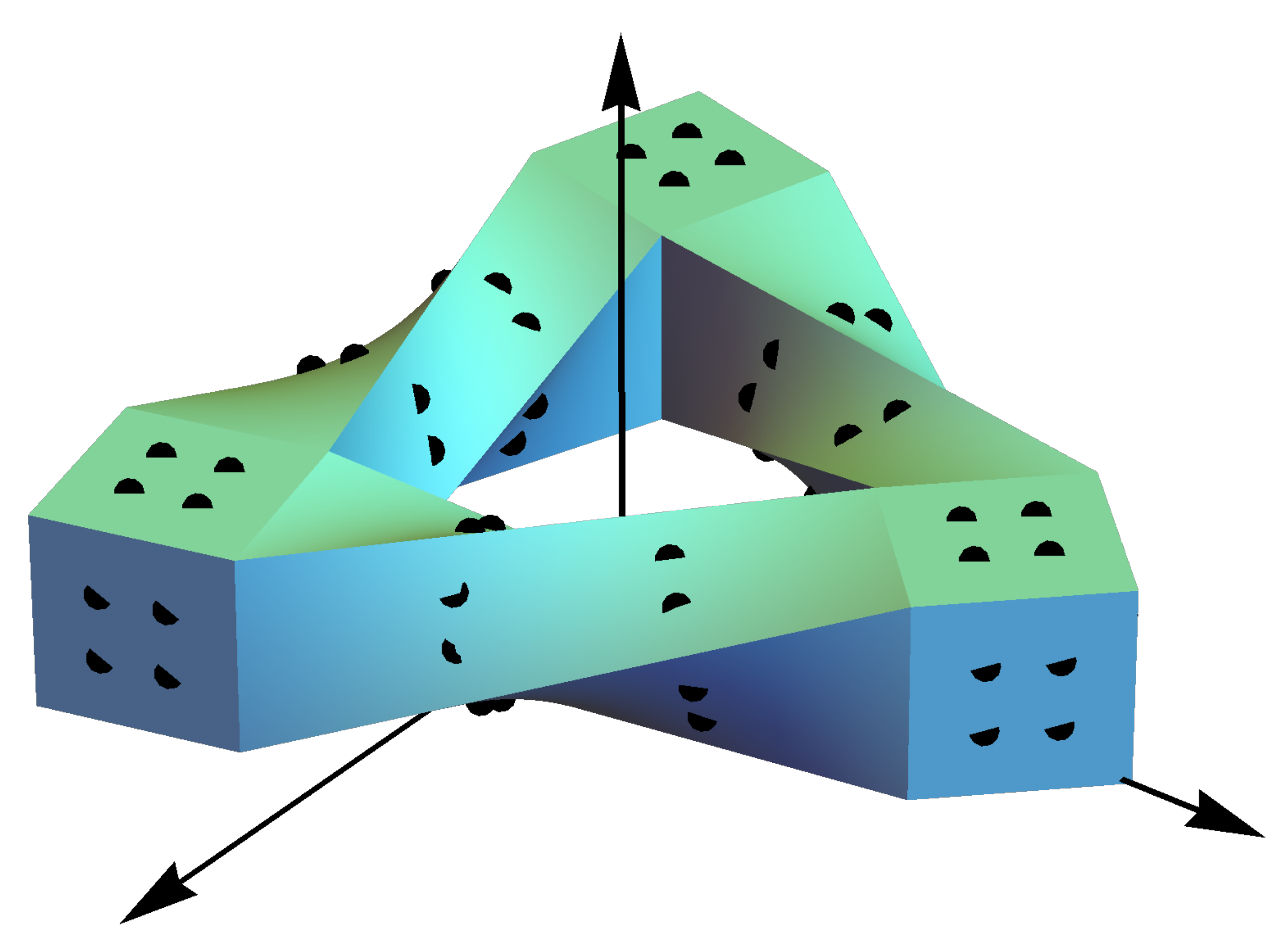}
\caption{Initial DOFs sampled from quad mesh}
\end{subfigure}
\begin{subfigure}[b]{0.32\textwidth}
\includegraphics[width=0.99\textwidth]{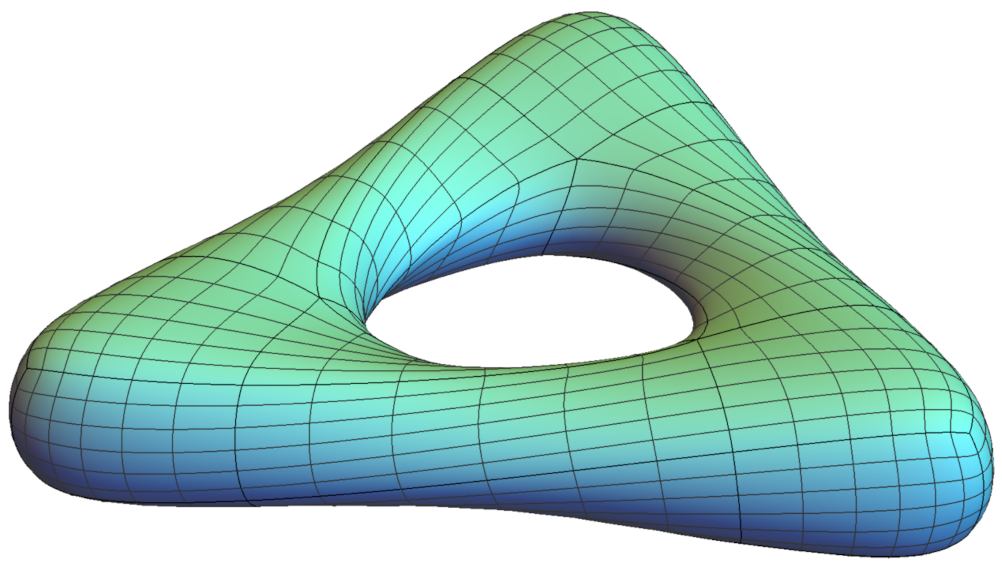}
\caption{Surface with mesh lines}
\end{subfigure}
\begin{subfigure}[b]{0.32\textwidth}
\includegraphics[width=0.99\textwidth]{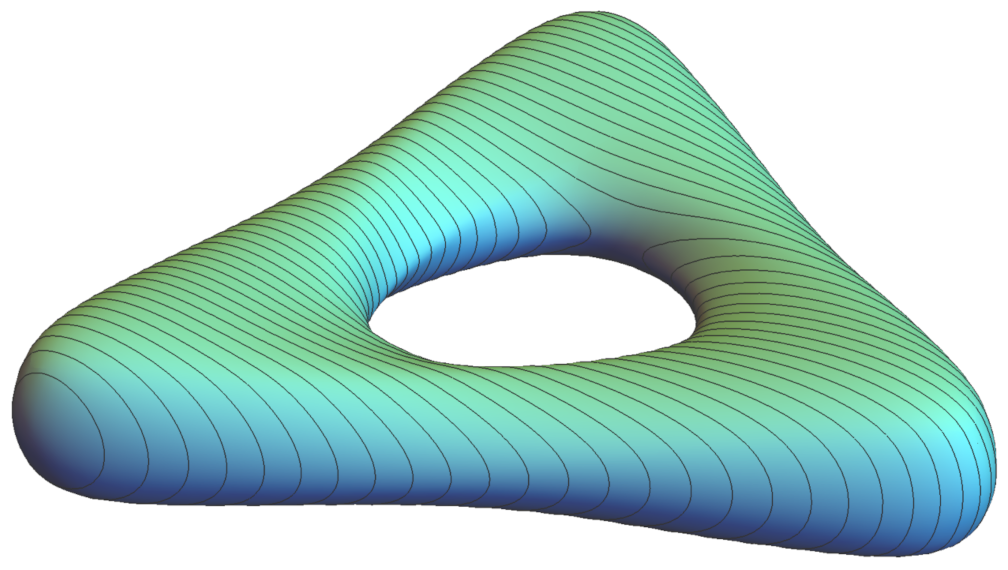}
\caption{Surface with contour lines $x=c$}
\end{subfigure}
\caption{Initial DOFs and resulting surface after three levels of refinement, using $p=3$, $r=2$, $\lambda=0.5$ with coplanar averaging.}\label{fig:example}
\end{figure}

We interpret the iterative averaging and refinement as a subdivision scheme, which generalizes tensor-product spline refinement from single patches to unstructured multi-patch domains. For degree $p=2$ the proposed constructions are very similar to Doo--Sabin subdivision. Although they do not reproduce Doo--Sabin subdivision exactly, they behave similarly in the limit. Similar constructions have been developed in the past, such as in~\cite{buchegger2016adaptively,toshniwal2017smooth,kapl2019isogeometric}.

Our work is motivated by Isogeometric Analysis (IGA) introduced by~\cite{hughes2005isogeometric}, which was developed to reduce the gap between geometric modeling (CAD) and engineering simulation. In IGA one uses the same smooth spline discretization both for geometric modeling and numerical analysis. Similarly, the work by~\cite{cirak2000subdivision} showed that subdivision surfaces can be used for thin shell simulations, which rely on globally $C^1$-smooth basis functions.

The paper is organized as follows. Section~\ref{sec:overview} introduces the notation and basic overview of the unstructured quadrilateral meshes considered in this work. Section~\ref{sec:averaging} presents the averaging techniques employed in regular regions as well as in the vicinity of extraordinary vertices. The subdivision scheme and its properties are analyzed in Section~\ref{sec:smoothness-analysis}, with particular emphasis on smoothness and convergence behavior. We summarize the construction and its properties in Section~\ref{sec:summary}. Section~\ref{sec:examples} presents examples of several geometries and numerical tests, showing the performance of the proposed construction. We conclude the paper in Section~\ref{sec:conclusions}.

\section{Overview of the construction}\label{sec:overview}

The mesh structure is defined by faces, edges, and vertices, and their relations to each other. Each face in the mesh refers to a single quadrilateral patch. The vertices of the mesh are divided into three types, corner, smooth boundary, and interior vertices. A corner vertex corresponds to a physical corner of the domain. A smooth boundary vertex lies on a boundary edge. We assume that all boundary vertices are regular, i.e., they are adjacent to one or two faces (corner or smooth boundary vertices, respectively). Interior vertices are shared by multiple adjacent faces. We call the number of faces adjacent to a vertex its \emph{valence}. The mesh edges are of two types, boundary edges which belong to only one quadrilateral face and smooth interior edges which are shared by two quadrilateral faces and lie in the interior of the mesh.

In our setup, we start from quadrilateral B\'ezier patches of degree $p$, interpreted as tensor-product B-splines without interior knots, which are refined by uniform bisection of all elements. For the basis on the refined mesh, we consider smoothness $r$, with $1 \leq r \leq p-1$, between the new elements inside a patch. This refinement is performed by knot insertion. One can interpret the mesh on each level as a piecewise B\'ezier mesh or as a multi-patch B-spline. While the B\'ezier interpretation is closer to classical subdivision, the B-spline interpretation can be implemented more efficiently in most IGA code.

Let $\mathcal{S}_{p,r,\ell}$ be the B-spline space of degree $p$ and regularity $r$ over a uniform partition of size $h=1/2^\ell$, and let $\boldsymbol{\mathcal{S}}_\ell = \mathcal{S}_{p,r,\ell} \otimes \mathcal{S}_{p,r,\ell}$ be the tensor-product B-spline space. On level $\ell \in \mathbb{Z}_0^+$, each patch
$\mathcal{B}_i$, $i = 1, \ldots, n$,
is represented by its local coefficients with respect to a global tensor product B-spline basis 
\[
({P}^{\,i}_{j,k})^K_{j,k=0}= \begin{bmatrix}
P^i_{0,0} & \dots & P^i_{0,K} \\
\vdots & \ddots & \vdots \\{P}^i_{K,0} & \dots & P^i_{K,K}
\end{bmatrix},
\]
where $K:=K(p,r,\ell)=(2^\ell-1)(p-r)+p$. Then, $K+1$ is the number of basis functions in each parametric direction at refinement level $\ell$. Note that level $\ell=0$ corresponds to a B\'ezier mesh. These patches are initially considered as B-spline surfaces, but the goal of the construction is to embed them into a globally consistent spline structure that supports refinement and later subdivision based analysis.
Coefficients as well as degrees of freedom (DOFs) are collected into global index sets
\[
\mathbb{I} = \{ (i,j,k) \mid i = 1,\ldots,n,\; j = 0,\ldots,K,\; k = 0,\ldots,K \} \quad \mbox{and}\quad \IDOF \subset \mathbb{I}.
\]
The index set $\mathbb{I}$ includes all interior and boundary coefficients, whereas the DOF index set $\IDOF$ is in general a proper subset, since not all coefficients are used as DOFs of our spline construction. For convenience, the global vector
of all coefficients is denoted by
\(
{P}_{\mathbb{I} }.
\)
We denote by $N := |\mathbb{I}| = n (K+1)^2$ the total number of coefficients and by $\NDOF := |\IDOF|$ the number of DOFs. 

We assign DOFs to mesh components:
\begin{itemize}
    \item Each face in the mesh has assigned $(K-1)^2$ DOFs, corresponding to the internal coefficients used in the patch-wise representation, i.e.,  $(P^i_{j,k})_{j,k=1}^{K-1}$. 
    \item Each boundary edge has assigned $(K-1)$ DOFs, which correspond to the boundary edge coefficients that do not lie at a vertex, e.g., for the left edge $(P^i_{j,0})_{j=1}^{K-1}$.
    \item Every corner vertex has assigned one DOF. 
\end{itemize}

%%%%%%%%%%%%%%%%%%%%%%%%
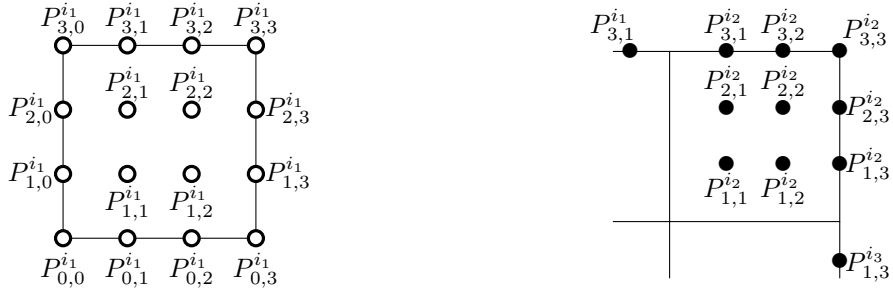
\begin{figure}[htbp]
\centering
\begin{minipage}{0.46\textwidth}
\centering
\begin{tikzpicture}[scale=0.85]

\draw (0,0) -- (3,0) -- (3,3) -- (0,3) -- cycle;

% Bottom
\node[draw=black, very thick, fill=white,circle,inner sep=2pt] at (0,0) {};
\node[draw=black, very thick, fill=white,circle,inner sep=2pt] at (1,0) {};
\node[draw=black, very thick, fill=white,circle,inner sep=2pt] at (2,0) {};
\node[draw=black, very thick, fill=white,circle,inner sep=2pt] at (3,0) {};

\node at (0,-0.5) {$P^{i_1}_{0,0}$};
\node at (1,-0.5) {$P^{i_1}_{0,1}$};
\node at (2,-0.5) {$P^{i_1}_{0,2}$};
\node at (3,-0.5) {$P^{i_1}_{0,3}$};

% Left
\node[draw=black, very thick, fill=white,circle,inner sep=2pt] at (0,1) {};
\node[draw=black, very thick, fill=white,circle,inner sep=2pt] at (0,2) {};
\node at (-0.5,1) {$P^{i_1}_{1,0}$};
\node at (-0.5,2) {$P^{i_1}_{2,0}$};
\node at (0,3.4) {$P^{i_1}_{3,0}$};

% Top
\node[draw=black, very thick, fill=white,circle,inner sep=2pt] at (0,3) {};
\node[draw=black, very thick, fill=white,circle,inner sep=2pt] at (1,3) {};
\node[draw=black, very thick, fill=white,circle,inner sep=2pt] at (2,3) {};
\node[draw=black, very thick, fill=white,circle,inner sep=2pt] at (3,3) {};

\node at (1,3.4) {$P^{i_1}_{3,1}$};
\node at (2,3.4) {$P^{i_1}_{3,2}$};

% Right
\node[draw=black, very thick, fill=white,circle,inner sep=2pt] at (3,1) {};
\node[draw=black, very thick, fill=white,circle,inner sep=2pt] at (3,2) {};
\node at (3.5,1) {$P^{i_1}_{1,3}$};
\node at (3.5,2) {$P^{i_1}_{2,3}$};
\node at (3,3.4) {$P^{i_1}_{3,3}$};

% Inner points
\node[draw=black, very thick, fill=white,circle,inner sep=2pt] at (1,1) {};
\node[draw=black, very thick, fill=white,circle,inner sep=2pt] at (2,1) {};
\node[draw=black, very thick, fill=white,circle,inner sep=2pt] at (1,2) {};
\node[draw=black, very thick, fill=white,circle,inner sep=2pt] at (2,2) {};

\node at (1,0.55) {$P^{i_1}_{1,1}$};
\node at (2,0.55) {$P^{i_1}_{1,2}$};
\node at (1,2.4) {$P^{i_1}_{2,1}$};
\node at (2,2.4) {$P^{i_1}_{2,2}$};
\end{tikzpicture}
\end{minipage}
\begin{minipage}{0.46\textwidth}
\centering
\begin{tikzpicture}[scale=0.75]

\draw (0,0) -- (3,0) -- (3,3) -- (0,3) -- cycle;

% Extensions
\draw (3,0) -- (3,-1);
\draw (0,0) -- (0,-1);
\draw (-1,3) -- (0,3);
\draw (-1,0) -- (0,0);

\node at (3,-0.7) {{\Large $\bullet$}};
\node at (-0.7,3) {{\Large $\bullet$}};

\node at (3.5,-0.75) {$P^{i_3}_{1,3}$};
\node at (-1,3.45) {$P^{i_1}_{3,1}$};

% Top
\node at (1,3) {{\Large $\bullet$}};
\node at (2,3) {{\Large $\bullet$}};
\node at (3,3) {{\Large $\bullet$}};

\node at (1,3.45) {$P^{i_2}_{3,1}$};
\node at (2,3.45) {$P^{i_2}_{3,2}$};

% Right
\node at (3,1) {{\Large $\bullet$}};
\node at (3,2) {{\Large $\bullet$}};
\node at (3.5,1) {$P^{i_2}_{1,3}$};
\node at (3.5,2) {$P^{i_2}_{2,3}$};
\node at (3.45,3.35) {$P^{i_2}_{3,3}$};

% Inner points
\node at (1,1) {{\Large $\bullet$}};
\node at (2,1) {{\Large $\bullet$}};
\node at (1,2) {{\Large $\bullet$}};
\node at (2,2) {{\Large $\bullet$}};

\node at (1,0.55) {$P^{i_2}_{1,1}$};
\node at (2,0.55) {$P^{i_2}_{1,2}$};
\node at (1,2.45) {$P^{i_2}_{2,1}$};
\node at (2,2.45) {$P^{i_2}_{2,2}$};
\end{tikzpicture}
\end{minipage}
\caption{On the left: B-spline coefficients on a single element (of degree $p=3$). On the right: Global DOFs for $p=3$, that are assigned to faces, boundary edges and corner vertices.}
 \label{fig:coefficients-and-DOFs}
\end{figure}

A visualization of the coefficients and the global DOF structure is given in Figure~\ref{fig:coefficients-and-DOFs}. 
This selection is performed by a DOF-extraction matrix
\[
D \in \{0,1\}^{\NDOF \times N},
\]
where in each row there is exactly one entry $1$ and all other are $0$. In each column there is either exactly one entry $1$ or all entries in the column are $0$. Thus, $D$ extracts exactly the DOFs 
\[
{\PDOF} = D\,{P}_{\mathbb{I}}.
\]
These DOFs are then used in the averaging step, which  consists of replacing all coefficients that lie on interfaces or at vertices by suitable averages of their neighbors. Symbolically, we write
\[
\widetilde{{P}}_{\mathbb{I} } = A\,{\PDOF},
\]
where 
\[
A \in \mathbb{R}^{N \times \NDOF},
\]
such that all row sums of $A$ are $1$. 
All DOFs remain unchanged, whereas interface coefficients are modified so that
their values coincide across adjacent patches and they form $C^1$-smooth or approximately $C^1$-smooth transitions over patch interfaces. Vertex coefficients are replaced by averages of their neighbors, which ensure continuity at the vertex and (approximate) $C^1$-smoothness across patch interfaces adjacent to the vertex. With the help of this smoothing, refinement can subsequently be applied patch-wise while preserving global continuity. An overview of the process is shown in Figure~\ref{fig:overview}.

    \begin{figure}[ht]
    \begin{minipage}[h]{0.22\textwidth}
        \begin{center}
\includegraphics[width=0.87\textwidth]{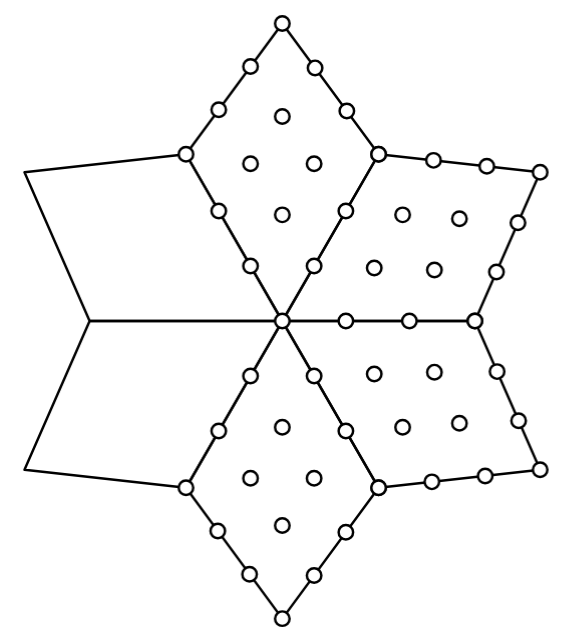}
        \end{center} 
    \end{minipage}
   $ \overset{\text{D}}{\rightarrow}$
    \begin{minipage}[h]{0.22\textwidth}
        \begin{center} \includegraphics[width=0.9\textwidth]{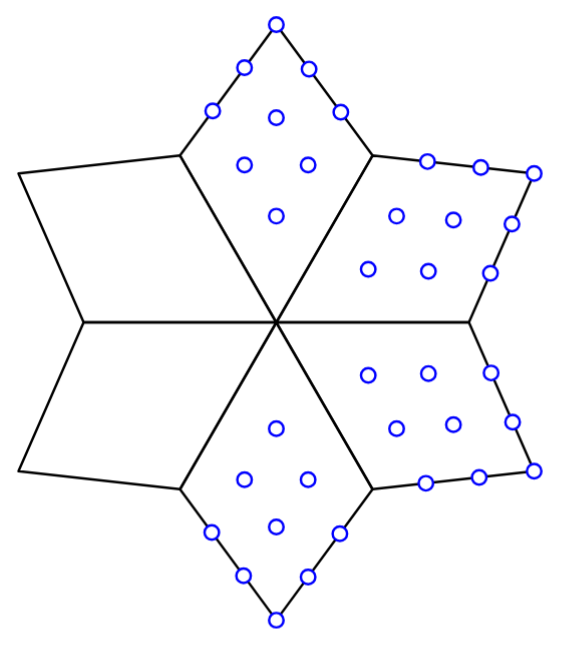}
        \end{center} 
    \end{minipage}
    $ \overset{\text{A}}{\rightarrow}$
    \begin{minipage}[h]{0.22\textwidth}
        \begin{center}  \includegraphics[width=0.9\textwidth]{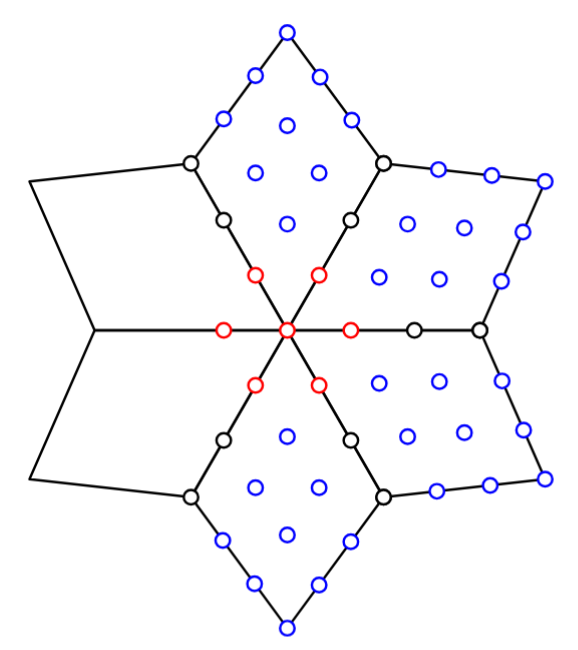}
        \end{center}
    \end{minipage}
    $ \overset{\text{R}}{\rightarrow}$
    \begin{minipage}[h]{0.22\textwidth}
        \begin{center}  \includegraphics[width=0.9\textwidth]{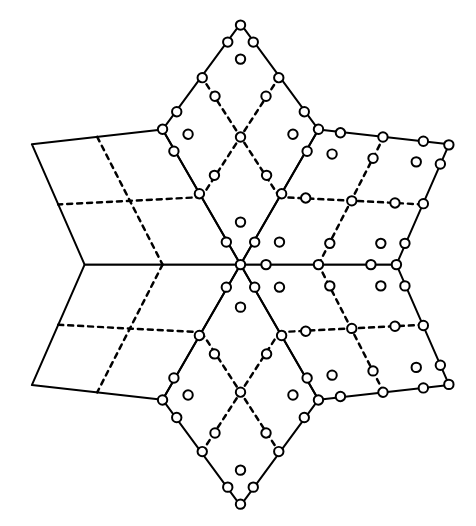}
        \end{center}
    \end{minipage}
    \caption{From left to right: B-spline coefficients on the mesh of level $\ell$; selected DOFs (in blue); new B-spline coefficients computed by averaging (in red and black); refined B-spline coefficients on the mesh of level $\ell+1$.
    }
    \label{fig:overview}
    \end{figure}

Each patch carries its own B-spline refinement operator, which is the Kronecker-product of two univariate operators corresponding to refinement along the two parametric directions. The refinement of the global control structure is performed patch by patch, applied to
\(\widetilde{{P}}_{\mathbb{I}}\).
In compact form, for a global refinement step we assemble
\[
{P}_{\mathbb{I} }^{1} = R\,A\,D\,{P}_{\mathbb{I} }^{0},
\]
where \(D\) represents the DOF extraction operator, \(A\) is the averaging operator, and \(R\) applies the tensor-product B-spline refinement operator to each patch. We can now define DOFs on each level $\ell=0,1,\ldots,\infty$ as \(\PDOF^\ell\), which satisfy the subdivision relation
\[
\PDOF^{\ell+1} = D^{\ell+1} R^\ell A^\ell  \PDOF^{\ell}, \quad \; \ell = 0,1,2,\ldots,\infty,
\]
where \(D^{\ell+1} R^\ell A^\ell\) is the subdivision matrix from level $\ell$ to $\ell+1$. The matrix $R^\ell$ is the patch-wise refinement matrix from the tensor-product B-spline space $\boldsymbol{\mathcal{S}}_\ell$ to $\boldsymbol{\mathcal{S}}_{\ell+1}$, obtained by knot insertion. 
This subdivision matrix will later be analyzed around extraordinary vertices in Section~\ref{sec:smoothness-analysis}.

\section{Averaging strategies}\label{sec:averaging}

In this section we describe how the B-spline representation on a patch is obtained using averaging. Once the degrees of freedom are assigned to corner vertices,  boundary edges, and faces, the corresponding local coefficients for the B-spline patches are computed by suitable averaging of the DOFs associated with the surrounding mesh components.

\subsection{Interior edge averaging}\label{sec:Interior edge averaging}

For all interior edges, we employ a standard averaging approach. The edge coefficients are determined by averaging the DOFs of the two neighboring faces, as illustrated in Figure~\ref{fig:averaging-a}, where $i_1$ and $i_2$ are the faces that share the edge. Then, the coefficients along the interface are computed as
\begin{equation*}
    P^i_{0,k} := \frac{1}{2}(P^{i_1}_{1,k}+P^{i_2}_{1,k}), \quad \mbox{for all }i\in\{i_1,i_2\} \mbox{ and }k\in\{k_{\min},\ldots,k_{\max}\},
\end{equation*}
where $P^{i_1}_{1,k}$ and $P^{i_2}_{1,k}$ denote the first row of DOFs of the adjacent faces. If there is an extraordinary vertex at the bottom of the edge, we set $k_{\min}=1$, otherwise, $k_{\min}=0$. If there is an extraordinary vertex at the top, we set $k_{\max} = K-1$, otherwise, $k_{\max} = K$. However, if the degree is $p=2$, $\ell=0$ and both vertices are extraordinary, we set $k_{\min} = k_{\max} = 1$ and perform standard averaging at that edge, see Figure~\ref{fig:vertex-averaging-b}.

\subsection{Regular vertex averaging}

First, we consider a regular vertex, i.e., a vertex of valence four. The coefficient at the interior vertex is computed as the average of the nearest DOFs of its adjacent faces. Let $i_1$, $i_2$, $i_3$ and $i_4$ be the faces that share the vertex. Without loss of generality, we denote the coefficients near the vertex as in Figure~\ref{fig:averaging-b}. Then, the vertex coefficient is given by
\begin{equation}\label{eq:regular-vertex-averaging}
P^{i_k}_{0,0} := \frac{1}{4} \sum_{j\in\{i_1,i_2,i_3,i_4\}} P^j_{1,1},\quad \mbox{for all }k\in\{1,2,3,4\}.
\end{equation}
Moreover, the interior edge averaging from Subsection~\ref{sec:Interior edge averaging} yields
\[
P^{i_k}_{1,0} = \frac{1}{2} (P^{i_k}_{1,1}+P^{i_{k-1}}_{1,1}) = P^{i_{k-1}}_{0,1},\quad \mbox{for all }k\in\{1,2,3,4\},
\]
where we consider all indices $k$ modulo four. This averaging enables a $C^1$-smooth transition at the vertex.

\subsection{Boundary vertex averaging}

In the boundary construction, a corner vertex, representing an endpoint of two boundary edges, is assigned a single degree of freedom corresponding to the function value at that vertex. And for a smooth boundary vertex of valence two, the vertex coefficient is defined as the average of the degrees of freedom associated with the adjacent boundary edges as in Figure~\ref{fig:averaging-c}. Let $i_1$ and $i_2$ be the adjacent faces and let $P^{i_1}_{0,1}$ and $P^{i_2}_{0,1}$ denote the adjacent DOFs of the two incident boundary edges. Then, the boundary vertex coefficient is given by
\[
P^{i_k}_{0,0} = \frac{1}{2} (P^{i_1}_{0,1} + P^{i_2}_{0,1}), \quad \mbox{for all }k\in\{1,2\},
\]
which provides a $C^1$-smooth placement of the vertex along the boundary.

    \begin{figure}
    \centering
    \subfloat[Interior edge\label{fig:averaging-a}]{\includegraphics[width=0.2\textwidth]{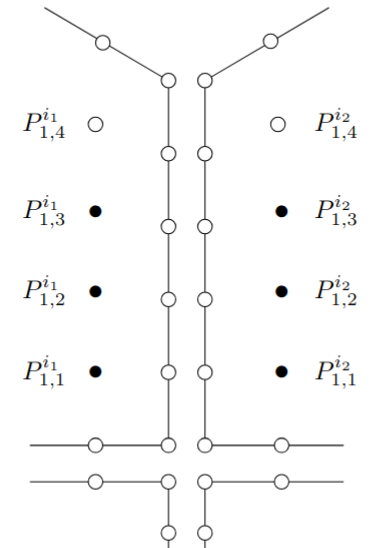}} \quad
    \subfloat[Regular interior vertex\label{fig:averaging-b}]{\includegraphics[width=0.25\textwidth]{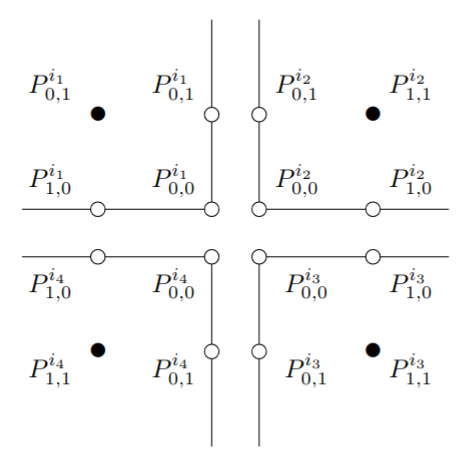}} \quad
    \subfloat[Boundary vertex\label{fig:averaging-c}]{\includegraphics[width=0.28\textwidth]{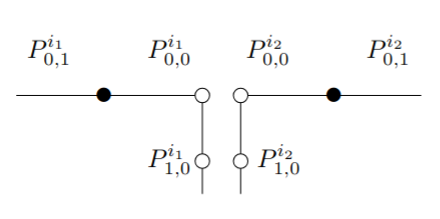}}
    \caption{(a) Smooth interior edge coefficients are defined by the corresponding face DOFs. (b) Regular interior vertex coefficients are defined by the adjacent face DOFs. (c) Smooth boundary vertex coefficients are calculated by the neighboring edge DOFs.}
    \label{fig:averaging}
    \end{figure}
    
\subsection{Extraordinary vertex averaging}

Note that we not only need to compute the coefficients corresponding to the extraordinary vertices, but also the neighboring edge coefficients. We consider an extraordinary vertex of valence $\nu \in\{3,5,6,\ldots\}$. As depicted in Figure~\ref{fig:vertex-averaging-a}, we have DOFs
\[
    Q_k := P^{i_k}_{1,1}, \quad \mbox{ for all }k\in\{1,\ldots, \nu\}.
\]
The coefficients corresponding to the vertex are, in analogy to~\eqref{eq:regular-vertex-averaging}, computed as
\[
P^{i_j}_{0,0} := \frac{1}{\nu} \sum_{k=1}^{\nu} Q_k,\quad \mbox{for all }j\in\{1,\ldots,\nu\}.
\]
In addition, we need to determine the coefficients
\[
    E_{k,k+1} := P^{i_k}_{0,1} = P^{i_{k+1}}_{1,0}, \quad \mbox{ for all }k\in\{1,\ldots, \nu\},
\]
which we compute by averaging as
\begin{equation}\label{eq:ev-averaging}
    \left(\begin{array}{c}
        E_{1,2}\\ E_{2,3} \\ \vdots \\ E_{\nu,1}
    \end{array}\right) = A^\nu \left(\begin{array}{c}
        Q_1\\ Q_2 \\ \vdots \\ Q_\nu
    \end{array}\right) =
    \left(\begin{array}{cccc}
        a_0 & a_1 & \ldots & a_{\nu-1} \\ a_{\nu-1} & a_0 & \ldots & a_{\nu-2} \\ \vdots & & \ddots & \vdots \\ a_1 & a_2 & \ldots & a_{0}
    \end{array}\right)
    \left(\begin{array}{c}
        Q_1\\ Q_2 \\ \vdots \\ Q_\nu
    \end{array}\right).
\end{equation}
Here we require a convex combination, i.e., 
\begin{equation}
    \sum_{j=0}^{\nu-1} a_j = 1, \label{convexity}
\end{equation}
and assume the symmetry condition
\begin{equation}
    a_{i+1}=a_{-i}, \quad\mbox{ for all } i\geq 0, \label{sy}
\end{equation}
with indices modulo $\nu$.

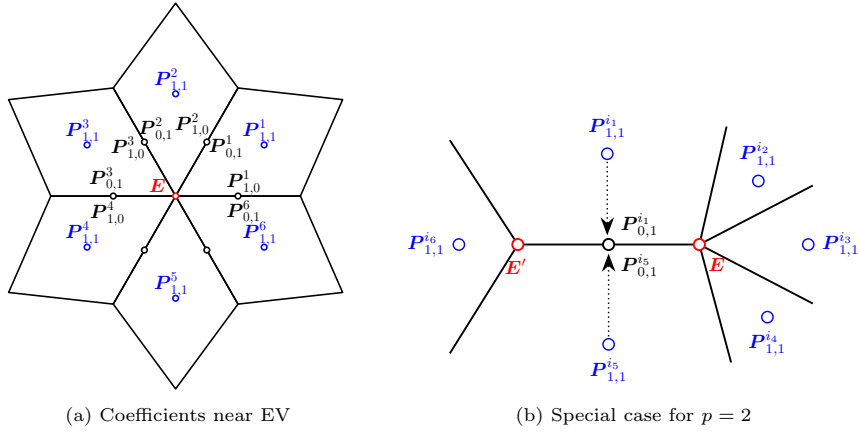
\begin{figure}
   \begin{center}
  \subfloat[Coefficients near EV\label{fig:vertex-averaging-a}]{\scalebox{0.5}{\begin{tikzpicture}[scale=3, line cap=round, line join=round]
\usetikzlibrary{calc}
\tikzset{ 
  edgedot/.style={circle, very thick, draw=black, fill=white, inner sep=1.5pt},
  edgelabel/.style={text=black},  evlabel/.style={text=red},  
  evdot/.style={circle, very thick, draw=red, fill=white, inner sep=1.5pt},
  blackdot/.style={circle, very thick, draw=blue, fill=white, inner sep=1.5pt}, facelabel/.style={text=blue}
}

\def\r{1.10}
\def\R{1.70}

% geometry
\coordinate (O) at (0,0);

% inner points
\foreach \i in {0,1,2,3,4,5}{
  \coordinate (P\i) at ({60*\i}:\r);
}

% outer points
\foreach \i in {0,1,2,3,4,5}{
  \coordinate (Q\i) at ({60*\i+30}:\R);
}

% patches
\foreach \i in {0,1,2,3,4,5}{
  \pgfmathtruncatemacro\next{mod(\i+1,6)}
  \draw[very thick] (O)--(P\i)--(Q\i)--(P\next)--cycle;
}

% face DOFs
\foreach \i in {0,1,2,3,4,5}{
  \pgfmathtruncatemacro\next{mod(\i+1,6)}
  \path
    coordinate (A) at ($(O)!0.5!(P\next)$)
    coordinate (B) at ($(P\i)!0.5!(Q\i)$)
    coordinate (C) at ($(A)!0.5!(B)$);
  \pgfmathtruncatemacro\face{\i+1}
  \node[blackdot] at (C) {};
  \node[facelabel] at ($(C)+(-0.038,0.12)$) {{\Large$\boldsymbol{P}^{\face}_{1,1}$}};
}

%inner edge dofs
\path coordinate (E1) at ($(O)!0.5!(P0)$);
\node[edgedot] at (E1) {};
\node[edgelabel] at ($(E1)+(0.05,0.12)$) {\Large $\boldsymbol{P}^{1}_{1,0}$};
\node[edgelabel] at ($(E1)+(0.05,-0.14)$) {\Large $\boldsymbol{P}^{6}_{0,1}$};

\path coordinate (E2) at ($(O)!0.5!(P1)$);
\node[edgedot] at (E2) {};
\node[edgelabel] at ($(E2)+(-0.13,0.15)$) {\Large $\boldsymbol{P}^{2}_{1,0}$};
\node[edgelabel] at ($(E2)+(0.17,-0.05)$) {\Large $\boldsymbol{P}^{1}_{0,1}$};

\path coordinate (E3) at ($(O)!0.5!(P2)$);
\node[edgedot] at (E3) {};
\node[edgelabel] at ($(E3)+(-0.145,-0.06)$) {\Large $\boldsymbol{P}^{3}_{1,0}$};
\node[edgelabel] at ($(E3)+(0.10,0.11)$) {\Large $\boldsymbol{P}^{2}_{0,1}$};

\path coordinate (E4) at ($(O)!0.5!(P3)$);
\node[edgedot] at (E4) {};
\node[edgelabel] at ($(E4)+(-0.06,0.15)$) {\Large $\boldsymbol{P}^{3}_{0,1}$};
\node[edgelabel] at ($(E4)+(-0.05,-0.155)$) {\Large $\boldsymbol{P}^{4}_{1,0}$};

\path coordinate (E5) at ($(O)!0.5!(P4)$);
\node[edgedot] at (E5) {};

\path coordinate (E6) at ($(O)!0.5!(P5)$);
\node[edgedot] at (E6) {};

\node[evdot] at (O) {};
\node[evlabel] at ($(O)+(-0.16,0.09)$) {\Large $\boldsymbol{E}$};

\end{tikzpicture} }}\qquad
\subfloat[Special case for $p=2$\label{fig:vertex-averaging-b}]{\scalebox{0.6}{\begin{tikzpicture}[
  very thick,
  vtx/.style={circle, draw, fill=white,very thick, inner sep=2.5pt},
  facedof/.style={circle, draw=blue, fill=white, very thick, inner sep=2.5pt, line width=0.9pt},
  facedoflabel/.style={blue}
]

% main nodes
\coordinate (L) at (0,0);
\coordinate (M) at (2,0);
\coordinate (R) at (4,0);

% main edge
\draw (L)--(M)--(R);

% left edges
\draw (L) -- ++(-1.5,  2.3);
\draw (L) -- ++(-1.5, -2.4);

% right edges
\draw (R) -- ++(0.6,  2.6);
\draw (R) -- ++(2.5,  1.3);
\draw (R) -- ++(2.5, -1.3);
\draw (R) -- ++(0.7, -2.6);

% arrows
\draw[line width=0.9pt, shorten <=5pt, -{Stealth[length=4mm,width=3mm]}, draw=black, dotted]
(1.97,2) -- ++(0.001,-1.8);

\draw[line width=0.9pt, shorten <=5pt, -{Stealth[length=4mm,width=3mm]}, draw=black, dotted]
(2,-2.2) -- ++(0,2);

% black dofs
\draw[fill=white,very thick, circle, draw=red] (L) circle (3.5pt);
\node[label={[red, scale=1.2] below left:\normalsize $\boldsymbol{E}'$}]  at (0.5,0.0) {};

\draw[fill=white,very thick, circle, draw=black] (M) circle (3.5pt);
\node[label={[black, scale=1.2] below left:\normalsize $\boldsymbol{P}^{i_1}_{0,1}$}]  at (3.4,1.0) {};
\node[label={[black, scale=1.2] below left:\normalsize $\boldsymbol{P}^{i_5}_{0,1}$}]  at (3.4,0.1) {};

\draw[fill=white,very thick, circle, draw=red] (R) circle (3.5pt);
\node[label={[red, scale=1.2] below right:\normalsize $\boldsymbol{E}$}] at (3.9,0.0) {};

% left face
\node[facedof,label={[facedoflabel, scale=1.2] left:\normalsize $\boldsymbol{P}^{i_6}_{1,1}$}] at (-1.3,0) {};

% top middle face
\node[facedof,label={[facedoflabel, scale=1.2] above:\normalsize $\boldsymbol{P}^{i_1}_{1,1}$}] at (1.97,2) {};

% bottom middle face
\node[facedof,label={[facedoflabel, scale=1.2]below:\normalsize $\boldsymbol{P}^{i_5}_{1,1}$}] at (2,-2.2) {};

% right faces
\node[facedof,label={[facedoflabel, scale=1.2]above:\normalsize $\boldsymbol{P}^{i_2}_{1,1}$}] at (5.3,1.4) {};
\node[facedof,label={[facedoflabel, scale=1.2]right:\normalsize $\boldsymbol{P}^{i_3}_{1,1}$}] at (6.4,0) {};
\node[facedof,label={[facedoflabel, scale=1.2] below:\normalsize $\boldsymbol{P}^{i_4}_{1,1}$}] at (5.5,-1.6) {};

\end{tikzpicture}}}

\end{center}

    \caption{(a) An extraordinary vertex $E$ together with its adjacent faces, relevant DOFs (blue) and coefficients computed by averaging (black and red). (b) Special case to compute edge coefficients for degree $p=2$. In that case, there may be conflicting averaging rules from the two neighboring EVs. Thus, we set the coefficient at the edge to be the average of the two adjacent DOFs.}
    \label{fig:vertex-averaging}
\end{figure}

In the following, we introduce two different averaging strategies, \emph{simple averaging} and \emph{coplanar averaging}. While simple averaging has easier to compute weights, coplanar averaging results in a spline surface that is $G^1$ at the extraordinary vertex. Both schemes result in $C^1$ limit surfaces and allow a scaling of the size of elements near the extraordinary vertex.

\subsubsection{Simple averaging}

In the simple averaging strategy, all new edge coefficients are computed as weighted averages between the center $E$ and the midpoints $(P_{1,1}^i+P_{1,1}^{i+1})\frac12$, cf. Figure~\ref{fig:vertex-averaging-a}. The weights $(a_0,\ldots, a_{\nu-1})$ for general valence $\nu$ are therefore defined as 
\[
\begin{aligned}
    \left(a_0,\ldots, a_{\nu-1}\right) &:= 
(1-\mu)\left( \tfrac{1}{\nu}, \ldots, \tfrac{1}{\nu} \right)
+ \mu \left( \tfrac{1}{2}, \tfrac{1}{2}, 0, \ldots, 0 \right)\\
 &\ = 
\left( \tfrac{1}{\nu}, \ldots, \tfrac{1}{\nu} \right)
+ \mu  \left( \tfrac{\nu-2}{2\nu}, \tfrac{\nu-2}{2\nu}, -\tfrac{1}{\nu}, \ldots, -\tfrac{1}{\nu} \right).
\end{aligned}
\]
For $\mu \in[0,1]$ the averaged B\'ezier coefficients are all convex combinations of the DOFs, since all averaging weights are non-negative in that case. 

\subsubsection{Coplanar averaging}

In the coplanar averaging strategy, the weights $(a_0,\ldots, a_{\nu-1})$ are defined as 
\begin{equation}\label{eq:coplanar-weights}
a_j := 
\tfrac{1}{\nu}
+ \mu \frac{\cos\left(\frac{2\pi j}{\nu}-\frac{\pi}{\nu}\right)}{\cos\left(\frac{\pi}{\nu}\right)} , \quad \mbox{for all }j\in\{0,\ldots, \nu-1\}.
\end{equation}
Note that all weights are non-negative for $\mu \in [-\frac{1}{\nu},\frac{1}{\nu}]$, if $p$ is even, and for $\mu \in [-\frac{1}{\nu},\frac{\cos({\pi}/{\nu})}{\nu}]$, if $p$ is odd. That means, the averaged B\'ezier coefficients are all convex combinations of the DOFs in those cases.
\begin{lemma}
    With the choice in~\eqref{eq:coplanar-weights}, the matrix $A^\nu \in \mathbb{R}^{\nu\times\nu}$ has rank three. Therefore, the points $E_{k,k+1}$ are coplanar. 
\end{lemma}
\begin{proof} 
Let $(a_0,\ldots, a_{\nu-1})$ be a vector of weights satisfying \eqref{convexity} and \eqref{sy} that define a circulant matrix $A^\nu$. The rank of $A^\nu$ is $\nu-d$, where $d$ is the degree of the $\mbox{\rm polynomial gcd}(f(x),x^\nu-1)$ where 
\[
    f(x)=a_0+a_{\nu-1}x+a_{\nu-2}x^2 + \ldots + a_1 x^{\nu-1}.
\]
The polynomial $x^\nu-1$ has $\nu$ complex roots $\eta_j = \exp(\frac{2\pi {\rm i}}{\nu}j)$ and a decomposition into $\nu$ linear factors $(x-\eta_j)$.
To obtain a matrix $A^\nu$ of rank $3=\nu-d$, we need to find real coefficients $a_0,\ldots,a_{\nu-1}$, such that 
\[
    \deg\big({\rm gcd}(f(x), (x-\eta_0)(x-\eta_1)\ldots(x-\eta_{\nu-1}))\big) = \nu-3.
\]
We have $\eta_0 = 1$ and pairs of complex conjugates $\eta_i$ and $\eta_{\nu-i}$. Thus, we can define $f(x)$ to be 
\[
    f(x) = \mathcal{X}(x) \frac{(x-\eta_1)\ldots(x-\eta_{\nu-1})}{(x-\eta_j)(x-\eta_{\nu-j})}(b_0+b_1x+b_2x^2),
\]
for some $j\in\{1,\ldots,\lfloor\nu/2\rfloor\}$ and some $b_0,b_1,b_2 \in\mathbb{R}$. If $\nu$ is odd, $\mathcal{X}(x)\equiv 1$. For even $\nu$ we have $\mathcal{X}(x) \in \{x+1,x-1\}$. The conditions \eqref{convexity} and \eqref{sy} determine all but one of the parameters $b_0,b_1,b_2$. The only choice that results in a geometrically desirable configuration, i.e., one that does not create self intersections or overlaps, is $j=1$ and $\mathcal{X}(x)=x+1$, for even $\nu$, which is precisely the $\mu$-dependent family presented above.
\end{proof}
The choice of reasonable weights is unique, up to the factor $\mu$, and numerically verified for all valences reported below. Thus, for coplanar averaging, we have the following weights $(a_0 , a_1 , \ldots , a_{\nu-1})$ for different valences:
\[
\begin{array}{rl}
\nu=3: & \left( \tfrac{1}{3}, \ldots, \tfrac{1}{3} \right)
+ \mu \left( 1,1,-2 \right) \\[5pt]
\nu=5: & \left( \tfrac{1}{5}, \ldots, \tfrac{1}{5} \right)
+ \mu \left( 1,1,\tfrac{\sqrt{5}-3}{2}, {1-\sqrt{5}}, \tfrac{\sqrt{5}-3}{2} \right) \\[5pt]
\nu=6: &\left( \tfrac{1}{6}, \ldots, \tfrac{1}{6} \right)
+ \mu \left( 1,1,0,-1, -1,0 \right)\\[5pt]
\nu=7: & \left( \tfrac{1}{7}, \ldots, \tfrac{1}{7} \right)
+ \mu \left( 1,1,0.2469,-0.6920,-1.1099,-0.6920,0.2469 \right)\\[5pt]
\nu=8:&\left( \tfrac{1}{8}, \ldots, \tfrac{1}{8} \right)
+ \mu \left( 1,\,1,\,-1+\sqrt{2},\,1-\sqrt{2},\,-1,\,-1,\,1-\sqrt{2},\,-1+\sqrt{2} \right)\\[5pt]
\nu=9:&\left( \tfrac{1}{9}, \ldots, \tfrac{1}{9} \right)
+ \mu \left( 1,\,1,\,0.5320,\,-0.1847,\,-0.8152,\,-1.0641,\,-0.8152,\,-0.1847,\,0.5320 \right)\\[5pt]
\nu=10:&\left( \tfrac{1}{10}, \ldots, \tfrac{1}{10} \right)
+ \mu \left( 1,\,1,\,\tfrac12(-1+\sqrt{5}),\,0,\,\tfrac12(1-\sqrt{5}),\, -1,\, -1,\, \tfrac12(1-\sqrt{5}),\,0,\,\tfrac12(-1+\sqrt{5})
 \right).
\end{array}
\]

\subsubsection{Special case for $p=2$}

Consider the case $p=2$. If we are at level $l=0$ and we have an edge containing two extraordinary vertices, as in Figure~\ref{fig:vertex-averaging-b}, then there are conflicting averaging rules for the coefficient assigned to that edge. To avoid this special case for $p=2$, we first perform simple averaging with $\mu=1$ globally in all our examples. Since, after a single refinement step, all extraordinary vertices are separated by at least two coefficients assigned to the edge, there is no conflict anymore. Thus, this case can only appear at the initial level and has no effect on the convergence analysis of the subdivision scheme.

\section{Convergence analysis following the subdivision interpretation}\label{sec:smoothness-analysis}

In this section, we interpret the spline construction as a subdivision scheme. This interpretation can be used to investigate the smoothness properties of the resulting limit spline surface in the neighborhood of an extraordinary vertex. Since the global spline spaces are constructed as multi-patch B-splines, in regular regions away from interfaces we reproduce tensor-product B-splines of degree $p$ and regularity $r$. Across interfaces sufficiently far away from extraordinary vertices, the splines are $C^1$-smooth. What remains is the smoothness analysis of the limit surface in a neighborhood of the extraordinary vertex, which can be done by studying the local subdivision matrix.

We start from a neighborhood of an extraordinary vertex $Q_0$ of valence $\nu$, with $\PDOF^*\subset \PDOF^\ell$ DOFs for some level $\ell \in \mathbb{Z}_0^+$. On each patch, we take the DOFs
\[
({P}^{\,i}_{j,k})^{K^*}_{j,k=1}= \begin{bmatrix}
P^i_{1,1} & \dots & P^i_{1,K^*} \\
\vdots & \ddots & \vdots \\{P}^i_{K^*,1} & \dots & P^i_{K^*,K^*}
\end{bmatrix}.
\]
with $K^*:=K-p+1=(2^\ell-1)(p-r)+1$. In our construction the subdivision of the control mesh is performed by an averaging followed by a local tensor product refinement of each patch. Thus, the subdivision matrix $S$ is defined as 
\[
 S^\ell:= D^* R^\ell A^\ell,
\]
where $D^*$ extracts the DOFs of level $\ell+1$ from the same sized neighborhood as the initial DOFs $\PDOF^*$ of level $\ell$, such that $S$ is a square matrix. Following~\cite{peters2008subdivision}, we study the eigenvalues of $S$ to determine the smoothness properties of the resulting subdivision scheme.

\subsection{Patch-wise $C^1$-smooth spline spaces}

Let the neighboring DOFs of the extraordinary vertex be denoted by $Q_1, \ldots, Q_\nu$, with $Q_i = P^i_{1,1}$, in counterclockwise order. 
Independent of the construction, the subdivision matrix $S$ has an eigenvalue $1$, which results in the limit position of the extraordinary vertex at
\[
    E=\frac{1}{\nu} \sum_{j=1}^\nu Q_j.
\]
Depending on the construction, we obtain a local subdivision matrix $S^{0}(\nu,\mu) \in \mathbb{R}^{\nu \times \nu}$, which is a circulant matrix, such that the new DOFs $Q^{\text{new}}_1, \ldots, Q^{\text{new}}_\nu$ are computed as
\[
\begin{bmatrix} Q^{\text{new}}_1 \\ \vdots \\ Q^{\text{new}}_v \end{bmatrix} = S^{0}(\nu,\mu)
\begin{bmatrix} Q_1 \\ \vdots \\ Q_v \end{bmatrix}.
\]
The dependence of the new DOF $Q^{\text{new}}_1$ on the old DOFs is visualized in Figure~\ref{fig:subdivision-relation}.

    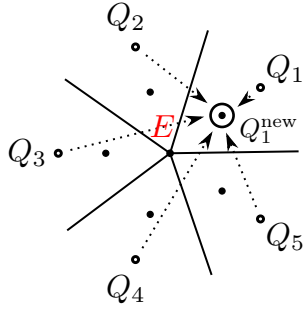
\begin{figure}[h]
        \begin{center}\scalebox{1}{
\begin{tikzpicture}[scale=1,  ultra thick]

% parameters
\def\R{1.7}        
\def\n{5}          
\def\gap{1.7}      
\def\dotr{0.04}   
\def\cut{1.2}     
\def\span{1.3}    

% radial lines
\foreach \k in {0,...,4} {
    \pgfmathsetmacro{\ang}{360/\n*\k}
    \draw[line width=0.75pt] (0,0) -- ++({\ang+\gap/2}:\R);
}

% midpoint
\foreach \k in {0,...,4} {
   \pgfmathsetmacro{\ang}{360/\n*\k + 360/(2*\n)}
   \filldraw[line width=0.45pt]
      ({\ang}:\R/2) circle (\dotr);
}

% outer ring for k=0
\pgfmathsetmacro{\angzero}{360/(2*\n)}
\draw[line width=1pt]
   ({\angzero}:\R/2) circle (3.8*\dotr);
 \node [scale=1.0, black]  at ({\angzero}:\R/2 + 0.55)[yshift=-15pt,xshift=5pt]
      {\normalsize${Q_1^{\text{new}}}$};

\foreach \k/\lab in {1/{2},0/{1},4/{5},3/{4},2/{3}} {

    \pgfmathsetmacro{\ang}{360/\n*\k + 360/(2*\n)}

    \draw[fill=white, line width=1.1pt]
      ({\ang}:\R/1.15) circle (\dotr);

    \node [scale=1.2]  at ({\ang}:\R/0.9)
      {\normalsize ${Q}_{\lab}$};     
}

% arrows
%%% form Q1 to Qnew
\draw[ line width=0.75pt, shorten <=5pt, -{Stealth[length=2.5mm,width=1.5mm]},dotted]
({360/\n*1 + 360/(2*\n)}:\R/1.15)
-- ++(-38:1.21);
%%% from Q2 to Qnew
\draw[line width=0.75pt, shorten <=5pt, -{Stealth[length=2.5mm,width=1.5mm]},dotted]
({360/\n*2 + 360/(2*\n)}:\R/1.15+0.01)
-- ++(14:2.0);
%%% from Q3 to Qnew
\draw[line width=0.75pt, shorten <=5pt, -{Stealth[length=2.5mm,width=1.5mm]},dotted]
({360/\n*3 + 360/(2*\n)}:\R/1.15)
-- ++(59:1.98);
%%% from Q4 to Qnew
\draw[line width=0.75pt, shorten <=5pt, -{Stealth[length=2.5mm,width=1.5mm]},dotted]
({360/\n*4 + 360/(2*\n)}:\R/1.15+0.01)
-- ++(112:1.239);
%%% From Q5 to Qnew
\draw[line width=0.75pt, shorten <=5pt, -{Stealth[length=2.5mm,width=1.5mm]},dotted]
({360/\n*5 + 360/(2*\n)}:\R/1.15)
-- ++(-140:0.4);

% center
\filldraw[line width=0.6pt] (0,0) circle (0.04);
\node[label={[red, scale=1.2] \normalsize ${E}$}]  at (-0.1,-0.1) {};

\end{tikzpicture}}

        \end{center} 
    \caption{Refinement stencil around an extraordinary vertex illustrating how the new point is computed from neighboring DOFs using subdivision weights. 
    }
    \label{fig:subdivision-relation}
    \end{figure}

In the following, we show the first row of $S^{0}(\nu,\mu)$. For the simple averaging we have:
\[
\begin{array}{rl}
\nu=3: & \left( \frac{1}{2}+ \frac{\mu}{12},
\frac{1}{4}-\frac{\mu}{24},
\frac{1}{4}-\frac{\mu}{24} \right) \\[5pt]
\nu=5: & \left( \frac{2}{5}+\frac{3\mu}{20},
\frac{3}{20}+\frac{\mu}{40},
\frac{3}{20}-\frac{\mu}{10},
\frac{3}{20}-\frac{\mu}{10},
\frac{3}{20}+\frac{\mu}{40} \right) \\[5pt]
\nu=6: &\left( \frac{3}{8} + \frac{\mu}{6}, 
\frac{1}{8}+\frac{\mu}{24},
\frac{1}{8}-\frac{\mu}{24},
\frac{1}{8}-\frac{\mu}{24},
\frac{1}{8}-\frac{\mu}{24},
\frac{1}{8}+\frac{\mu}{24} \right).
\end{array}
\]
For the coplanar averaging we have:
\[
\begin{array}{rl}
\nu=3: & \left( \tfrac{1}{2}+\frac{\mu}{2}, \tfrac{1}{4}-\frac{\mu}{4}, \tfrac{1}{4}-\frac{\mu}{4} \right) \\[5pt]
\nu=5: & \left( \frac{2}{5} + \frac{\mu}{2},
\frac{3}{20}+\frac{(-1+\sqrt{5})\,\mu}{8},
\frac{3}{20} - \frac{(1+\sqrt{5})\,\mu}{8},
\frac{3}{20} - \frac{(1+\sqrt{5})\,\mu}{8},
\frac{3}{20} + \frac{(-1+\sqrt{5})\,\mu}{8} \right) \\[5pt]
\nu=6: &\left( \frac{3+4\mu}{8},
\frac{1+2\mu}{8},
\frac{1-2\mu}{8},
\frac{1-4\mu}{8},
\frac{1-2\mu}{8},
\frac{1+2\mu}{8} \right).
\end{array}
\]

Linearity of the subdivision operator implies that the subdominant eigenvalue $\lambda$ can be expressed as a linear function of $\mu$. For any $\ell$ and a suitable choice of $\mu$, we have that the spectrum of $S^\ell(\nu,\mu)$ is $(1,\lambda(\nu,\mu),\lambda(\nu,\mu),\lambda_4,\ldots)$, with $1>\lambda(\nu,\mu)>\lambda_4$, placing the subdominant eigenvalue strictly between the dominant translational mode and the lower-order radial decay. 

In the simple averaging case, we have 
\[
  \lambda(3,\mu)=\frac{ \mu }{8} + \frac{1}{4} , \quad \lambda(5,\mu)=\frac{(3+\sqrt{5}) \mu}{16}  + \frac{1}{4}, \quad \text{and}\quad \lambda(6,\mu)=\frac{3 \mu}{8} + \frac{1}{4}.
\]
Therefore, for $\nu=3,5, \,\text{and} \,6$ we need 
\[
 0 < \mu < 6,\quad
 0 < \mu < \frac{12}{3+\sqrt{5}}, \quad \text{and}\quad
 0 < \mu < 2,
\]
providing a direct restriction on the admissible range of $\mu$. Then, the matrix $S^\ell(3,\mu)$ has the spectrum
\[
 \left(1,\frac{ \mu }{8} + \frac{1}{4},\frac{ \mu }{8} + \frac{1}{4},\frac{1}{4},\frac{1}{4},\frac{1}{4},\ldots\right)
\]
and the matrix $S^\ell(5,\mu)$ has the spectrum
\[
 \left(1,\frac{ \mu(3+\sqrt{5})+4}{16},\frac{ \mu(3+\sqrt{5})+4}{16},\frac{\mu(3-\sqrt{5})+4}{16},\frac{\mu(3-\sqrt{5})+4}{16},\frac{1}{4},\frac{1}{4},\frac{1}{4},\ldots\right).
\]
Similarly, for higher valences, there are more than three eigenvalues that are greater than $\frac{1}{4}$. For any $\nu$, in the coplanar averaging case we have
\[
\lambda(\nu,\mu) =  \frac{\nu \mu +1}{4},
\]
where $\mu$ must satisfy 
\[
 0 < \mu < \frac{3}{\nu},
\]
to obtain \(\frac{1}{4} < \lambda(\nu,\mu) < 1\). Then, the matrix $S^\ell(\nu,\mu)$ has the spectrum
\[
 \left(1,\frac{\nu \mu +1}{4},\frac{\nu \mu +1}{4},\frac{1}{4},\frac{1}{4},\frac{1}{4},\ldots\right).
\]

\section{A summary of the spline construction}\label{sec:summary}

Following the subdivision interpretation and using the linearity of the eigenvalue computation, we can describe the averaging weights from~\eqref{eq:ev-averaging} in terms of the subdominant eigenvalues $\lambda$.
\begin{lemma}\label{lem:ev-weights-lambda}
For simple averaging, we have the following averaging weights $(a_0 , a_1 , \ldots , a_{\nu-1})$, depending on $\lambda$, for different valences:
\[
\begin{array}{rl}
\nu=3: & \left(a_0,a_1, a_{2}\right) = \left( \tfrac{4\lambda}{3}, \tfrac{4\lambda}{3}, 1-\tfrac{8\lambda}{3} \right)\\[5pt]
\nu=5: & \left(a_0,\ldots, a_{4}\right) = \left( \frac{-3+\sqrt{5}+24\lambda}{5(3+\sqrt{5})},\frac{-3+\sqrt{5}+24\lambda}{5(3+\sqrt{5})}, \frac{7+\sqrt{5}-16\lambda}{5(3+\sqrt{5})},\frac{7+\sqrt{5}-16\lambda}{5(3+\sqrt{5})},\frac{7+\sqrt{5}-16\lambda}{5(3+\sqrt{5})} \right) \\[5pt]
\nu=6: &\left(a_0,\ldots, a_{5}\right) = \left( \frac{16\lambda-1}{18},\frac{16\lambda-1}{18},\frac{5-8\lambda}{18},\frac{5-8\lambda}{18},\frac{5-8\lambda}{18},\frac{5-8\lambda}{18} \right).
\end{array}
\]
For coplanar averaging, the weights $(a_0 , a_1 , \ldots , a_{\nu-1})$ are:
\[
\begin{array}{rl}
\nu=3: & \left(a_0,a_1, a_{2}\right) = \left( \tfrac{4\lambda}{3}, \tfrac{4\lambda}{3}, 1-\tfrac{8\lambda}{3} \right)\\[5pt]
\nu=5: & \left(a_0,\ldots, a_{4}\right) = \left( \frac{4\lambda}{5},\frac{4\lambda}{5}, \frac{5-\sqrt{5}+4(-3+\sqrt{5})\lambda}{10},\frac{\sqrt{5}+(4-\sqrt{5})\lambda}{5},\frac{5-\sqrt{5}+4(-3+\sqrt{5})\lambda}{10} \right) \\[5pt]
\nu=6: &\left(a_0,\ldots, a_{5}\right) = \left( \frac{4\lambda}{6},\frac{4\lambda}{6},\frac{1}{6},\frac{2-4\lambda}{6},\frac{2-4\lambda}{6},\frac{1}{6} \right).
\end{array}
\]
\end{lemma}
\begin{proof}
For the simple averaging we have 
$$
\left(a_0,\ldots, a_{\nu-1}\right) = 
\left( \tfrac{1}{\nu}, \ldots, \tfrac{1}{\nu} \right)
+ \mu  \left( \tfrac{\nu-2}{2\nu}, \tfrac{\nu-2}{2\nu}, -\tfrac{1}{\nu}, \ldots, -\tfrac{1}{\nu} \right)
$$
and use the relations $\mu(\lambda)=8\lambda-2$ for $\nu=3$, $\mu(\lambda)=\frac{16\lambda-4}{3+\sqrt{5}}$ for $\nu=5$ and $\mu(\lambda)=\frac{8\lambda-2}{3}$ for $\nu=6$. This gives the result.

For coplanar averaging we have 
$$
a_j := 
\tfrac{1}{\nu}
+ \mu \frac{\cos\left(\frac{2\pi j}{\nu}-\frac{\pi}{\nu}\right)}{\cos\left(\frac{\pi}{\nu}\right)} , \quad \mbox{for all }j\in\{0,\ldots, \nu-1\}
$$
and use the relation $\lambda =  \frac{\nu \mu +1}{4}$. This completes the proof.
\end{proof}
The characteristic rings for various valences and degrees are shown in Figure~\ref{fig:char_ring_table}. Note that the characteristic rings for simple and coplanar averaging are the same, as long as the subdominant eigenvalues $\lambda$ are the same. Characteristic rings can be computed for all feasible choices of the parameter $\mu$, for any valence $\nu$, degree $p$ and patch interior regularity $r$. We have opted to present a small selection, to keep the presentation short.

\begin{figure}[ht]
\centering
\setlength{\tabcolsep}{6pt}

\begin{tabular}{c c c c c}
\raisebox{9ex}{$p = 2$}&
\includegraphics[width=0.20\textwidth]{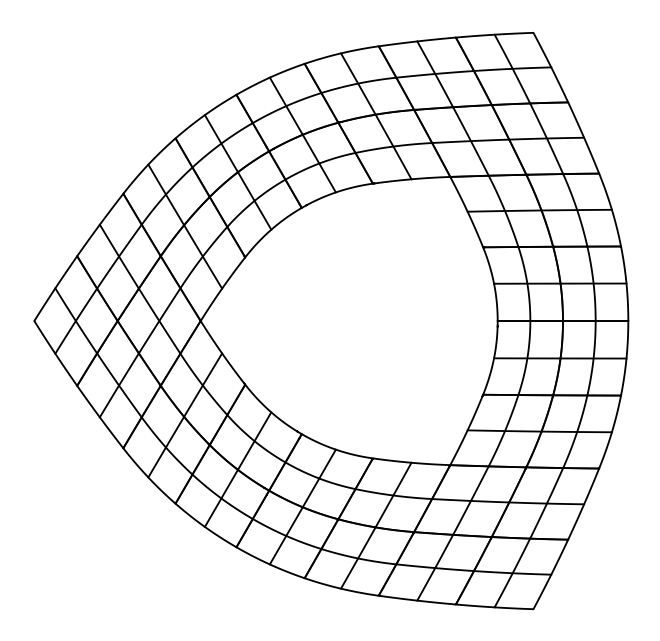} &
\includegraphics[width=0.20\textwidth]{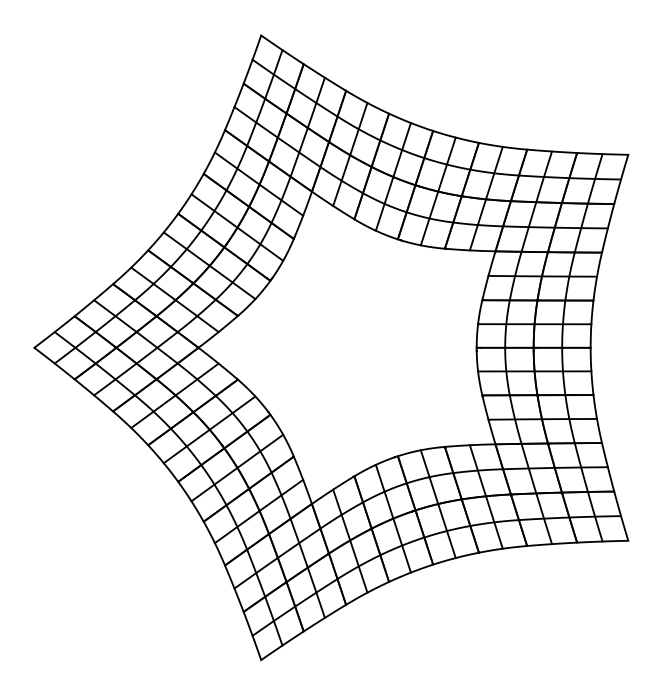} &
\includegraphics[width=0.20\textwidth]{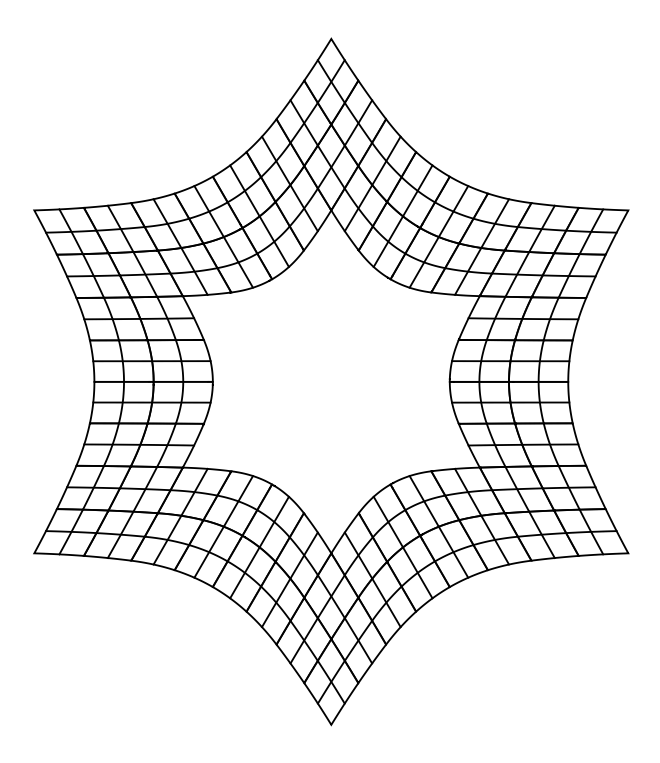} &
 \includegraphics[width=0.20\textwidth]{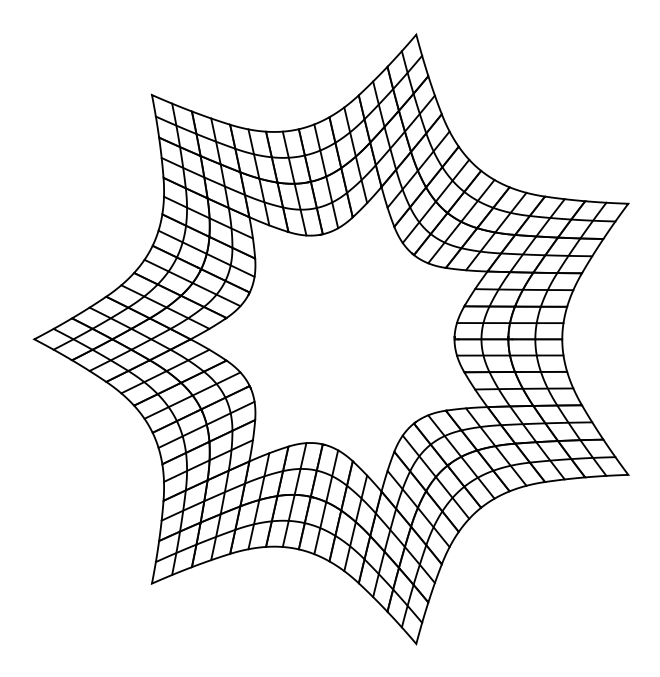}
\\
\raisebox{9ex}{$p = 3$}&
\includegraphics[width=0.20\textwidth]{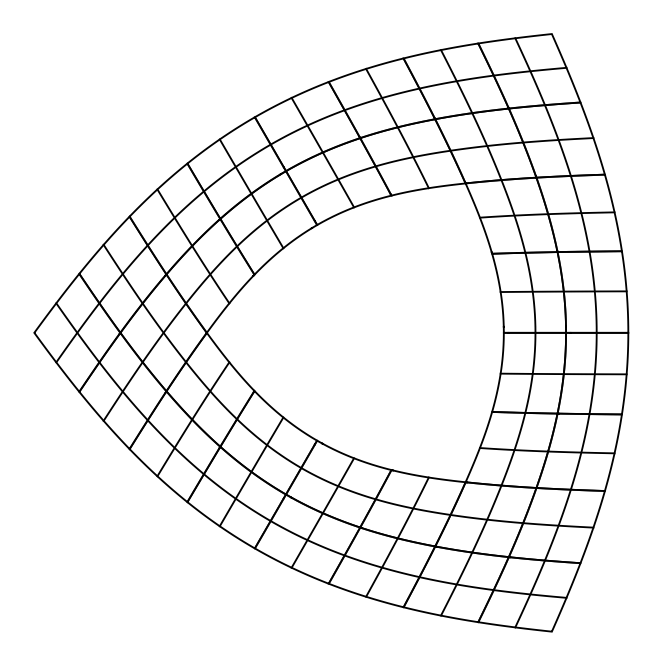} &
\includegraphics[width=0.20\textwidth]{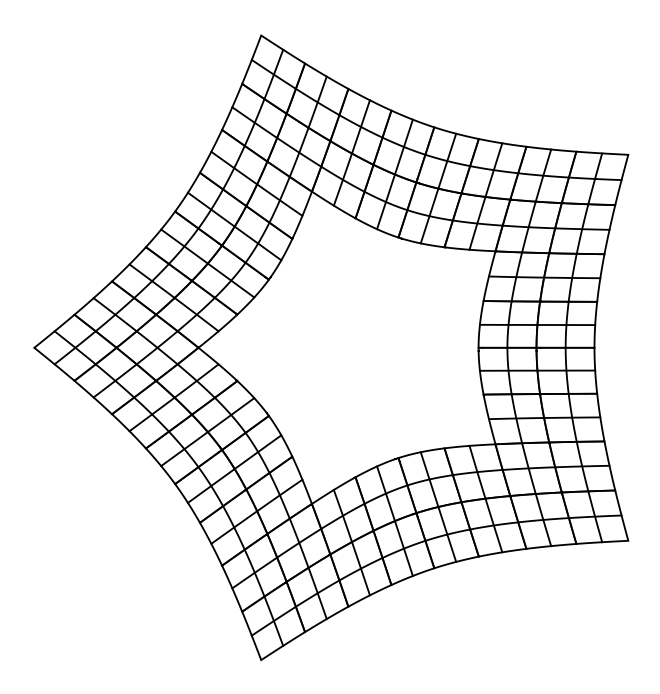} &
\includegraphics[width=0.20\textwidth]{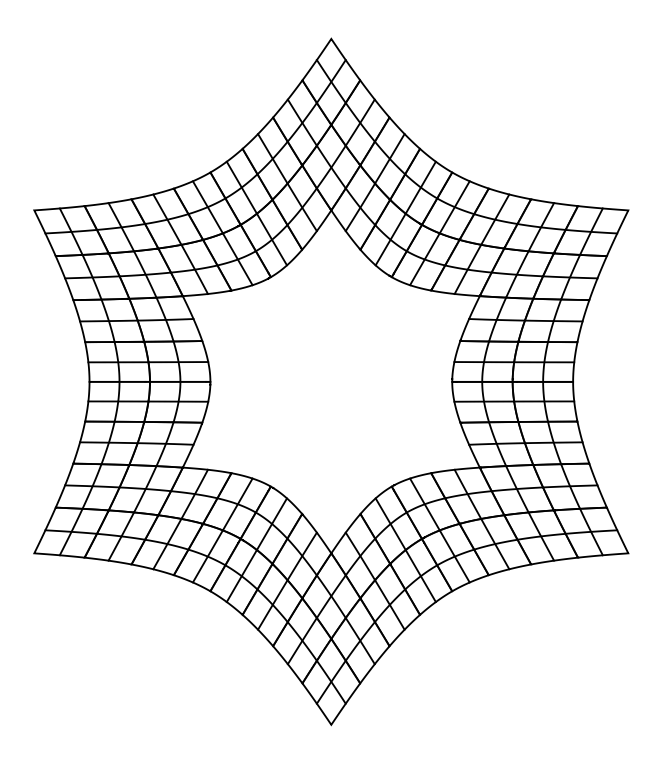} &
\includegraphics[width=0.20\textwidth]{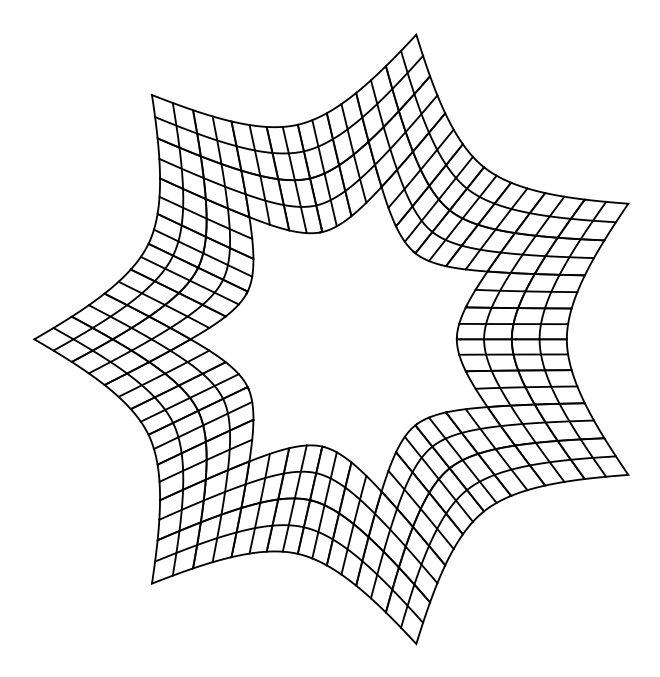}
\\
\raisebox{9ex}{$p = 4$}&
\includegraphics[width=0.20\textwidth]{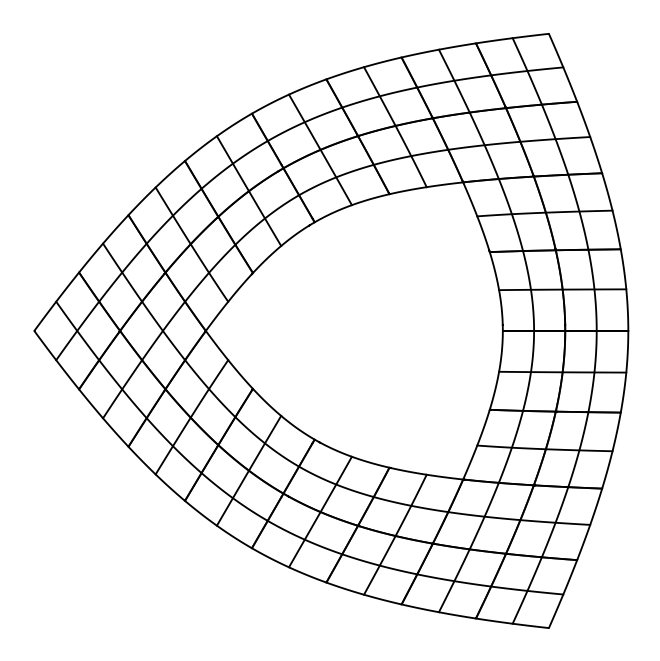} &
\includegraphics[width=0.20\textwidth]{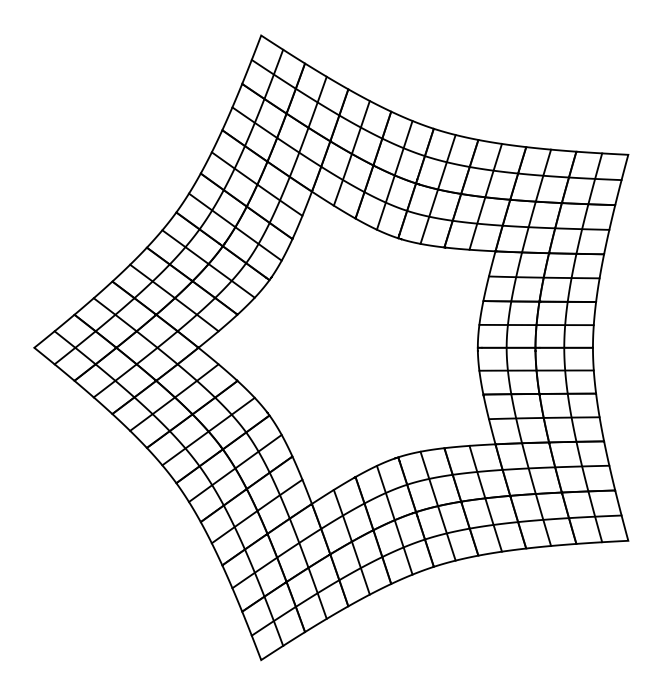} &
\includegraphics[width=0.20\textwidth]{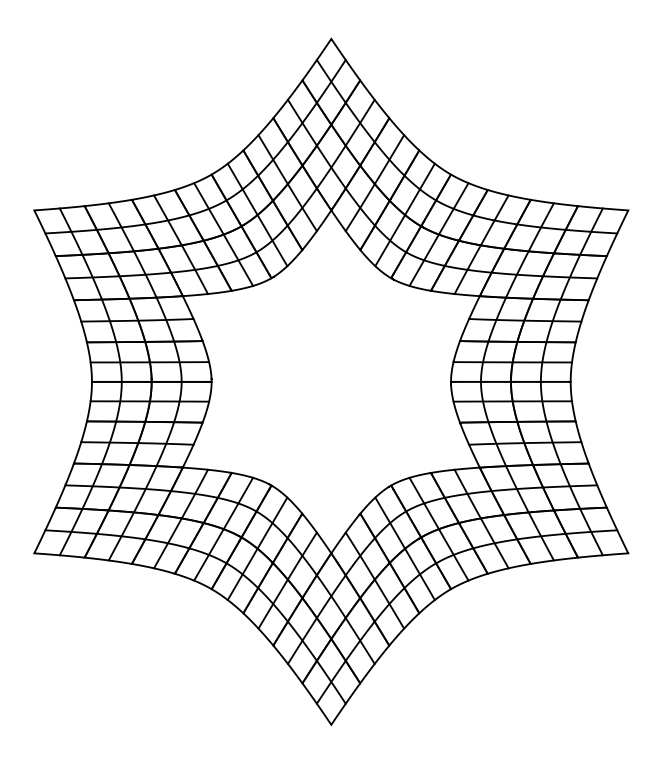} &
\includegraphics[width=0.20\textwidth]{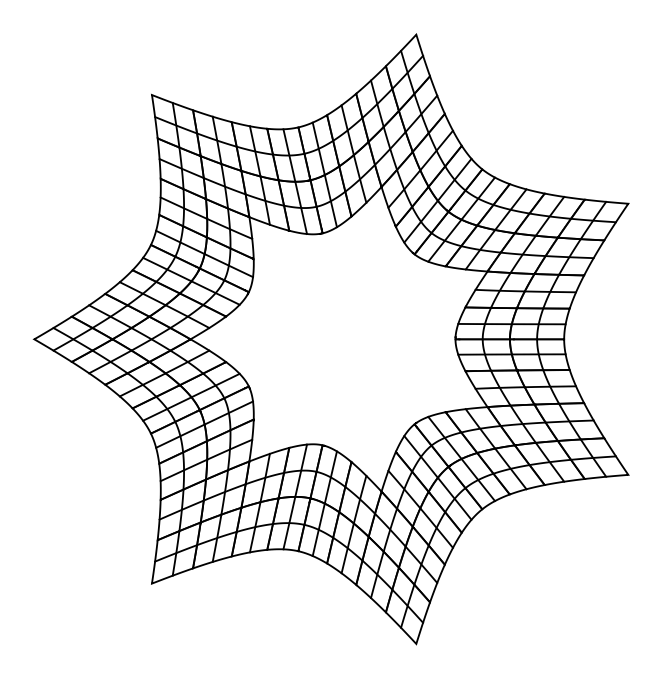}
\end{tabular}

\caption{Comparison of characteristic rings for degrees $p=2$, $p=3$ and $p=4$ with different valences \(\nu = 3,5,6,7\). In all cases $r=1$ and the parameter $\mu$ is selected such that $\lambda=1/2$.}
\label{fig:char_ring_table}
\end{figure}

We can now summarize the properties of the proposed spline construction.
\begin{lemma}
    Given an initial mesh and initial degrees of freedom $\PDOF$, we construct a spline surface following the averaging described in Section~\ref{sec:averaging}, where for every extraordinary vertex $E$ the weights are computed as in Lemma~\ref{lem:ev-weights-lambda}, for some $\lambda_E$. Then, the spline construction satisfies the following properties.
    \begin{enumerate}[label=(\alph*)]
    
        \item \textbf{Feasibility and limit behavior:} the construction is feasible i.e., it creates a regular limit surface if the input is regular and if $\frac{1}{4} < \lambda_E < 1$ for all $\lambda_E$. Then, the limit surface is $C^1$.
        
        \item \textbf{Element scaling:} element sizes scale as $h = O(2^{-\ell})$ away from extraordinary vertices and $h= O(\lambda_E^{\ell})$ near an extraordinary vertex $E$.
        
        \item \textbf{Local support:} each DOF generates a basis function with compact support. All basis functions are supported on level dependent subsets of patches of size $O(\max\{2^{-\ell},\lambda^\ell\})$, where $\lambda = \max_E(\lambda_E)$. There are three types of basis functions, those that are supported on a single patch, those that cross interfaces and are supported on two patches, and those near vertices that are supported on $\nu$ patches, where $\nu$ is the valence of the corresponding vertex.
        
        \item \textbf{Partition of unity:} the basis functions sum to 1 at every point of the limit surface for any feasible choice of $\lambda_E$. 
        
        \item \textbf{Non--negativity:} for a feasible choice of $\lambda_E$ the basis functions are pointwise non--negative if and only if $\lambda_E \leq \lambda_{\max}$, where $\lambda_{\max}$ depends on the valence and on the averaging scheme. In case of simple averaging $\lambda_{\max}=\frac{3}{8}$ if the valence $\nu=3$, $\lambda_{\max}=\frac{7+\sqrt{5}}{16}$ if $\nu=5$, and $\lambda_{\max}=\frac{5}{8}$ if $\nu=6$. In case of coplanar averaging $\lambda_{\max}=\frac{1}{2}$ if $\nu$ is even and $\lambda_{\max}=\frac{\cos(\pi/\nu)+1}{4}$ if $\nu$ is odd.
    \end{enumerate}
\end{lemma}
\begin{proof}
   \begin{enumerate}[label=(\alph*)]
    \item For all $\lambda \in (\frac14,1)$, the resulting characteristic rings are regular and the subdivision schemes create $C^1$-smooth limit surfaces, cf.~\cite[Theorem 5.8]{peters2008subdivision}. We tested the regularity of all characteristic rings which are used in the presented examples numerically. Further tests were also performed for higher valences and higher degrees. A proof of regularity for all $\lambda \in (\frac14,1)$ and arbitrary valences and degrees is beyond the scope of this analysis.
    
    \item Away from extraordinary vertices, we have uniform tensor product B-spline refinement. Hence the element size $h$ behaves as $h\sim 2^{-\ell}$. The elements adjacent to an extraordinary vertex scale as $h\sim \lambda_E^{\ell}$, which follows directly from the spectral properties of the subdivision operator.
    
    \item The statement follows directly from the fact that all spline coefficients are computed from averaged DOFs. Patch interior DOFs (away from the one ring neighborhood of interfaces) are mapped one-to-one to patch-wise coefficients, thus the corresponding basis functions are standard B-splines and have local support. The averaging for DOFs that are near interfaces, but away from vertices, results in two non-zero coefficient on one patch and one non-zero coefficient on the neighboring patch, while the averaging for vertex-adjacent DOFs results in up to $4$ non-zero coefficients on any of the $\nu$ patches around the vertex $E$ of valence $\nu$. The support of each basis function is thus bounded by the union of local supports, each of size $O(\max\{2^{-\ell},\lambda^\ell\})$.
     
     \item The averaged B\'ezier coefficients are affine combinations of the DOFs because both simple and coplanar averaging at EVs satisfy $\sum_j a_j = 1$ and the edge averages are always with weights $(\frac12,\frac12)$. Since the patch-wise basis functions form a partition of unity, also the spline basis functions computed by averaging form a partition of unity. 
     
     \item A non-negative basis is obtained if and only if all averaging coefficients are non-negative. The bounds on $\lambda_E$ are directly derived from the description of averaging weights in Lemma~\ref{lem:ev-weights-lambda}. This completes the proof.
   \end{enumerate} 
\end{proof}

\section{Examples}\label{sec:examples}

In this section we perform several numerical tests. In Section~\ref{sec:surface-by-subdivision} we compute limit surfaces of the resulting subdivision scheme for several configurations of extraordinary vertex neighborhoods. In Sections~\ref{sec:interpolation} and~\ref{sec:L2-approximation} we solve an interpolation and $L^2$-approximation problem, respectively, on a square domain. We conclude the tests with a sphere approximation in Section~\ref{sec:sphere-fitting}.

\subsection{Surfaces obtained by subdivision}\label{sec:surface-by-subdivision}

To illustrate the behavior of the proposed construction near extraordinary vertices, we perform a sequence of numerical tests for neighborhoods of extraordinary vertices with valences $\nu = 3,\,5,\,\text{and }6$. For each configuration, we construct an initial control mesh composed of \(\nu\) segments with DOFs on a tensor-product grid and apply repeated refinement and averaging steps. We examine constructions for polynomial degrees $p = 2,3, \text{ and }4$ and varying $\lambda$.

In all surface plots we include contour lines $x=c$, for varying $c \in \mathbb{R}$, to visualize the local non-smoothness and undesired shape deformations. If the surface is not $C^1$, the contour lines are also not $C^1$, i.e., they make an angle at the lines of reduced continuity, which is clearly visible in the plots.

\begin{figure}[h!]
\centering
\begin{subfigure}[b]{0.32\textwidth}
\includegraphics[width=0.99\textwidth]{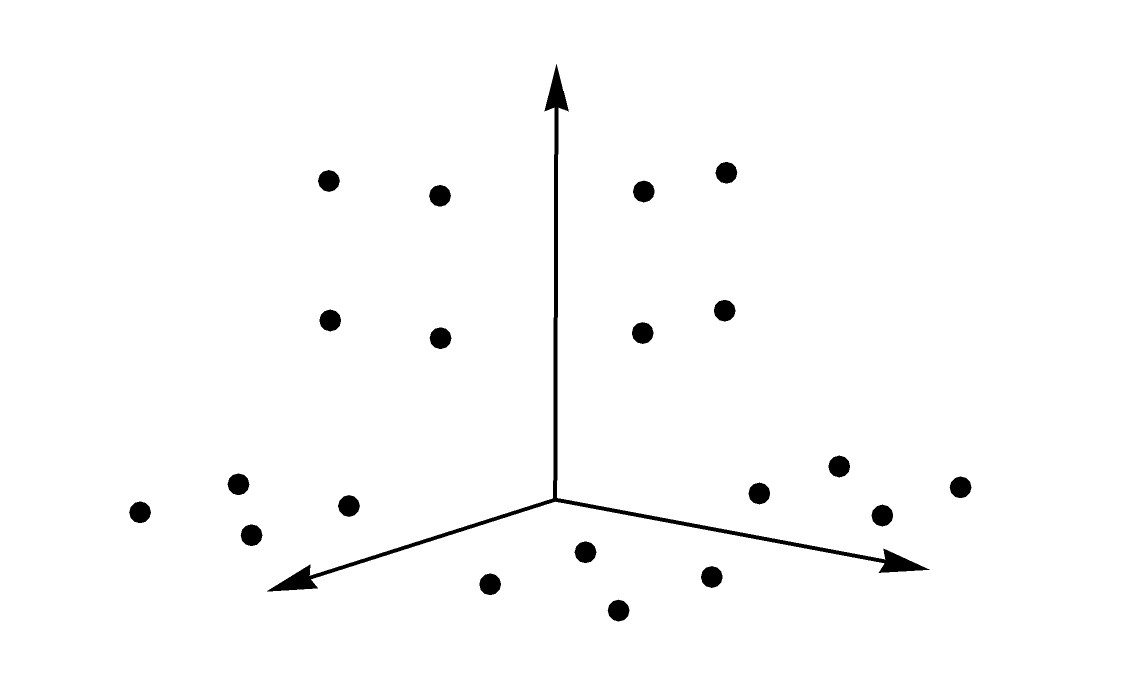} 
\vspace{30pt}
\caption{Initial DOFs}\label{fig:surface-p2-val5-different-levels-a}
\end{subfigure}
\begin{subfigure}[b]{0.32\textwidth}
\includegraphics[width=0.99\textwidth]{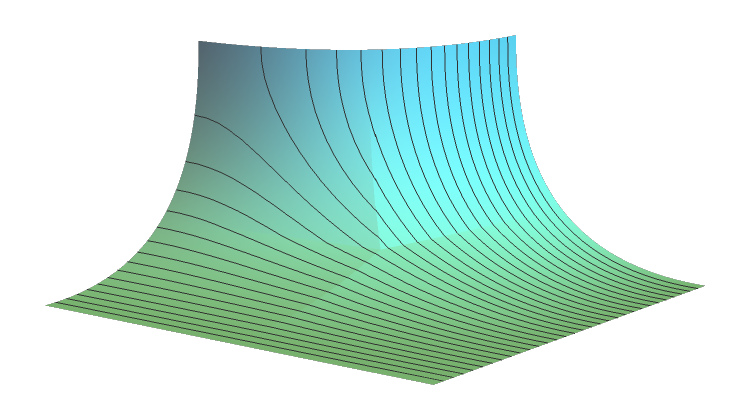} \\
\includegraphics[width=0.99\textwidth]{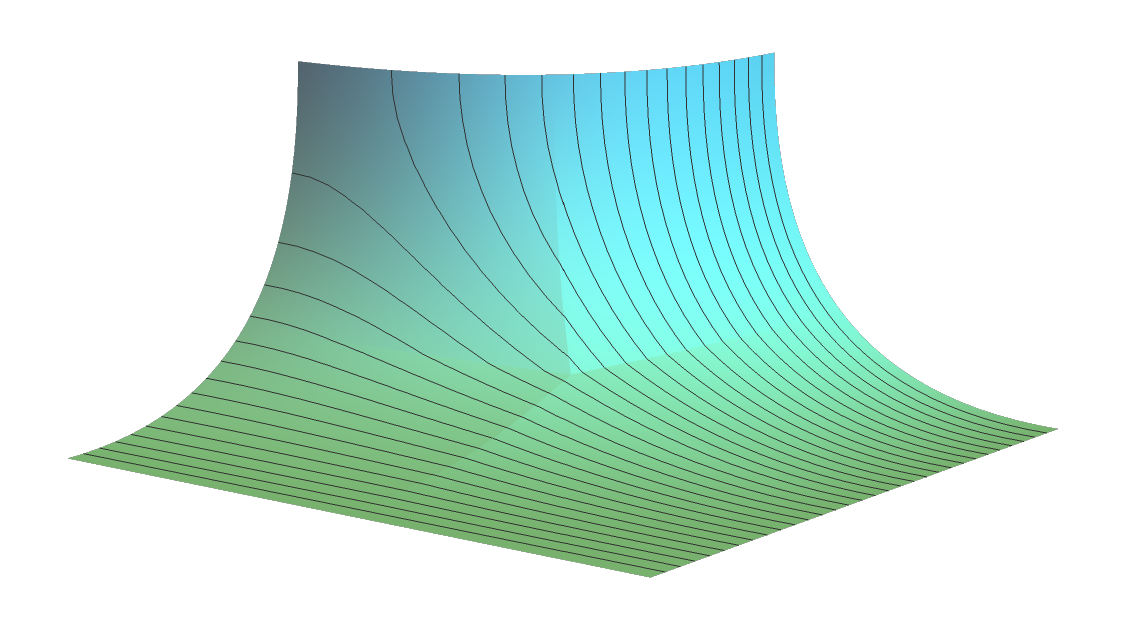} 
\caption{Surface for level $\ell=0$}\label{fig:surface-p2-val5-different-levels-b}
\end{subfigure}
\begin{subfigure}[b]{0.32\textwidth}
\includegraphics[width=0.99\textwidth]{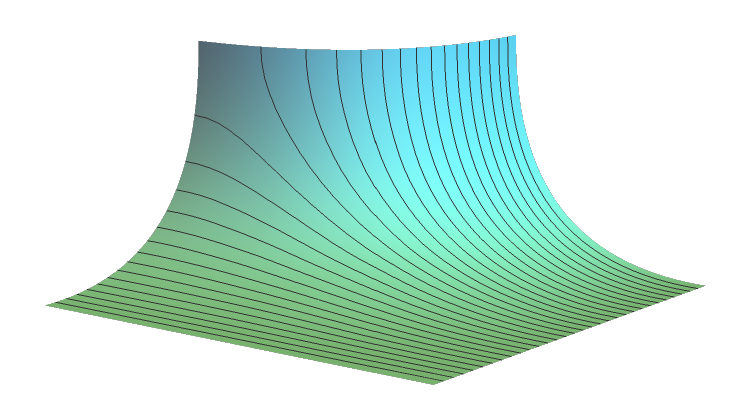} \\
\includegraphics[width=0.99\textwidth]{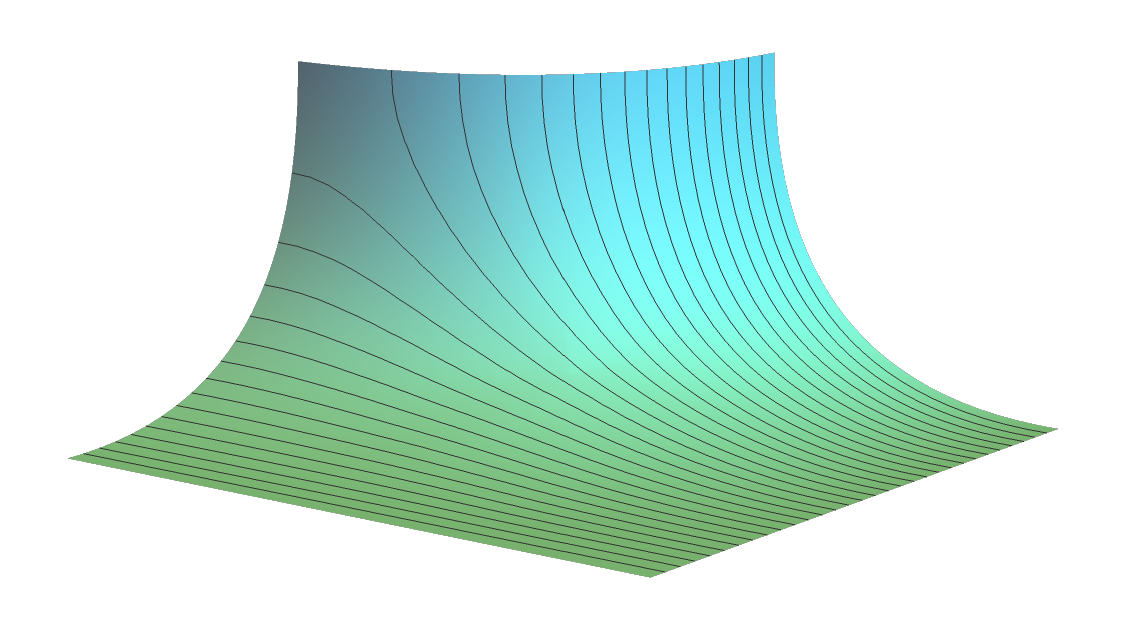} 
\caption{Surface for level $\ell=3$}\label{fig:surface-p2-val5-different-levels-c}
\end{subfigure}
\caption{Surface around an EV of valence $5$ before and after refinement, using $p=2$, $r=1$, $\lambda=0.5$ for coplanar averaging (top row) and simple averaging (bottom row).}\label{fig:surface-p2-val5-different-levels}
\end{figure}

In Figure~\ref{fig:surface-p2-val5-different-levels} we visualize the neighborhood of an extraordinary vertex of valence 5. The figure shows the initial control points. It also shows the surface at level zero and after three refinement steps. Both coplanar and simple averaging are used with $\lambda=0.5$. We add contour lines to the surface plots to visualize the local non-smoothness. If the surface is not $C^1$ then the contour lines also show a visible angle or kink. In case of coplanar averaging, the surface is always $C^1$ at the EV itself and non-smooth in a neighborhood of the EV at level zero, whereas, in case of simple averaging, the surface is non-smooth in the entire EV neighborhood at level zero. For both averaging families the contour lines appear smooth after refinement.

\begin{figure}[h!]
\centering
\begin{subfigure}[b]{0.3\textwidth}
\includegraphics[width=0.99\textwidth]{figures/5patch-p2-r1-la0-5-copl-L3.pdf} \\
\includegraphics[width=0.99\textwidth]{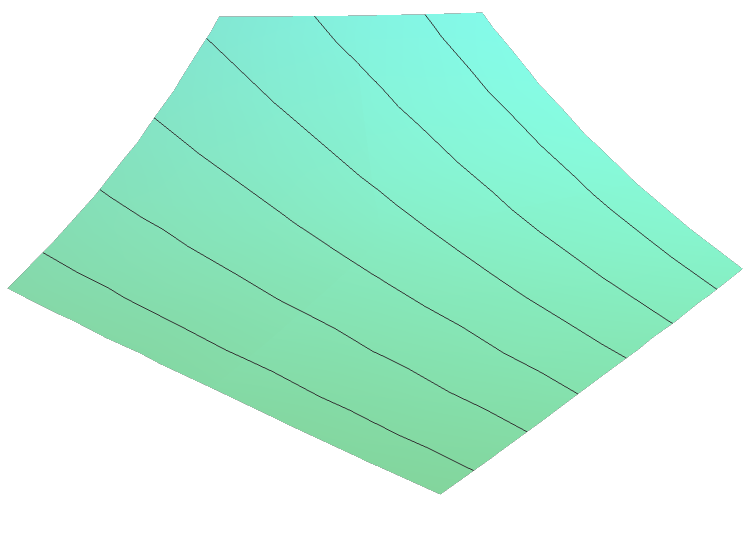} 
\caption{Surface with $\lambda=0.5$}\label{fig:surface-p2-val5-different-lambda-a}
\end{subfigure}
\begin{subfigure}[b]{0.3\textwidth}
\includegraphics[width=0.99\textwidth]{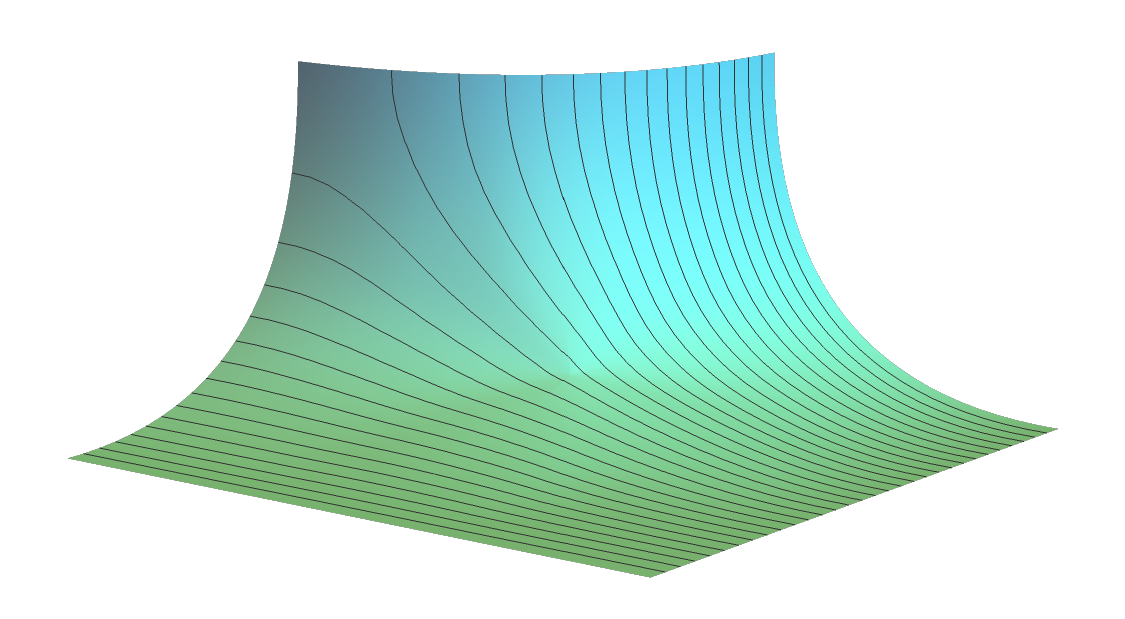} \\
\includegraphics[width=0.99\textwidth]{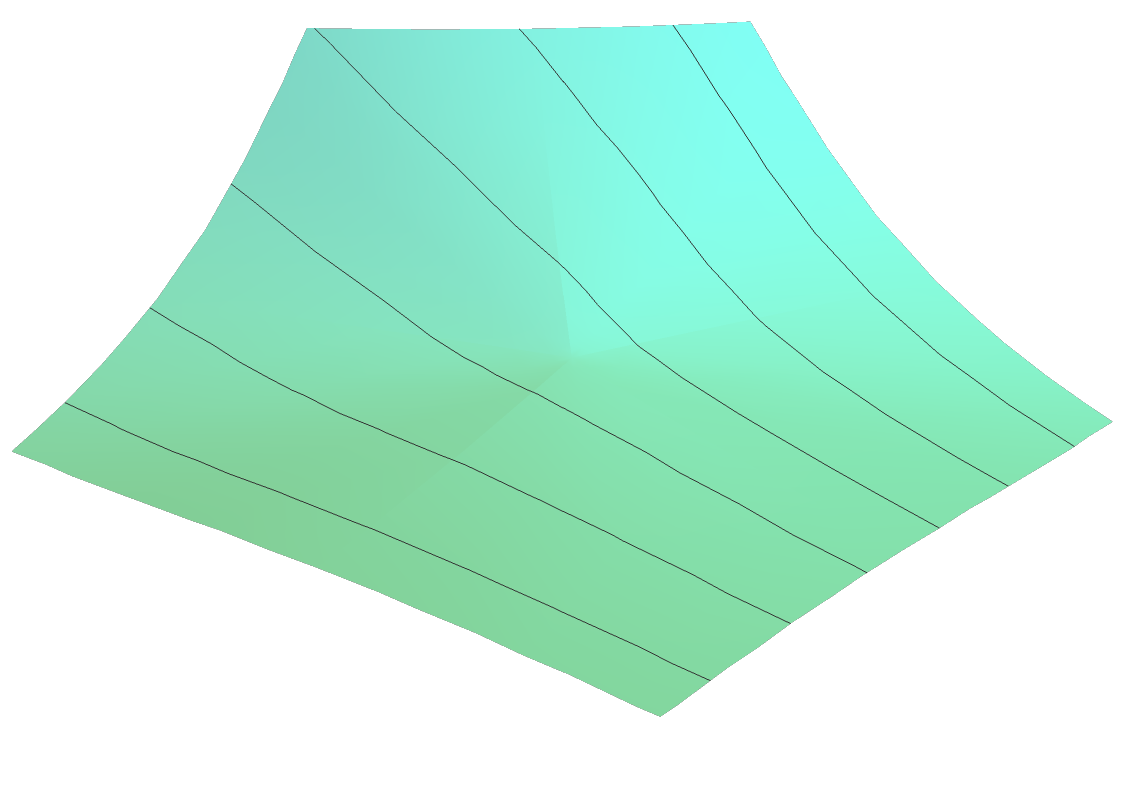} 
\caption{Surface with $\lambda=0.35$}\label{fig:surface-p2-val5-different-lambda-b}
\end{subfigure}
\begin{subfigure}[b]{0.3\textwidth}
\includegraphics[width=0.99\textwidth]{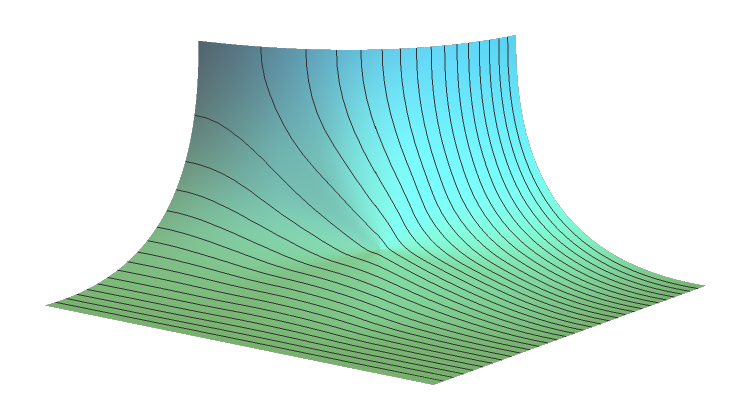} \\
\includegraphics[width=0.99\textwidth]{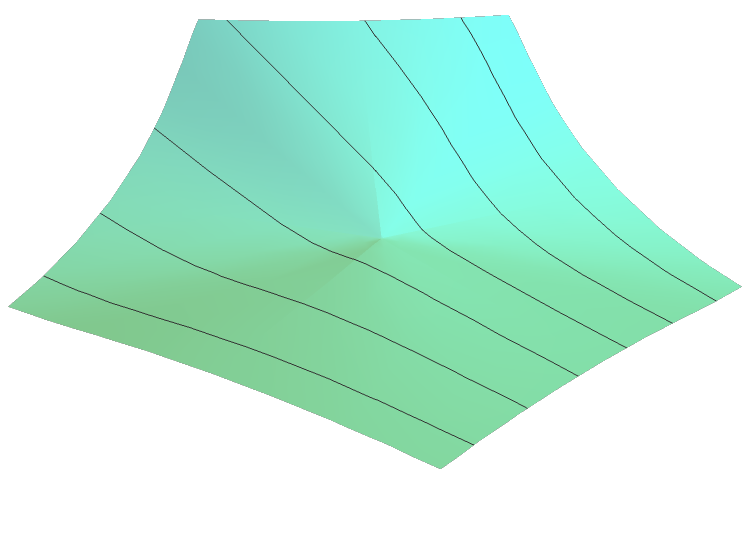} 
\caption{Surface with $\lambda=0.26$}\label{fig:surface-p2-val5-different-lambda-c}
\end{subfigure}
\caption{Surface around an EV of valence $5$ after refinement to level $\ell=3$, using $p=2$, $r=1$ for coplanar averaging with varying $\lambda$. Above is the full surface and below a close-up near the EV.}\label{fig:surface-p2-val5-different-lambda}
\end{figure}

Figure~\ref{fig:surface-p2-val5-different-lambda} illustrates the effect of the subdominant eigenvalue $\lambda$ on the surface around an extraordinary vertex of valence 5 after three refinement steps, using coplanar averaging with $p=2$. Three values are shown: $\lambda=0.5$, $0.35$, and $0.26$. As $\lambda$ decreases, the surface becomes more distorted and piecewise flat in the immediate vicinity of the extraordinary vertex, as seen in the close-up view. The contour lines remain smooth and free of kinks away from the EV, confirming the $C^1$ smoothness there.

\begin{figure}[h!]
\centering
\begin{subfigure}[b]{0.3\textwidth}
\includegraphics[width=0.99\textwidth]{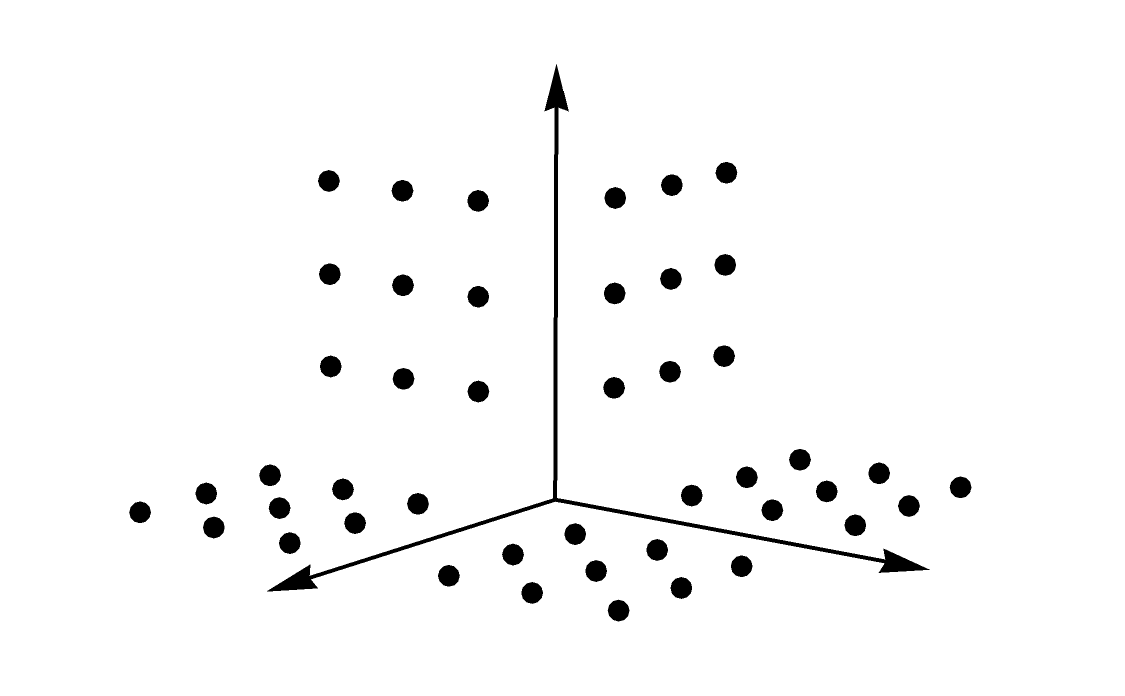} 
\includegraphics[width=0.99\textwidth]{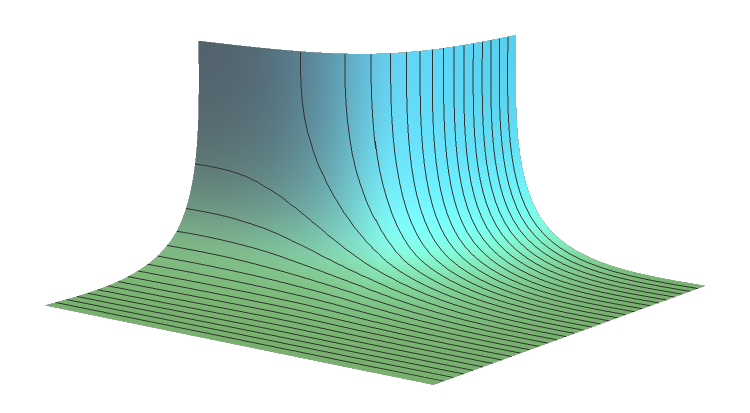} 
\includegraphics[width=0.99\textwidth,trim=0 0 0 60,clip]{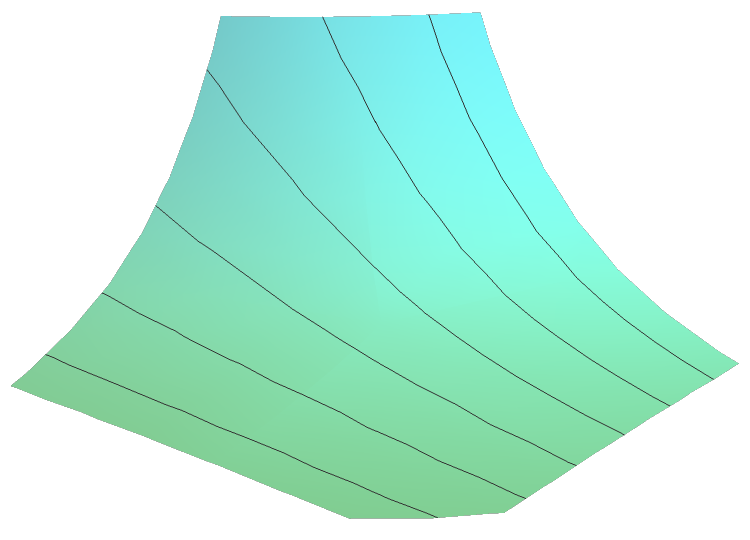} 
\caption{Surface with $p=3$, $r=2$, $\lambda=0.5$}\label{fig:surface-varying-degree-a}
\end{subfigure}
\begin{subfigure}[b]{0.3\textwidth}
\includegraphics[width=0.99\textwidth]{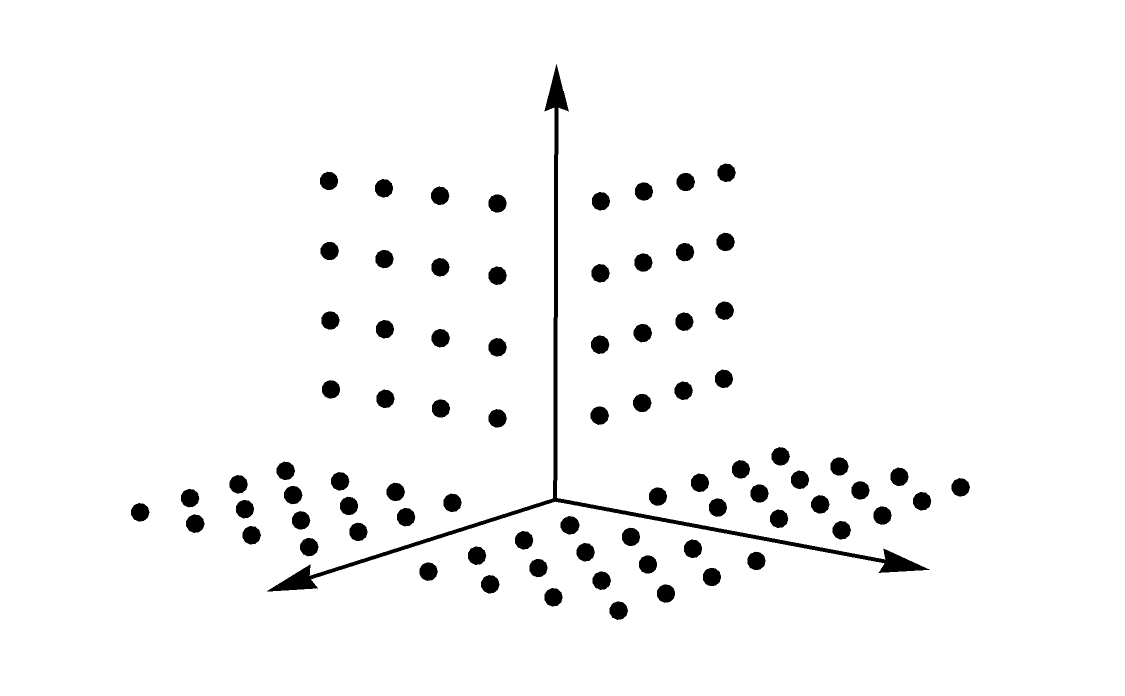} 
\includegraphics[width=0.99\textwidth]{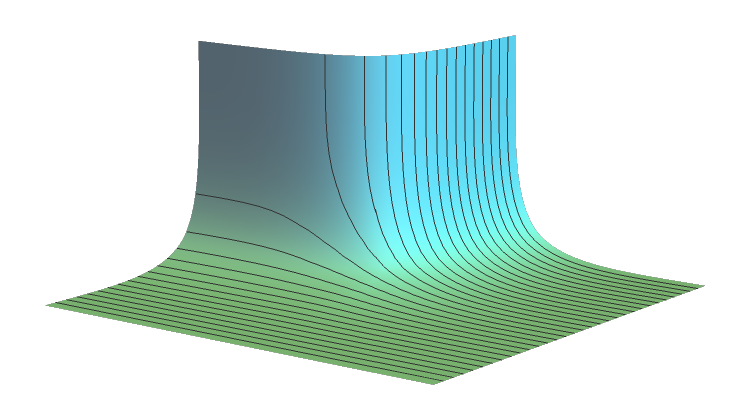} 
\includegraphics[width=0.99\textwidth,trim=0 0 0 60,clip]{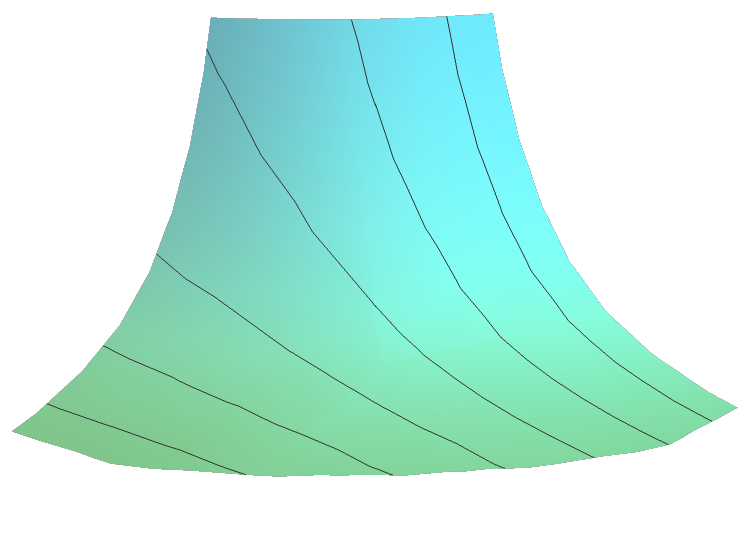} 
\caption{Surface with $p=4$, $r=3$, $\lambda=0.5$}\label{fig:surface-varying-degree-b}
\end{subfigure}
\begin{subfigure}[b]{0.3\textwidth}
\includegraphics[width=0.99\textwidth]{figures/5patch-p4-dofs-L0.pdf} 
\includegraphics[width=0.99\textwidth]{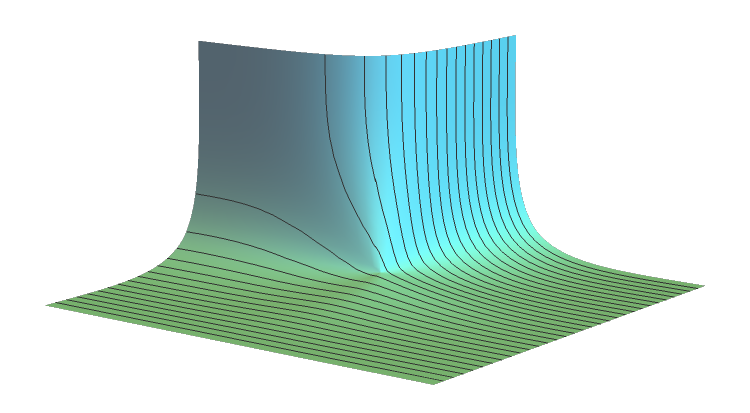} 
\includegraphics[width=0.99\textwidth,trim=0 0 0 60,clip]{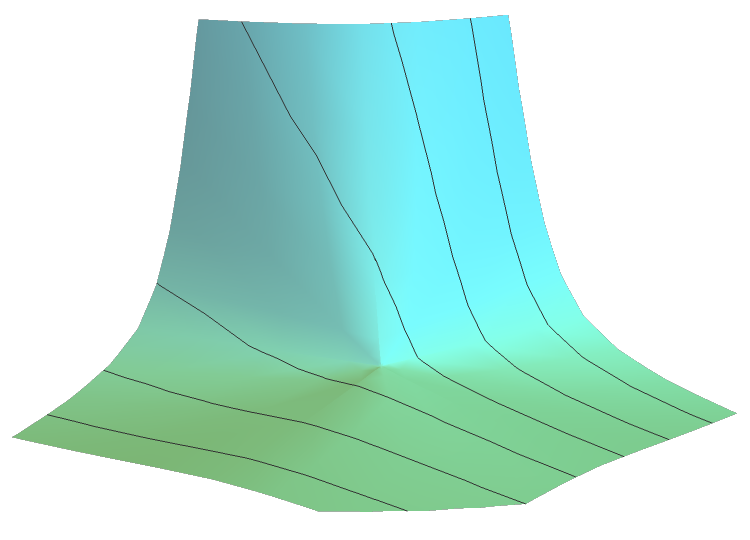} 
\caption{Surface with $p=4$, $r=3$, $\lambda=0.26$}\label{fig:surface-varying-degree-c}
\end{subfigure}
\caption{Surface around an EV of valence $5$ after refinement to level $\ell=3$, using coplanar averaging for varying degree and $\lambda$. Above are the DOFs, in the middle row the full surface and below a close-up near the EV.}\label{fig:surface-varying-degree}
\end{figure}

In Figure~\ref{fig:surface-varying-degree} we show surfaces around a valence 5 extraordinary vertex after three refinement steps. We use coplanar averaging with different polynomial degrees and $\lambda$ values. For $p=3$ with $\lambda=0.5$ and for $p=4$ with $\lambda=0.5$, the surfaces are smooth. For $p=4$ with $\lambda=0.26$, the surfaces become noticeably flatter near the extraordinary vertex. The local smoothness near an extraordinary vertex is visible in the close-up views. In all cases, the contour lines show no angular discontinuities.

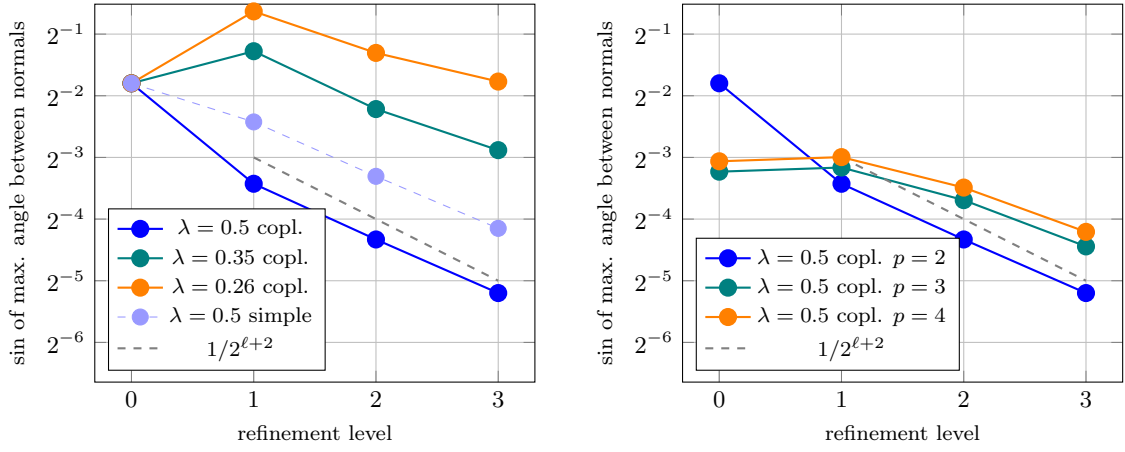
\begin{figure}[h!]
    \centering
\begin{tikzpicture}
\begin{axis}[
    width=.45\textwidth,
    height=.4\textwidth,
    xlabel={refinement level},
    ylabel={$\sin$ of max. angle between normals},
    xtick={0,1,2,3},
    ymode=log,
    ytick={1,0.5,0.25,0.125,0.0625,0.03125,0.015625},
    mark options={solid},
    log basis y=2,
    ymax=0.7, 
    ymin=0.01,
    grid=both,
    tick label style={font=\small},
    label style={font=\footnotesize},,
    legend style={font=\footnotesize},
    legend pos=south west,
    title={}
]
\addplot[
    color=blue,
    % p=2, r=1, coplanar, lambda=0.5
    mark=*,
    mark size=3pt,
    thick,
] coordinates {
    (0, 0.287621)
    (1, 0.0929655)
    (2, 0.0496945)
    (3, 0.0272084)

};
\addplot[
    color=teal,
     % p=2, r=1, coplanar, lambda=0.35355
    mark=*,
    mark size=3pt,
    thick,
    ] coordinates {
    
   (0, 0.287621)
    (1, 0.412492)
    (2, 0.215259)
    (3, 0.135643)
};
\addplot[
    color=orange,
     % p=2, r=1, coplanar, lambda=0.26
    mark=*,
    mark size=3pt,
    thick,
] coordinates {
    (0, 0.287621)
    (1, 0.645496)
    (2, 0.404664)
    (3, 0.293281)
};
\addplot[
    color=blue!40,
     % p=2, r=1, simple, lambda=0.5
    mark=*,
    mark size=3pt,
    dashed,
] coordinates {
   (0, 0.287621)
    (1, 0.186334)
    (2, 0.101069)
    (3, 0.0563137)
};
\addplot[
    color=gray,
    thick,
    dashed,
] coordinates {
    (1, 0.5/4)
    (2, 0.5/8)
    (3, 0.5/16)
};
\legend{$\lambda=0.5$ copl., $\lambda=0.35$ copl., $\lambda=0.26$ copl., $\lambda=0.5$ simple, $1/2^{\ell+2}$}
\end{axis}
\end{tikzpicture}\qquad
\begin{tikzpicture}
\begin{axis}[
    width=.45\textwidth,
    height=.4\textwidth,
    xlabel={refinement level},
    ylabel={$\sin$ of max. angle between normals},
    xtick={0,1,2,3},
    ymode=log,
    ytick={1,0.5,0.25,0.125,0.0625,0.03125,0.015625},
    mark options={solid},
    log basis y=2,
    ymax=0.7, 
    ymin=0.01,
    grid=both,
    tick label style={font=\small},
    label style={font=\footnotesize},
    legend style={font=\footnotesize},
    legend pos=south west,
    title={}
]
\addplot[
    color=blue,
    % p=2, r=1, coplanar, lambda=0.5
    mark=*,
    mark size=3pt,
    thick,
] coordinates {
    (0, 0.287621)
    (1, 0.0929655)
    (2, 0.0496945)
    (3, 0.0272084)

};
\addplot[
    color=teal,
     % p=3, r=2, coplanar, lambda=0.5
    mark=*,
    mark size=3pt,
    thick,
    ] coordinates {
    (0, 0.106499)
    (1, 0.111606)
    (2, 0.0773368)
    (3, 0.0459728)
};
\addplot[
    color=orange,
     % p=4, r=3, coplanar, lambda=0.5
    mark=*,
    mark size=3pt,
    thick,
] coordinates {
    (0, 0.119663)
    (1, 0.125438)
    (2, 0.0890616)
    (3, 0.054144)
};
\addplot[
    color=gray,
    thick,
    dashed,
] coordinates {
    (1, 0.5/4)
    (2, 0.5/8)
    (3, 0.5/16)
};
\legend{$\lambda=0.5$ copl. $p=2$, $\lambda=0.5$ copl. $p=3$, $\lambda=0.5$ copl. $p=4$, $1/2^{\ell+2}$}
\end{axis}
\end{tikzpicture}
    \caption{Maximum angle between normals on neighborhood of extraordinary vertex of valence $5$. On the left: varying $\lambda$ for $p=2$ and $r=1$. On the right: varying $p$ with $r=p-1$ and $\lambda=0.5$.}
    \label{fig:jump-of-normal}
\end{figure} 

Figure~\ref{fig:jump-of-normal} shows the maximum angle between surface normals on the neighborhood of a valence 5 extraordinary vertex against the refinement level. 
In the left figure we compare coplanar averaging with $\lambda=0.5$, $0.35$, and $0.26$ for degree $p=2$. The right figure compares different polynomial degrees with $\lambda=0.5$ and $r=p-1$. We observe that a smaller subdominant eigenvalue $\lambda$ leads to a significantly larger angle between normals. This means the surface normal changes more abruptly near the EV when $\lambda$ is small. For simple averaging with $\lambda=0.5$ the angle is larger than for coplanar averaging. Increasing the polynomial degree also raises the angle slightly. In all cases the observed convergence rate of the maximum angle is close~to~$1/2^\ell$.

\begin{figure}[h!]
\centering
\begin{subfigure}[b]{0.3\textwidth}
\includegraphics[width=0.9\textwidth]{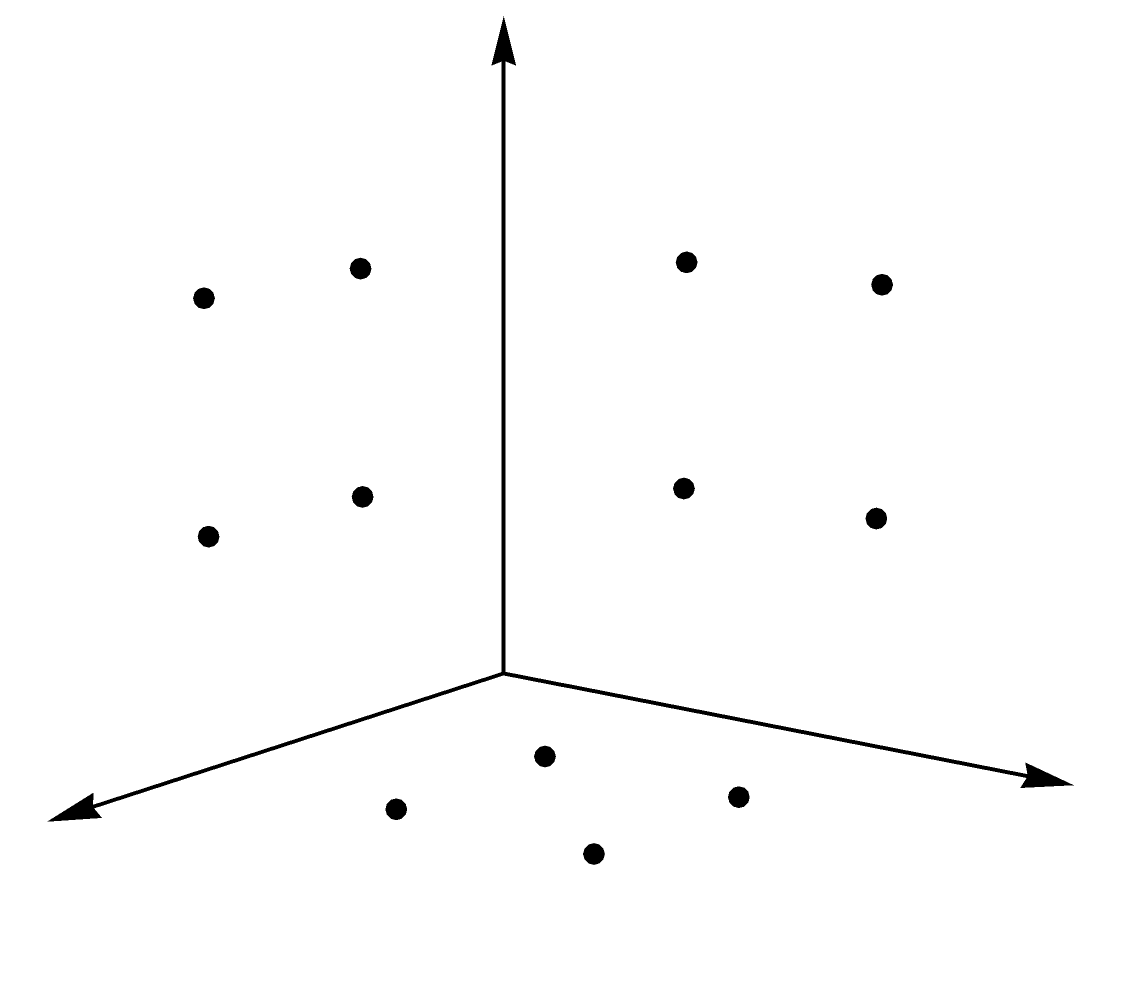} \\
\includegraphics[width=0.9\textwidth]{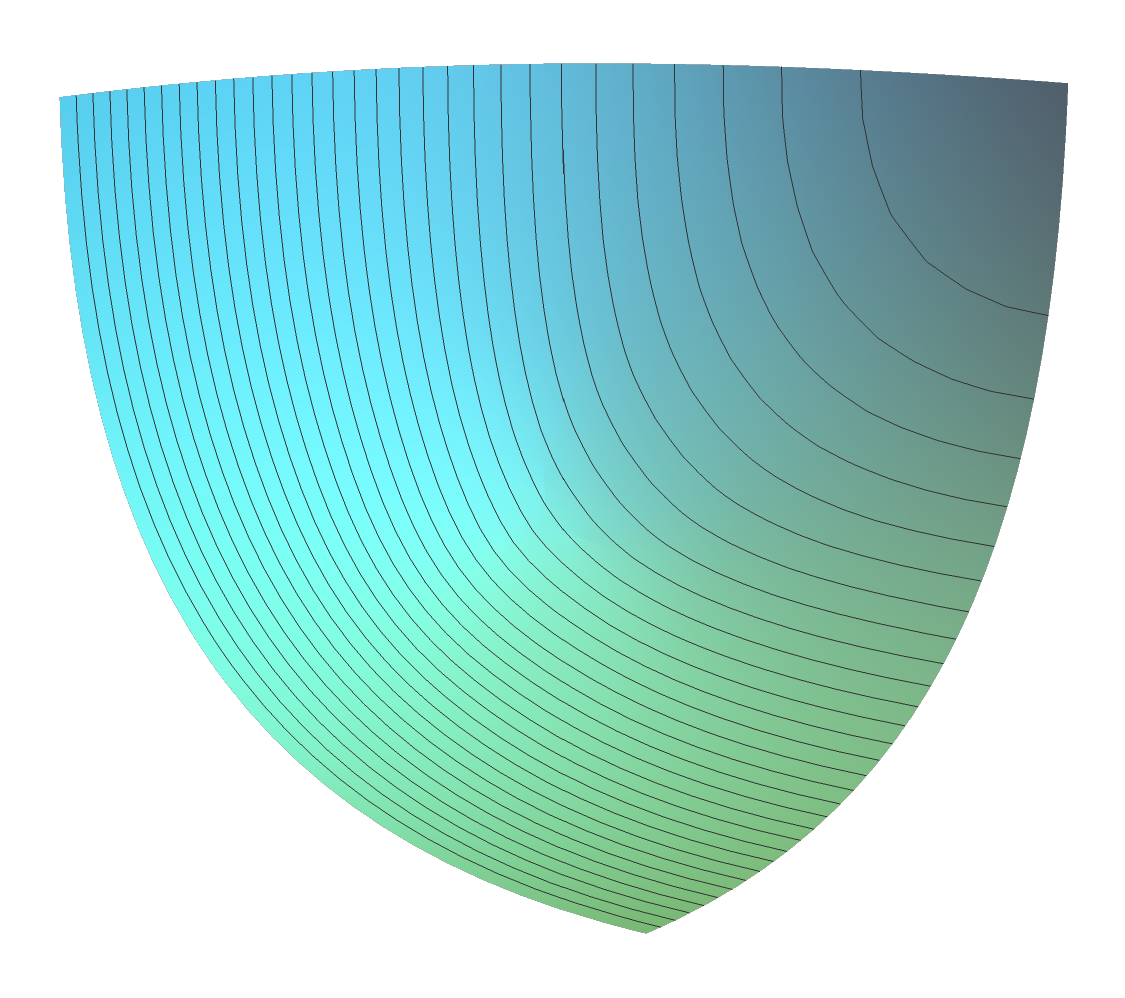} \\
\includegraphics[width=0.9\textwidth]{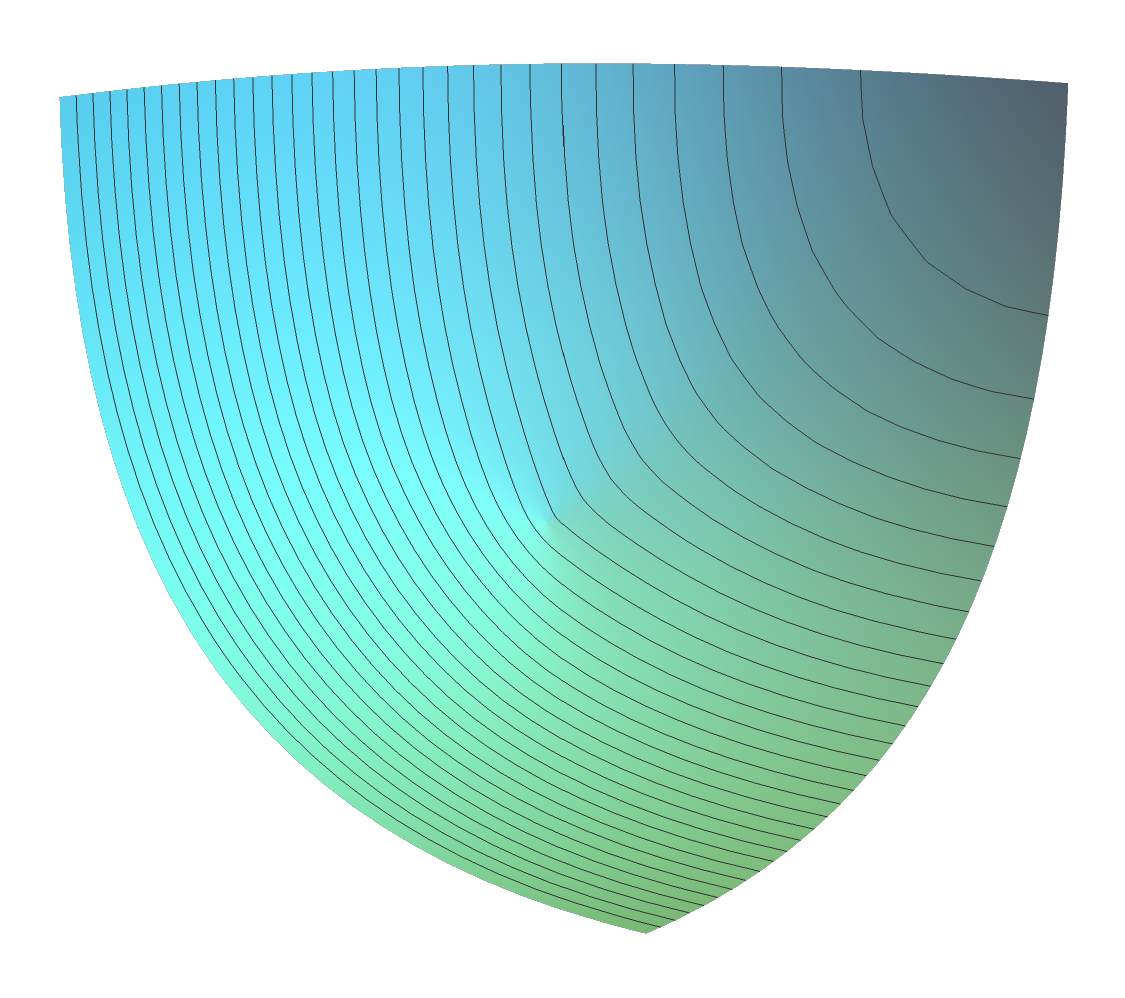} 
\caption{Surfaces with $p=2$, $r=1$}\label{fig:surface-valence-3-a}
\end{subfigure}
\begin{subfigure}[b]{0.3\textwidth}
\includegraphics[width=0.9\textwidth]{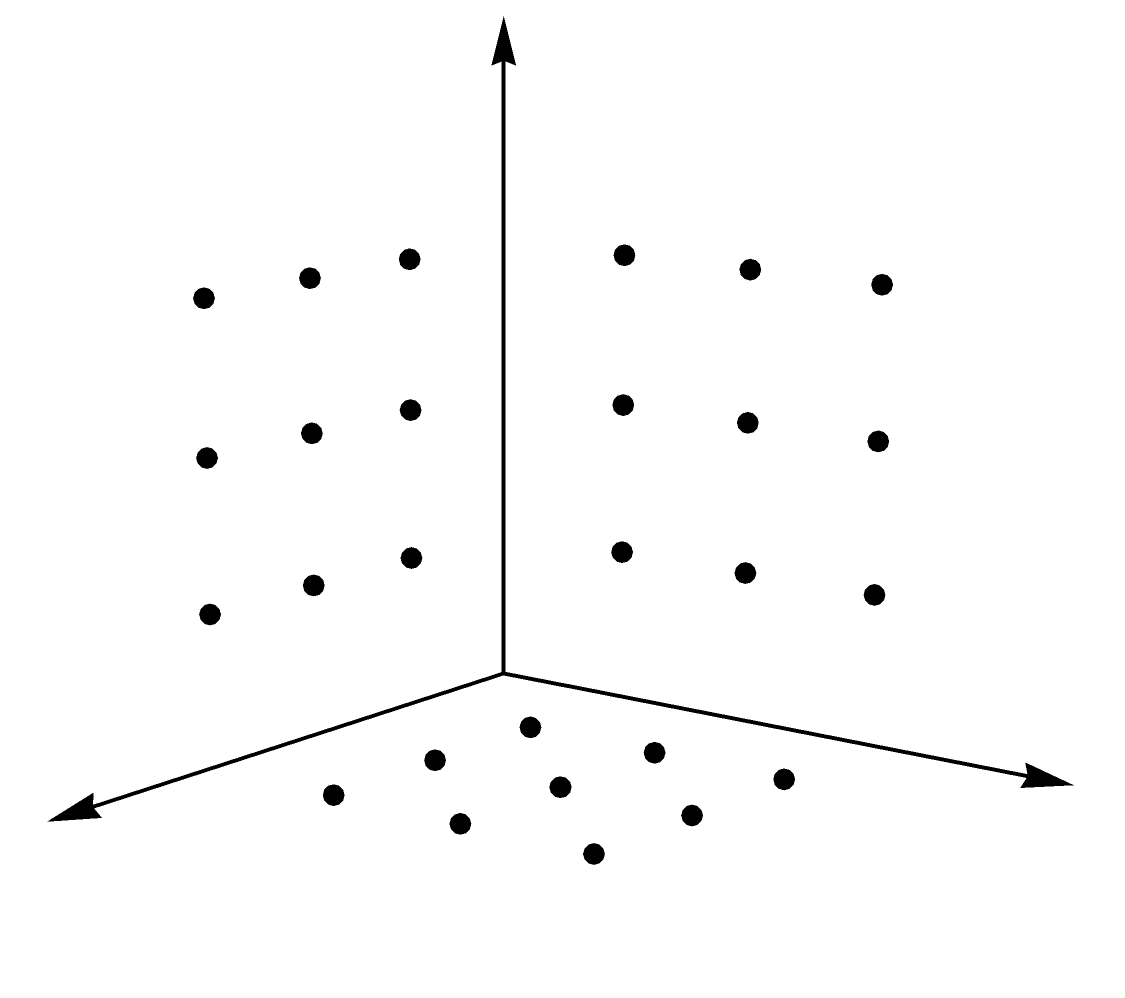} \\
\includegraphics[width=0.9\textwidth]{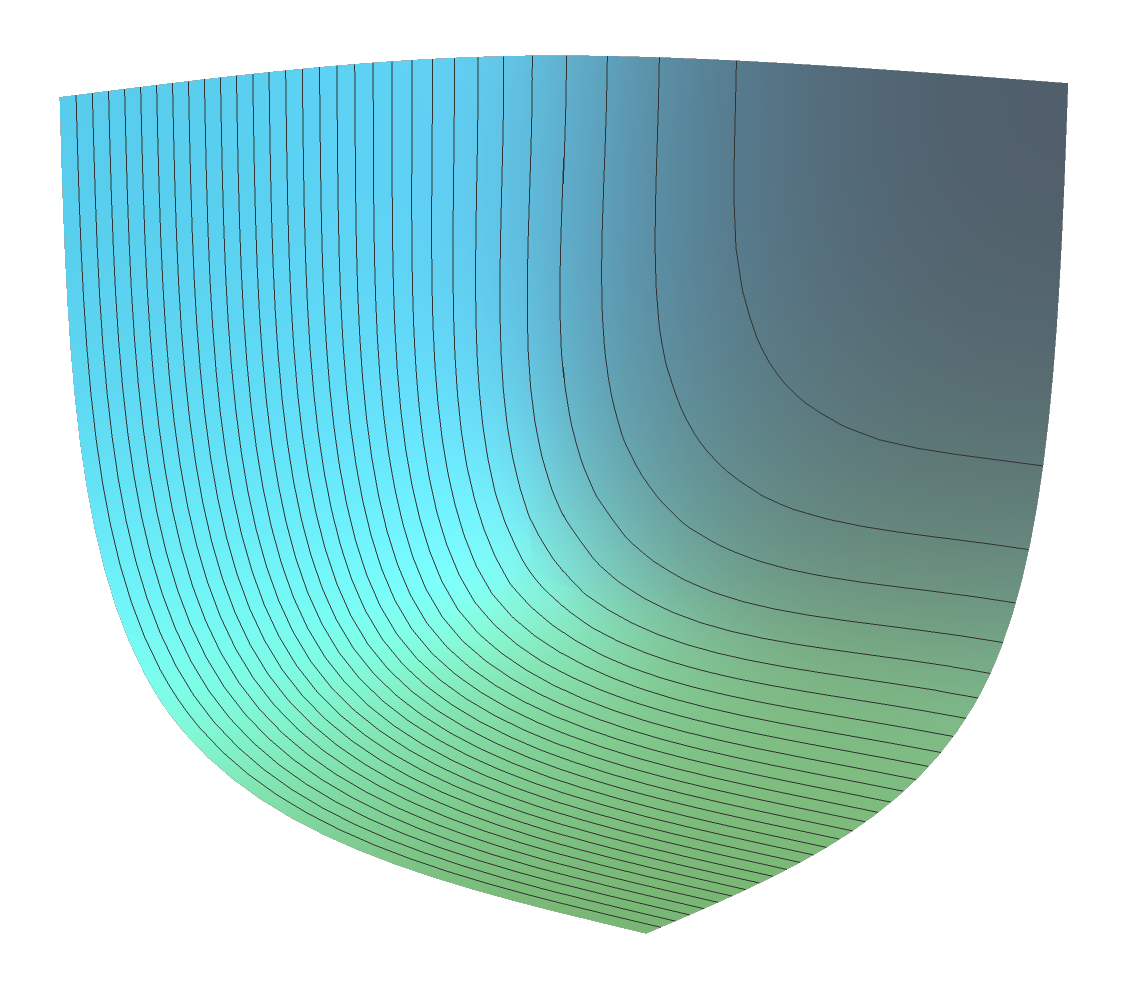} \\
\includegraphics[width=0.9\textwidth]{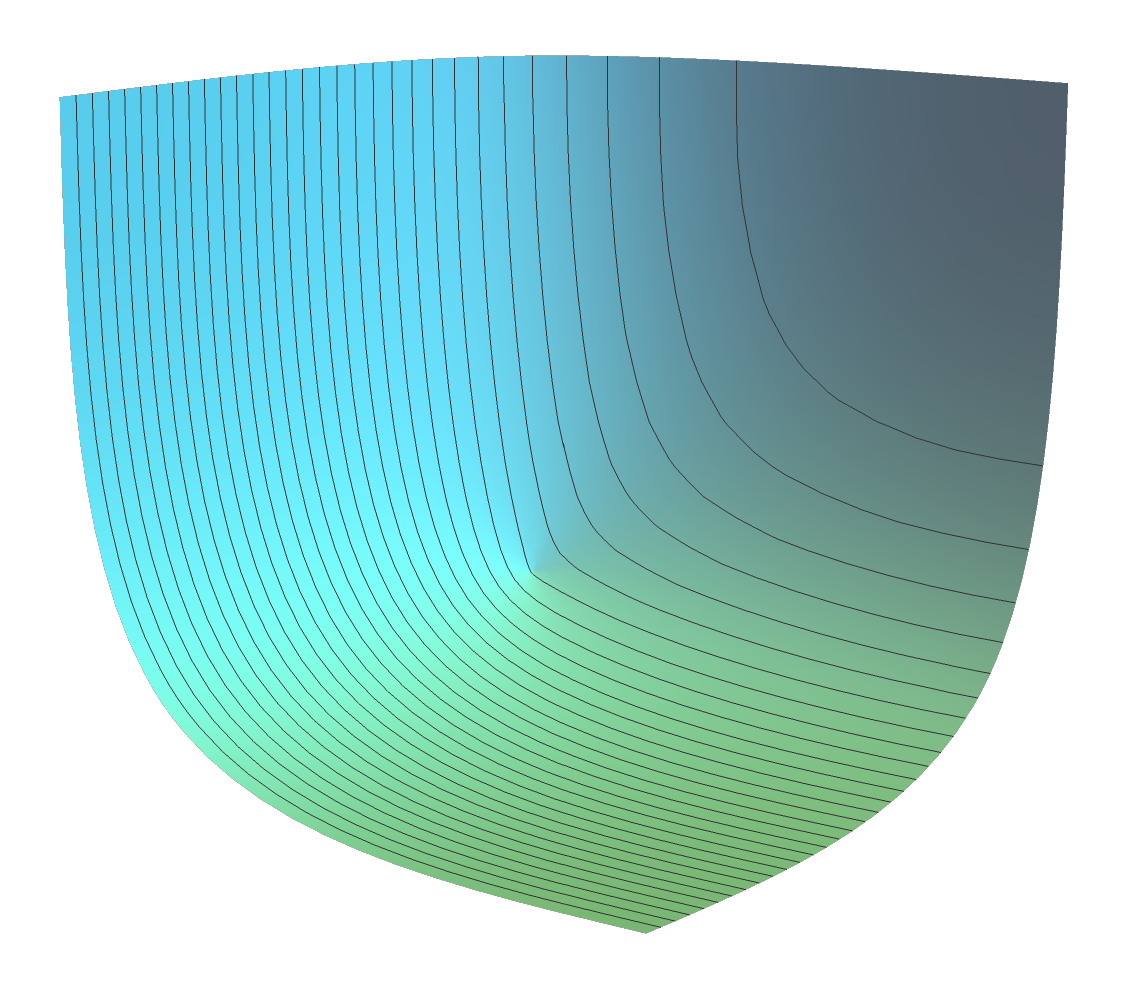} 
\caption{Surfaces with $p=3$, $r=2$}\label{fig:surface-valence-3-b}
\end{subfigure}
\begin{subfigure}[b]{0.3\textwidth}
\includegraphics[width=0.9\textwidth]{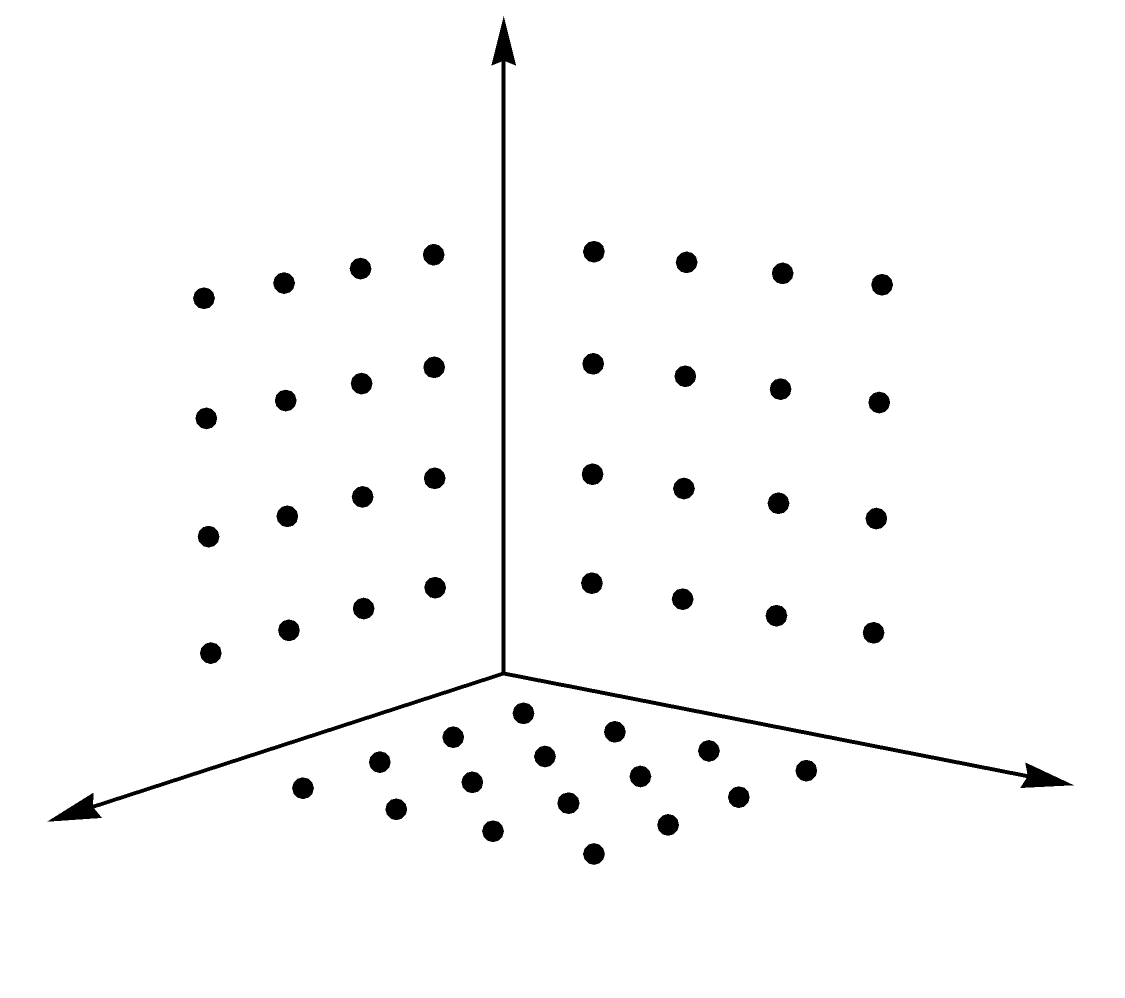} \\
\includegraphics[width=0.9\textwidth]{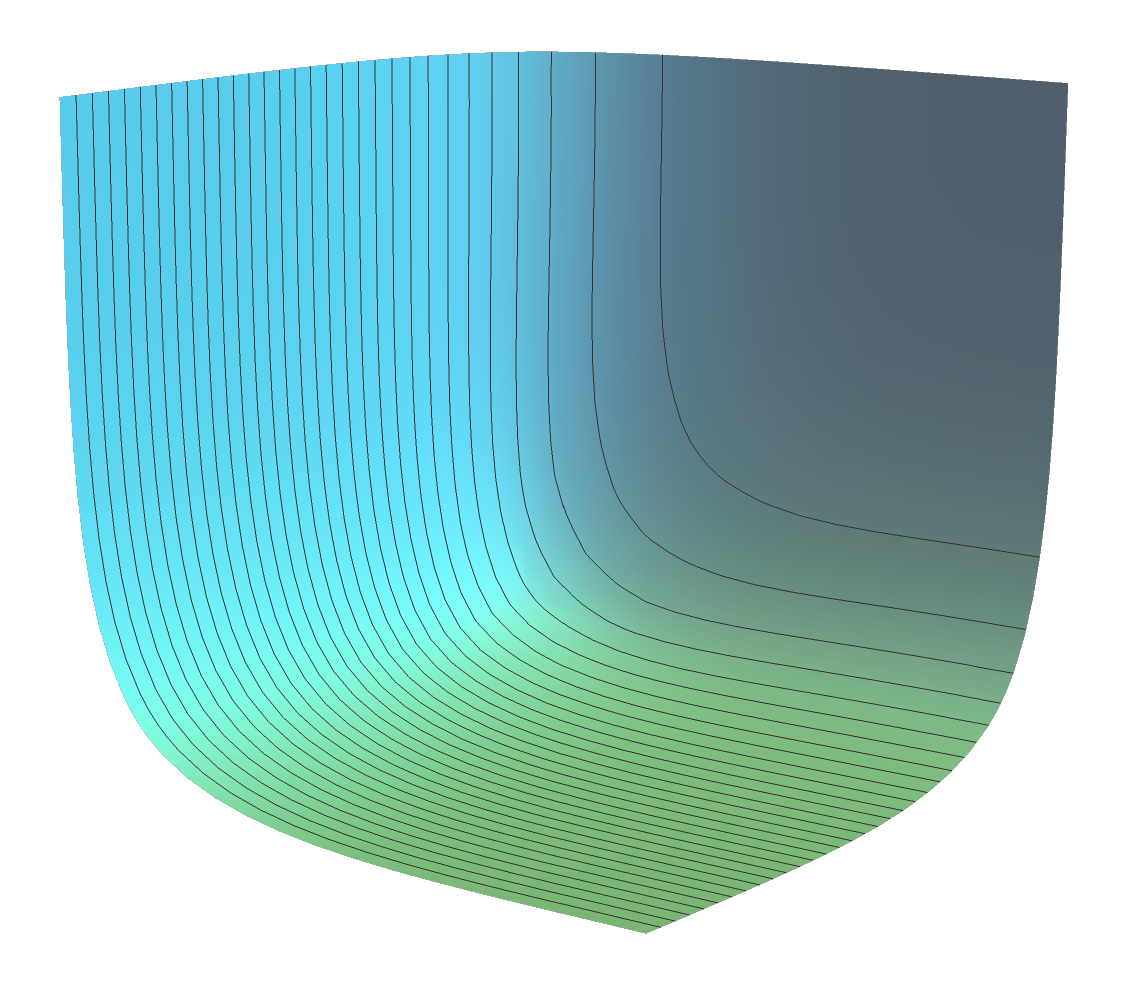} \\
\includegraphics[width=0.9\textwidth]{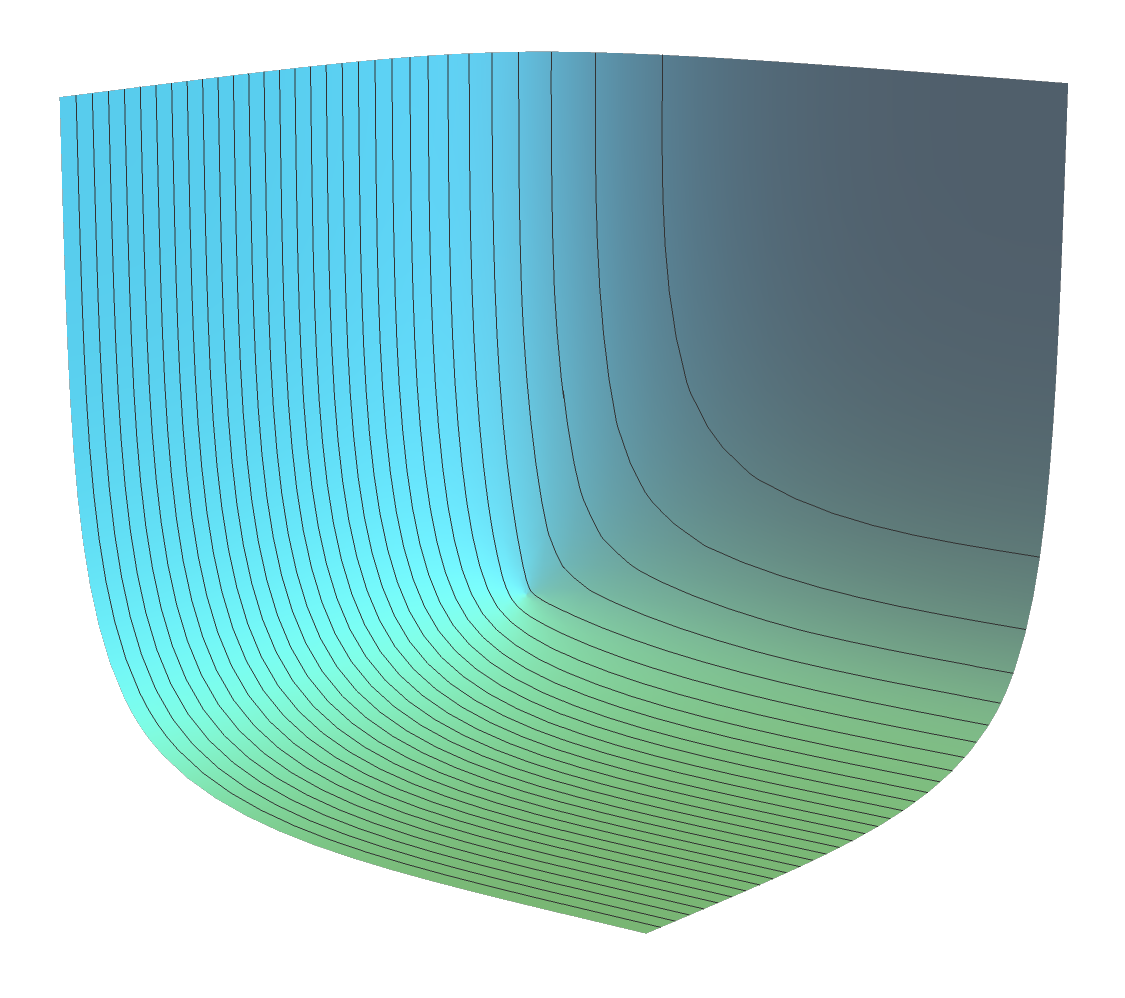} 
\caption{Surfaces with $p=4$, $r=3$}\label{fig:surface-valence-3-c}
\end{subfigure}
\caption{Surface around an EV of valence $3$ after refinement to level $\ell=3$, using coplanar averaging for varying degree. Above are the DOFs, in the middle row the surfaces for $\lambda=0.5$ and below the surfaces for $\lambda=0.26$.
}\label{fig:surface-valence-3}
\end{figure}

\begin{figure}[h!]
\centering
\begin{subfigure}[b]{0.3\textwidth}
\includegraphics[width=0.9\textwidth]{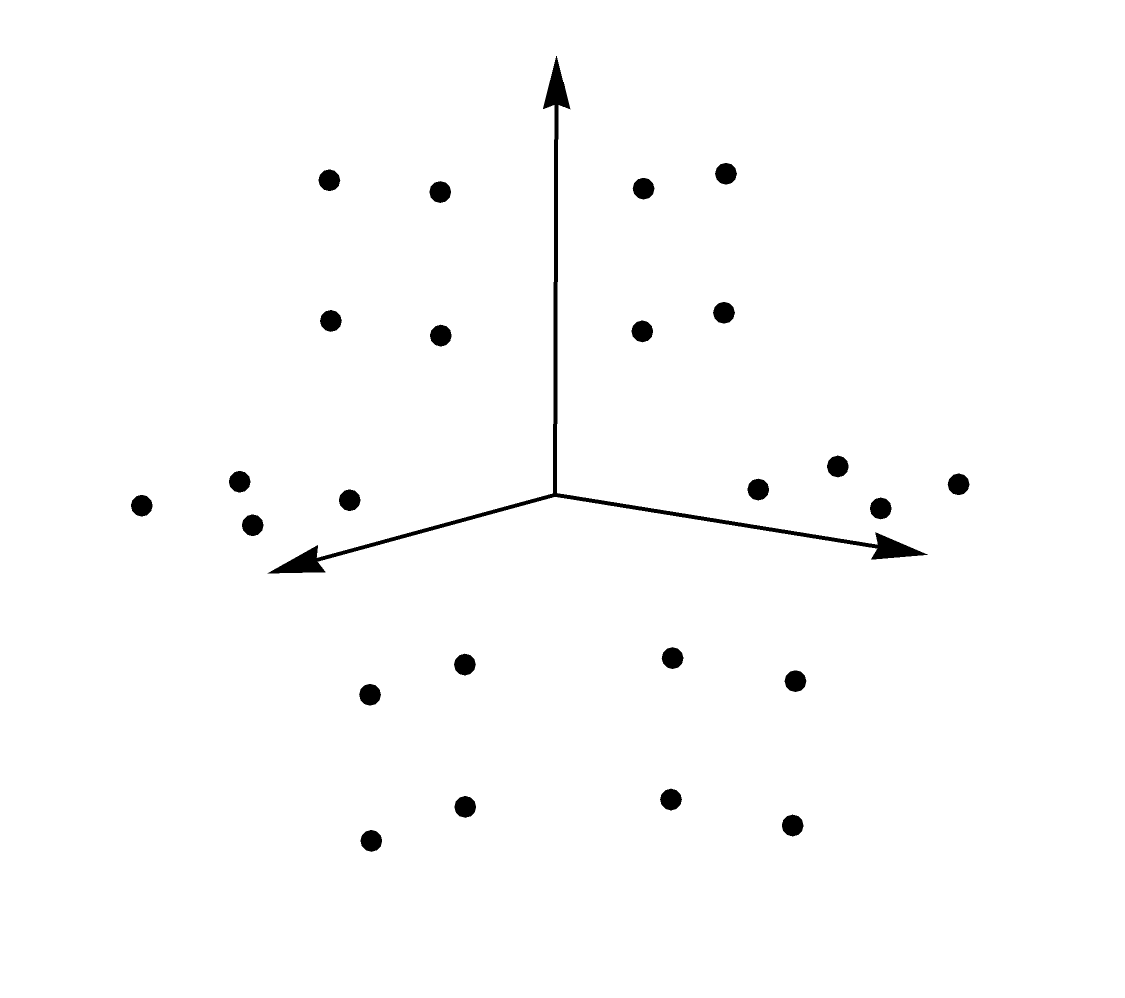} \\[-10pt]
\includegraphics[width=0.9\textwidth]{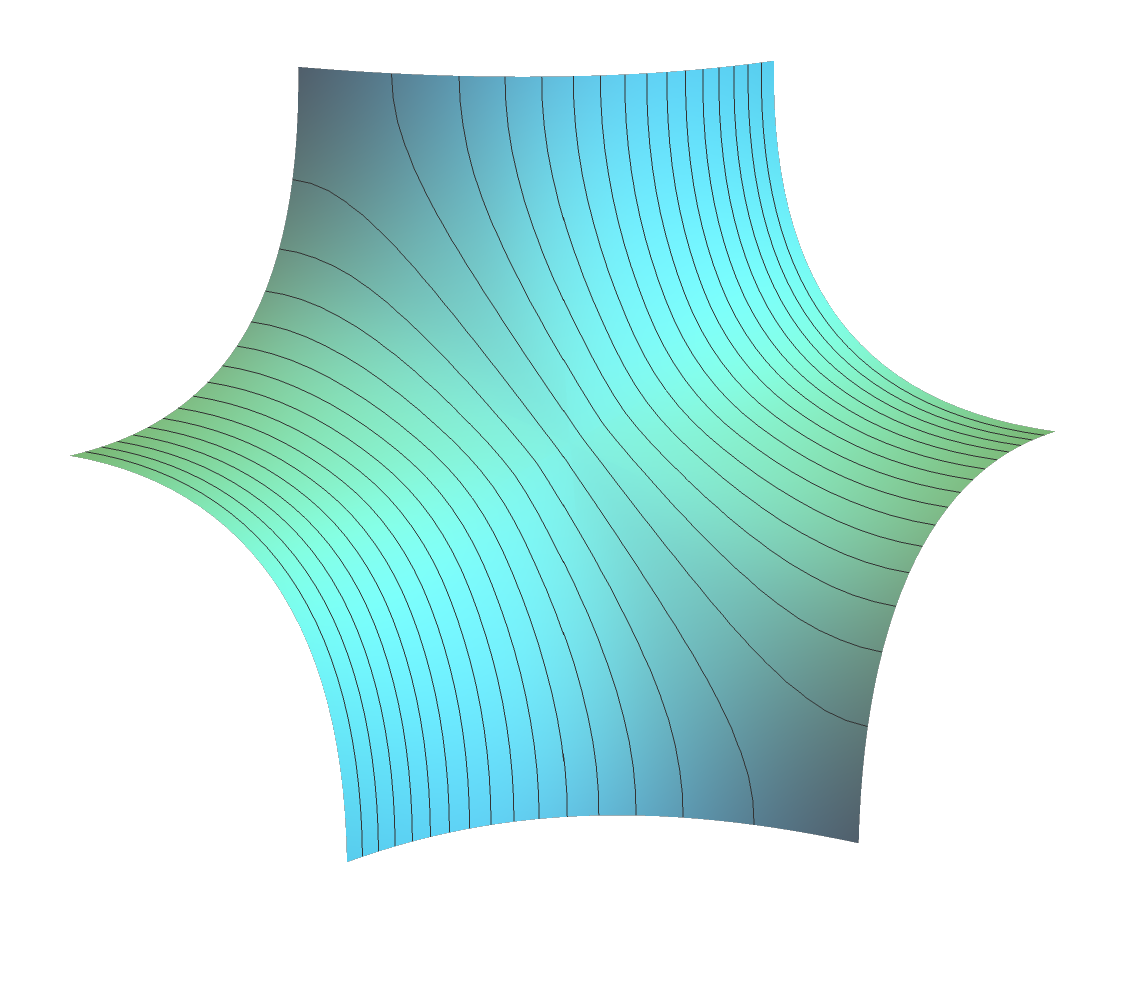} \\[-10pt]
\includegraphics[width=0.9\textwidth]{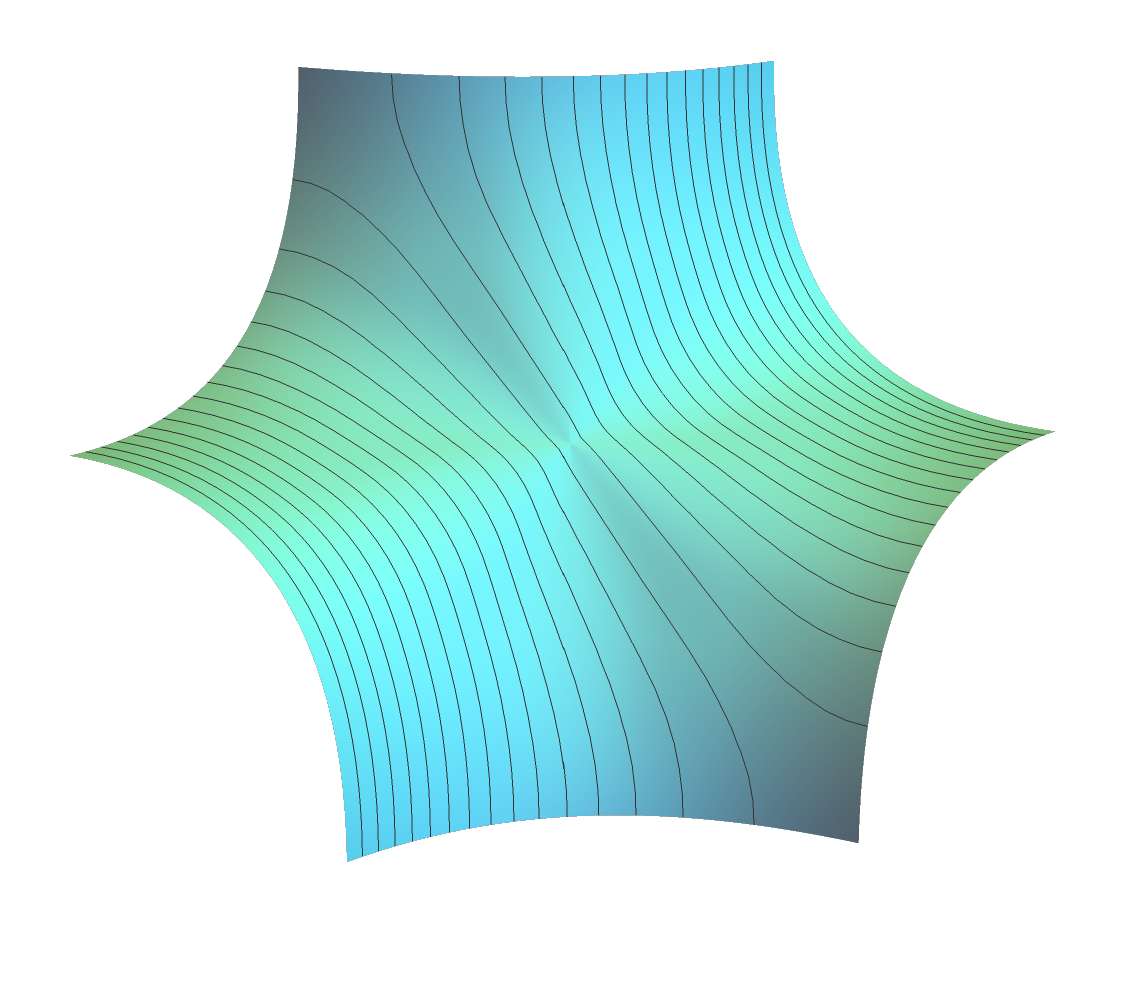} 
\caption{Surfaces with $p=2$, $r=1$}\label{fig:surface-valence-6-a}
\end{subfigure}
\begin{subfigure}[b]{0.3\textwidth}
\includegraphics[width=0.9\textwidth]{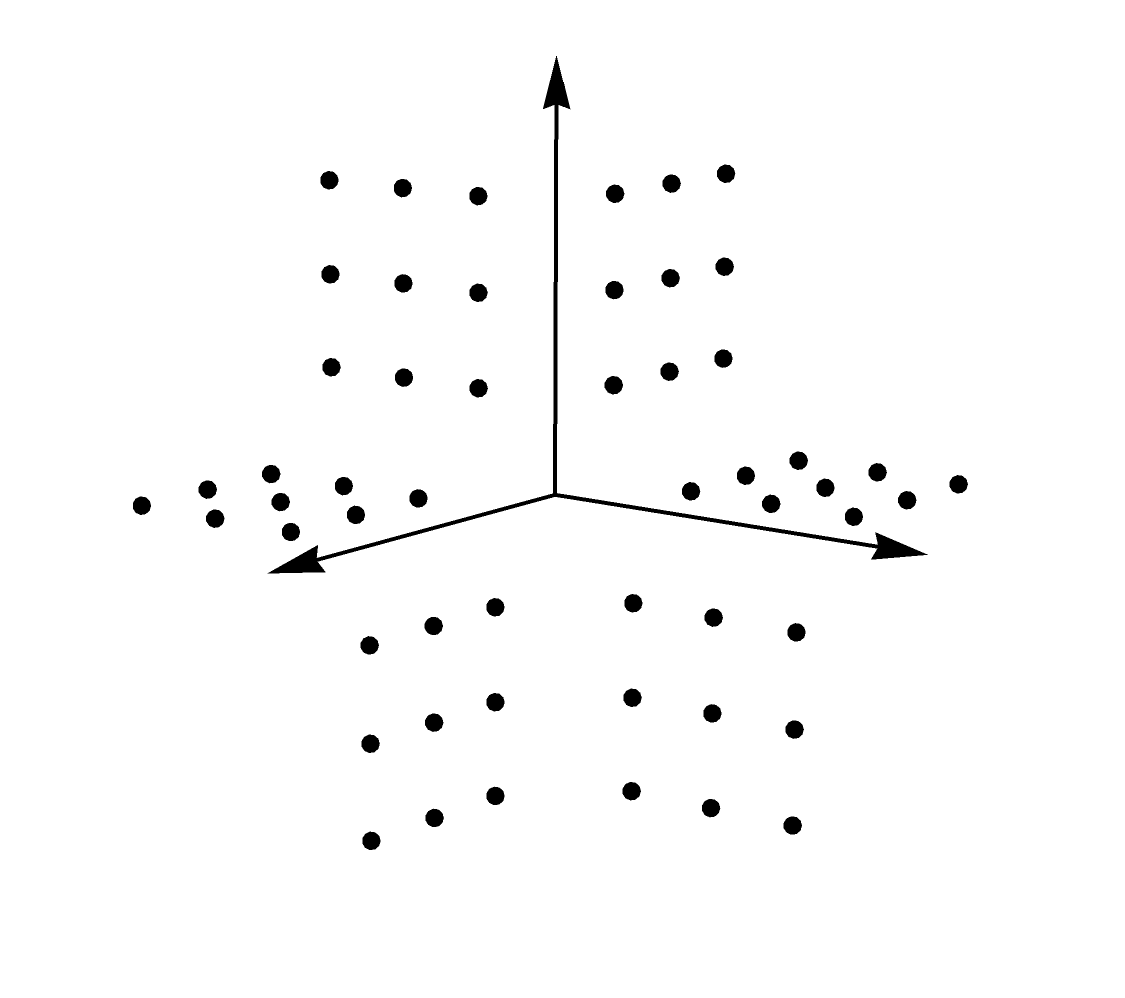} \\[-10pt]
\includegraphics[width=0.9\textwidth]{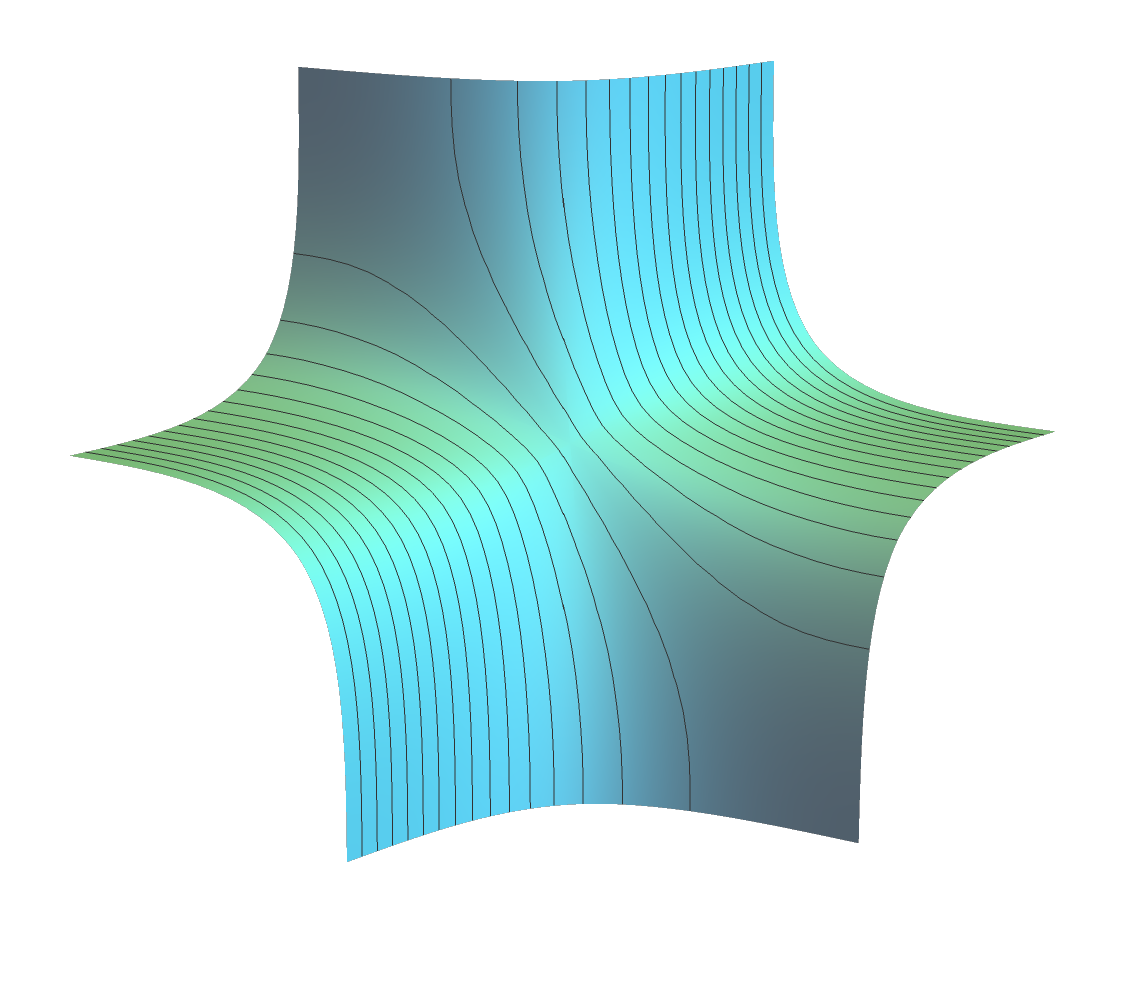} \\[-10pt]
\includegraphics[width=0.9\textwidth]{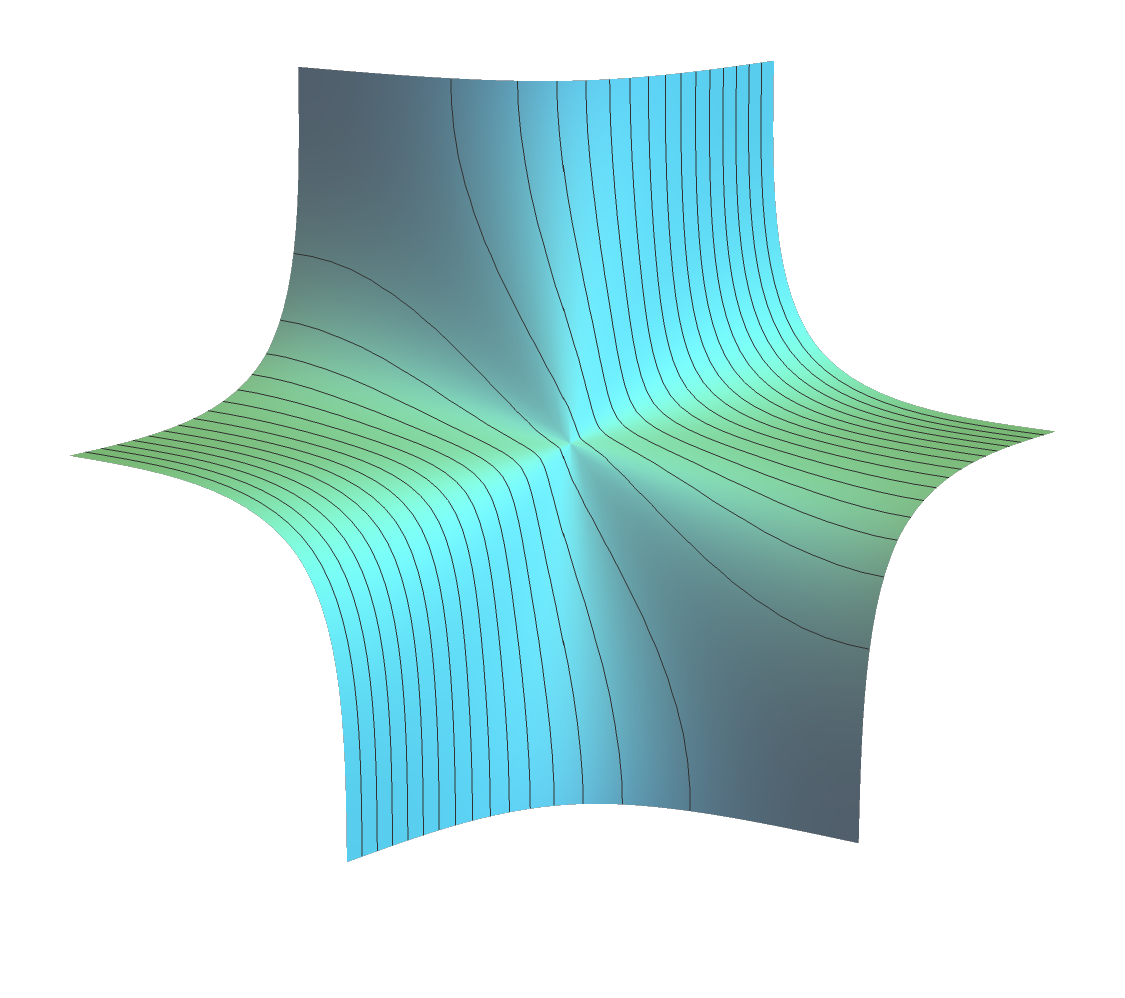} 
\caption{Surfaces with $p=3$, $r=2$}\label{fig:surface-valence-6-b}
\end{subfigure}
\begin{subfigure}[b]{0.3\textwidth}
\includegraphics[width=0.9\textwidth]{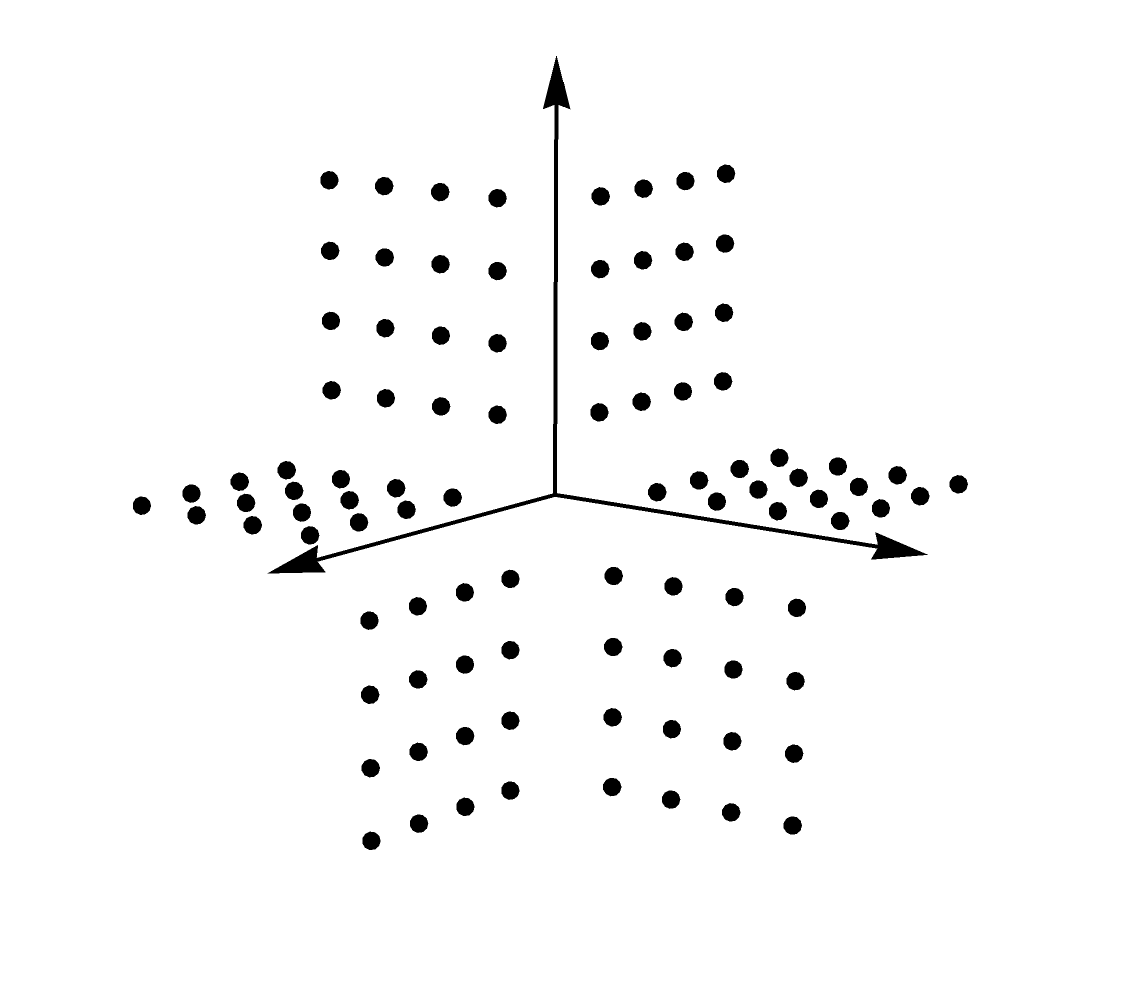} \\[-10pt]
\includegraphics[width=0.9\textwidth]{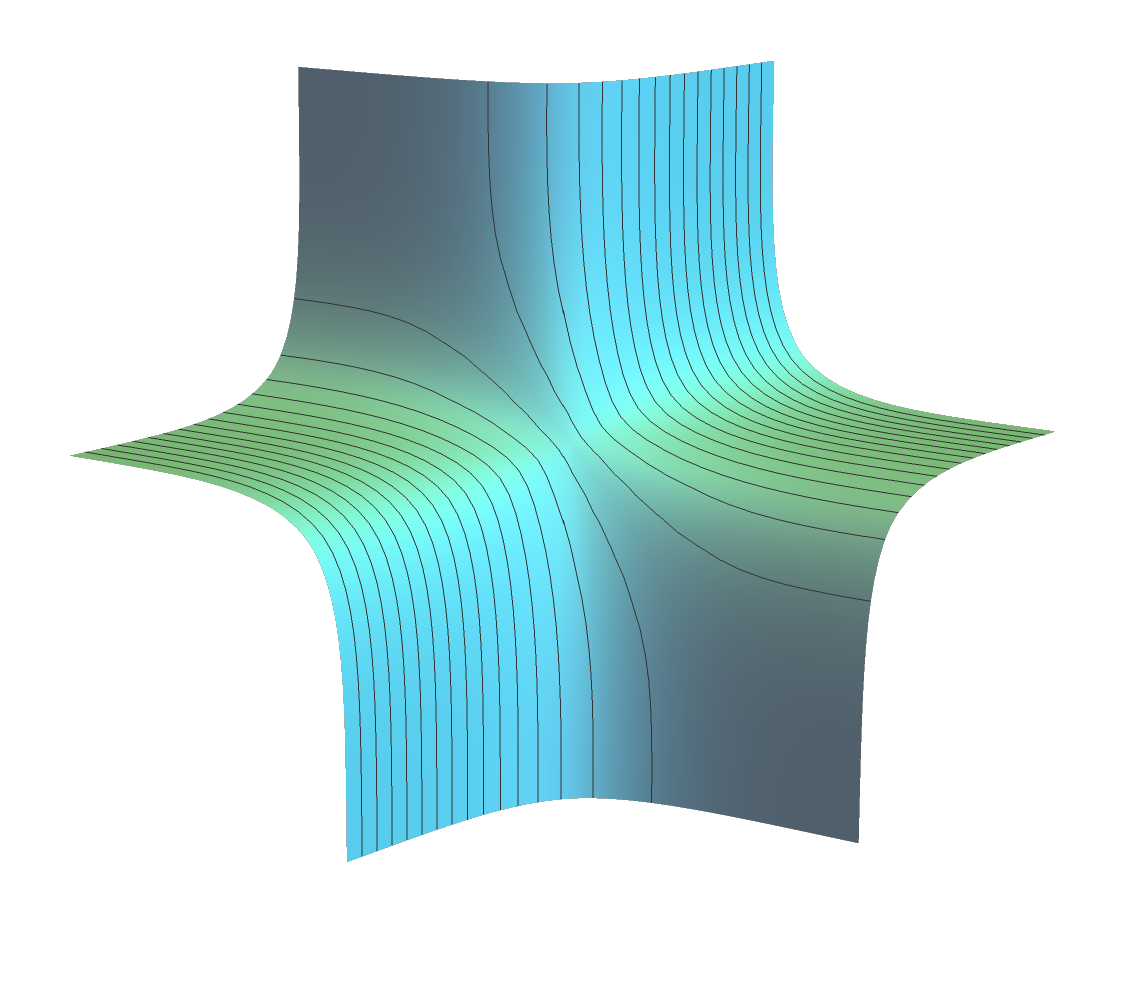} \\[-10pt]
\includegraphics[width=0.9\textwidth]{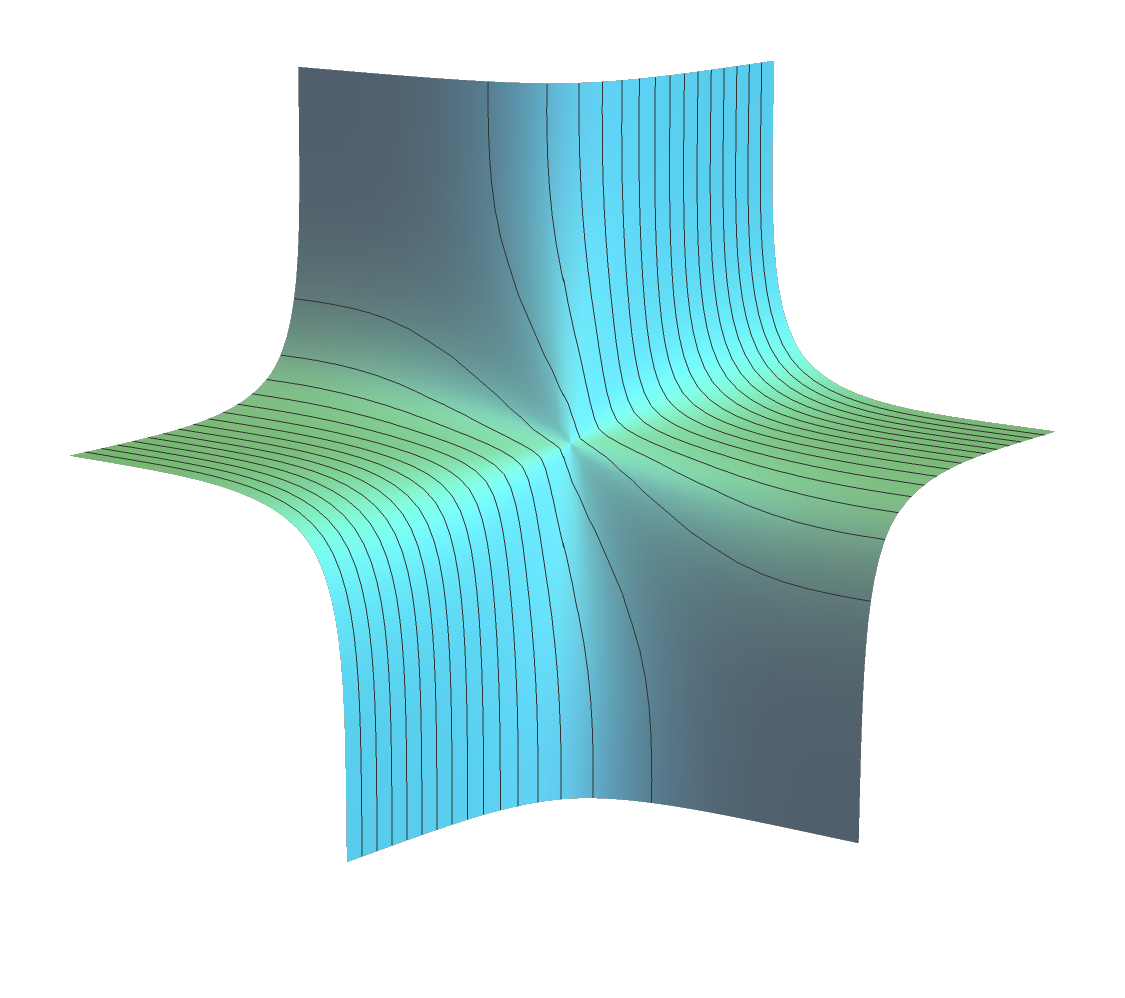} 
\caption{Surfaces with $p=4$, $r=3$}\label{fig:surface-valence-6-c}
\end{subfigure}
\caption{Surface around an EV of valence $6$ after refinement to level $\ell=3$, using coplanar averaging for varying degree. Above are the DOFs, in the middle row the surfaces for $\lambda=0.5$ and below the surfaces for $\lambda=0.26$.
}\label{fig:surface-valence-6}
\end{figure}

Figure~\ref{fig:surface-valence-3} shows surfaces around an extraordinary vertex of valence 3 for degrees $p=2$, $3$, and $4$ after three refinement steps with coplanar averaging. Two values of $\lambda$ are shown: $0.5$ and $0.26$. For $\lambda=0.5$ the surfaces are smoothly rounded at the vertex. For the smaller $\lambda=0.26$ the surfaces become noticeably flatter on the three segments around the EV, especially at higher degrees. The contour lines appear smooth, which is in line with the $C^1$ smoothness in the limit.
In Figure~\ref{fig:surface-valence-6} we visualize a valence 6 extraordinary vertex after three refinement steps for degrees $p=2$, $3$, and $4$. The figure uses coplanar averaging with two $\lambda$ values: $0.5$ and $0.26$. As for lower valences, the surfaces appear smooth for both $\lambda$ values, with a more smoothed-out appearance for $\lambda=0.5$.

\subsection{Interpolation tests on a square domain}\label{sec:interpolation}

In the following we present the results of solving an interpolation problem for the given function 
\begin{equation}\label{eq:test-function}
    \varphi(x,y) = \sin(x-0.5)\cos(1.5\,y+0.3)
\end{equation}
defined on the unit square. We interpolate at the mapped Greville points corresponding to the DOFs of the spline space. The geometry is parameterized using splines of varying degree $p$, regularity $r$ and eigenvalue $\lambda$. Note that the change in patch-interior regularity $r$ has no significant effect on the geometry parameterization. Different parameterizations are shown in Figure~\ref{fig:domains-coplanar}.

\begin{figure}[h!]
\centering
\begin{subfigure}[b]{0.16\textwidth}
\includegraphics[width=0.99\textwidth]{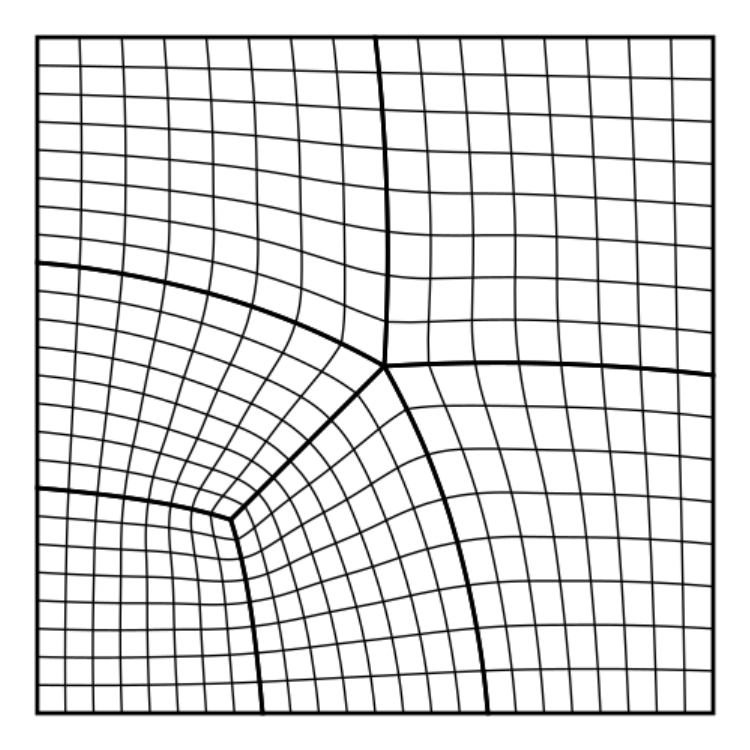} \\ 
\includegraphics[width=0.99\textwidth]{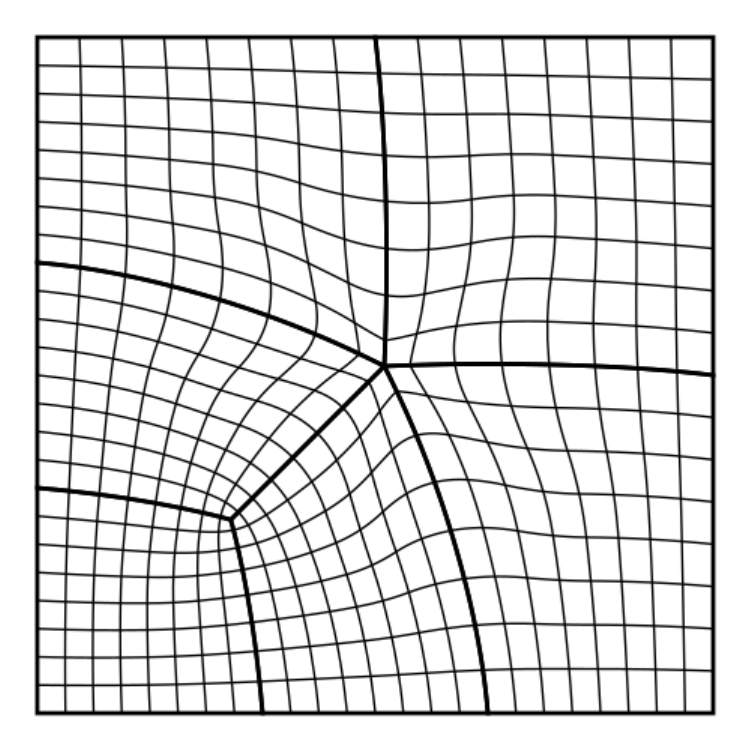} \\
\includegraphics[width=0.99\textwidth]{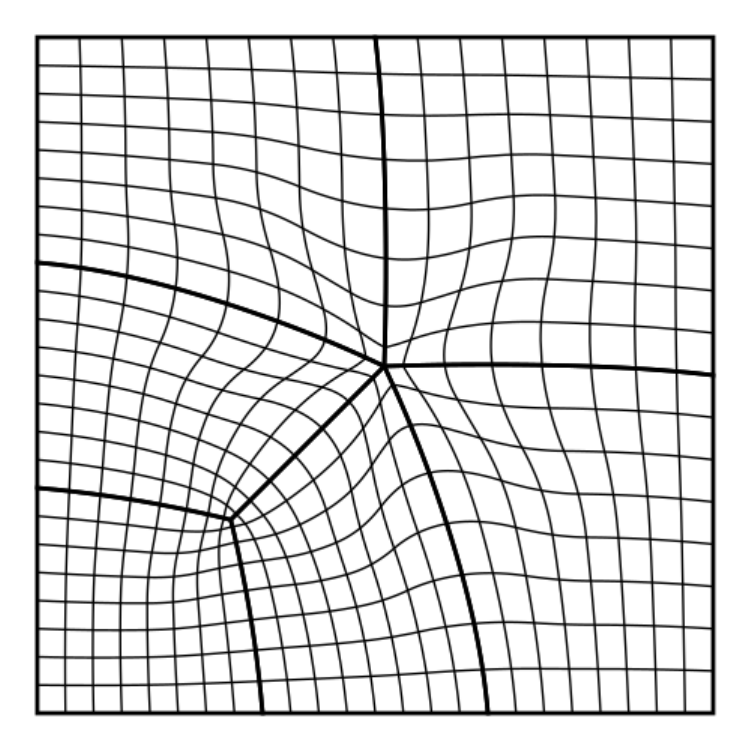} 
\caption{$p=2$, coplanar}
\end{subfigure}
\begin{subfigure}[b]{0.16\textwidth}
\includegraphics[width=0.99\textwidth]{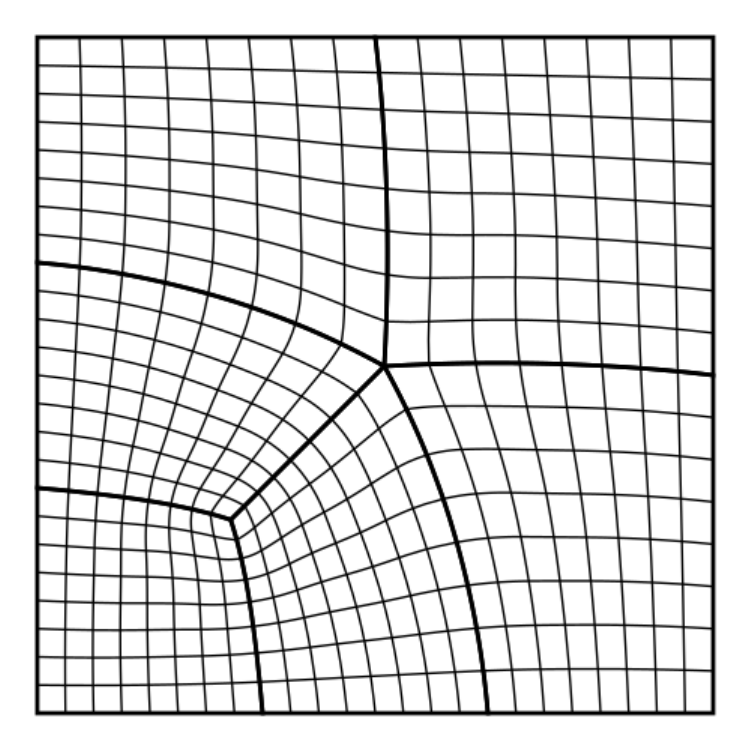} \\ 
\includegraphics[width=0.99\textwidth]{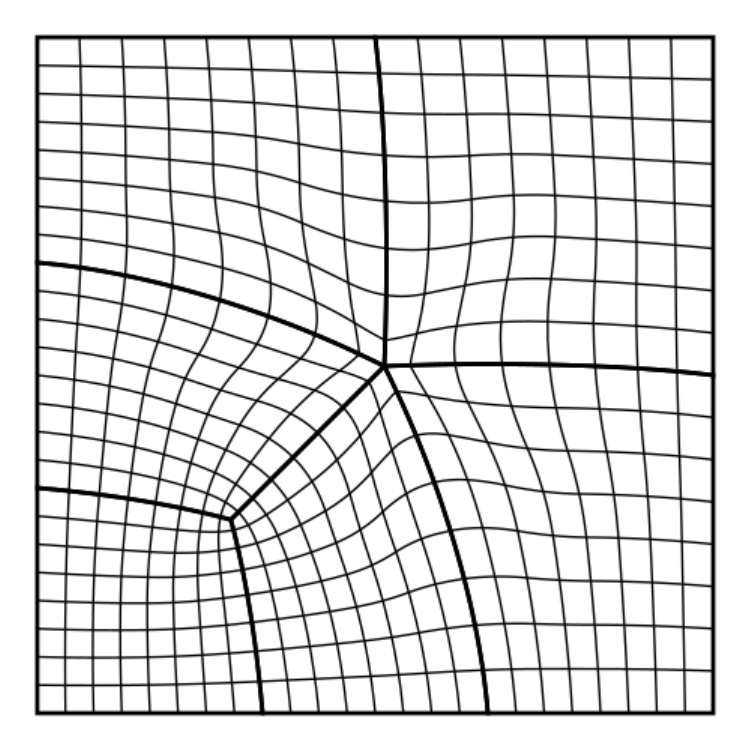} \\
\includegraphics[width=0.99\textwidth]{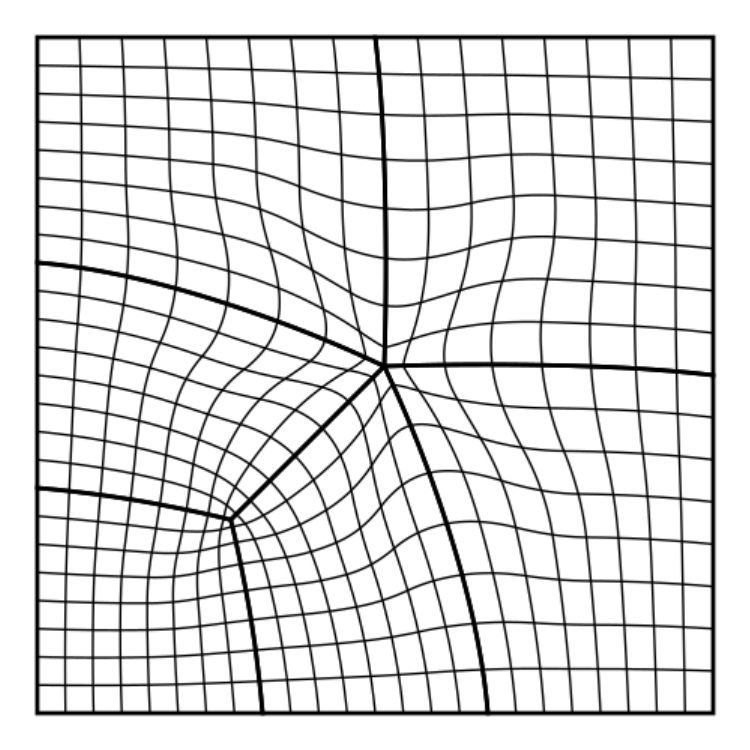} 
\caption{$p=2$, simple}
\end{subfigure}
\begin{subfigure}[b]{0.16\textwidth}
\includegraphics[width=0.99\textwidth]{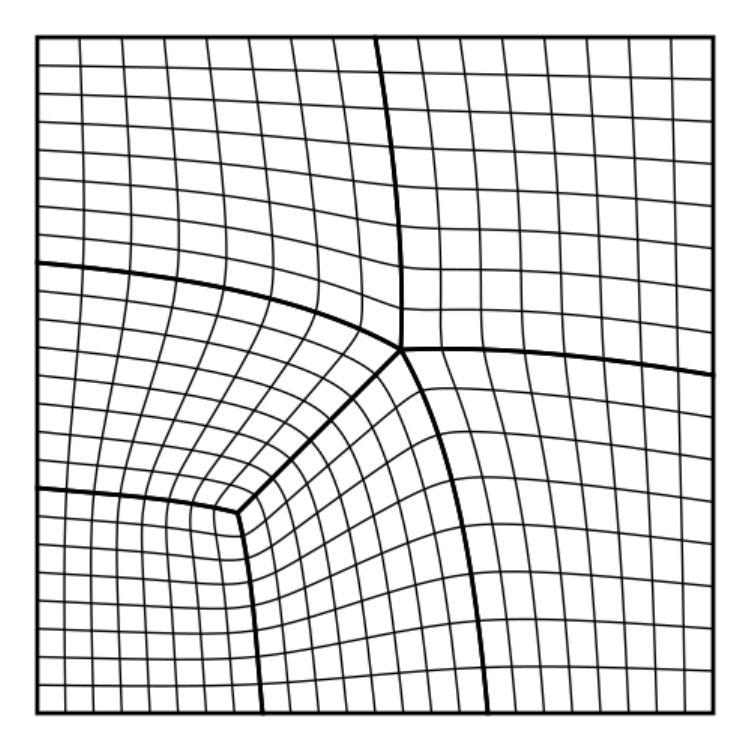} \\ 
\includegraphics[width=0.99\textwidth]{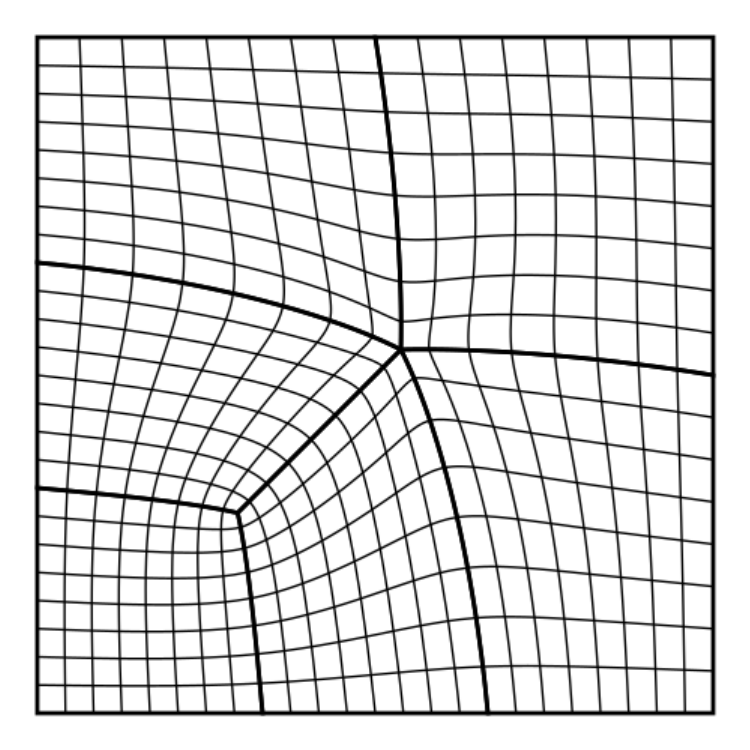} \\
\includegraphics[width=0.99\textwidth]{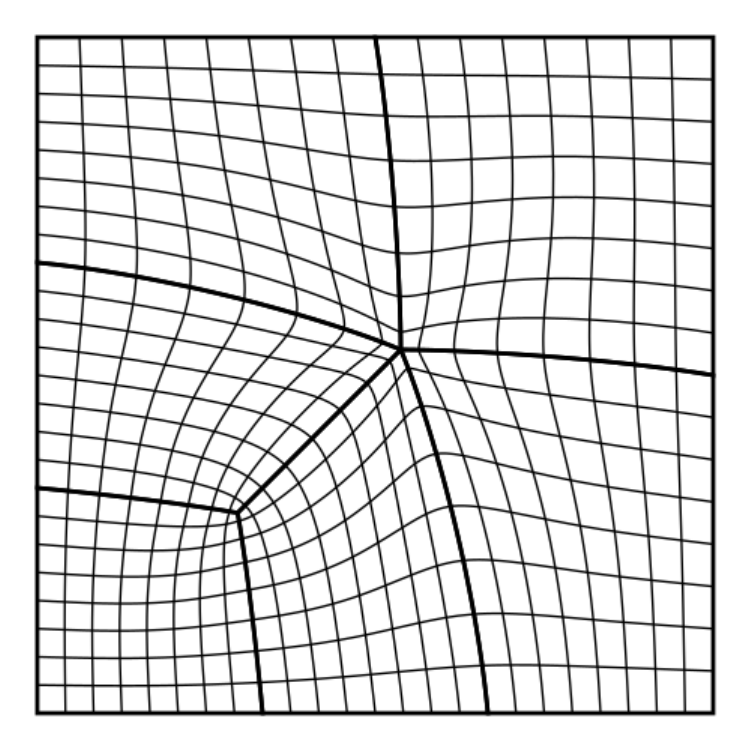} 
\caption{$p=3$, coplanar}
\end{subfigure}
\begin{subfigure}[b]{0.16\textwidth}
\includegraphics[width=0.99\textwidth]{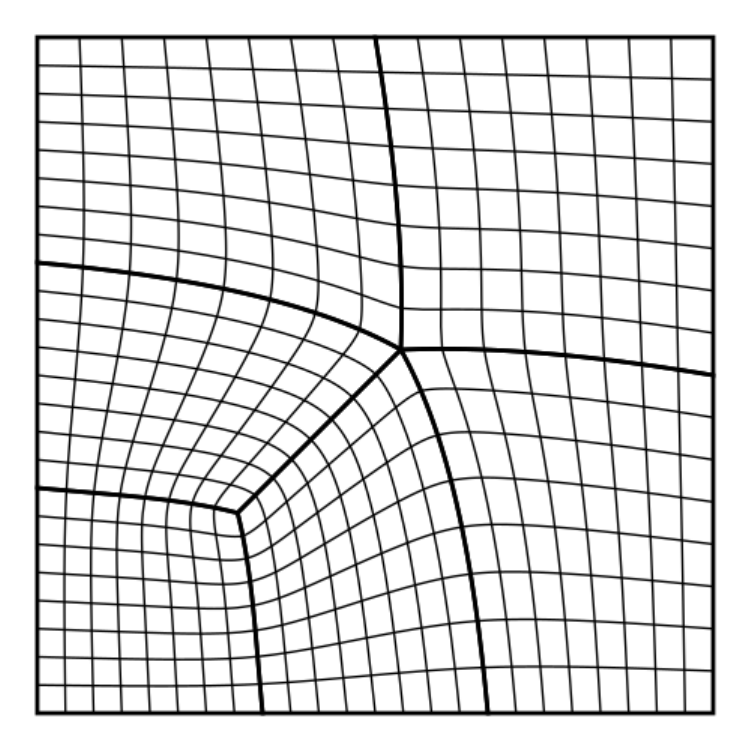} \\ 
\includegraphics[width=0.99\textwidth]{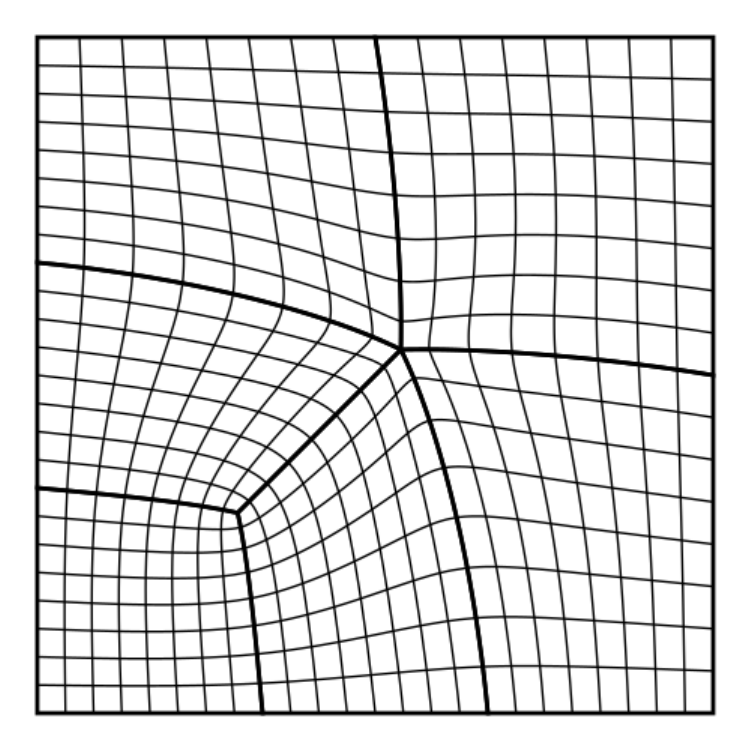} \\
\includegraphics[width=0.99\textwidth]{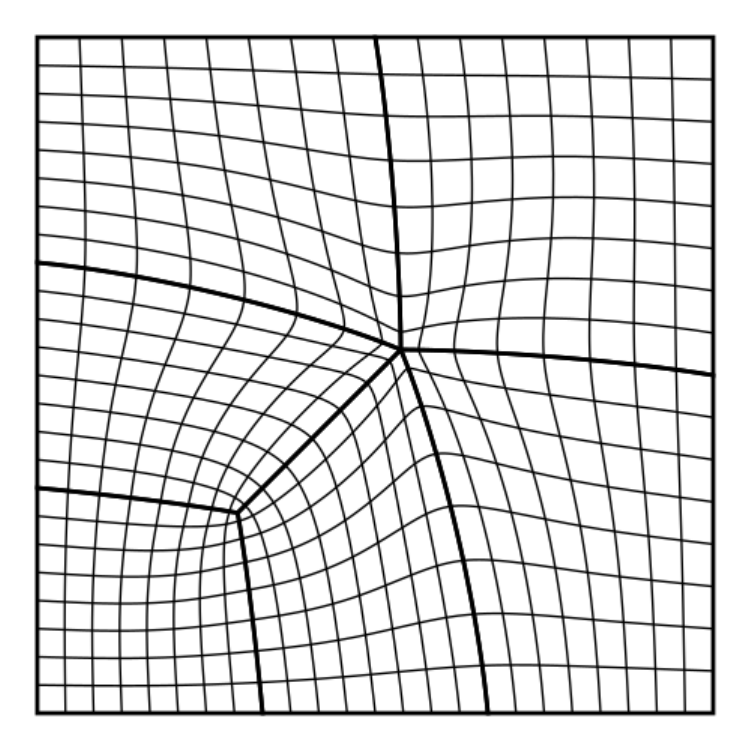} 
\caption{$p=3$, simple}
\end{subfigure}
\begin{subfigure}[b]{0.16\textwidth}
\includegraphics[width=0.99\textwidth]{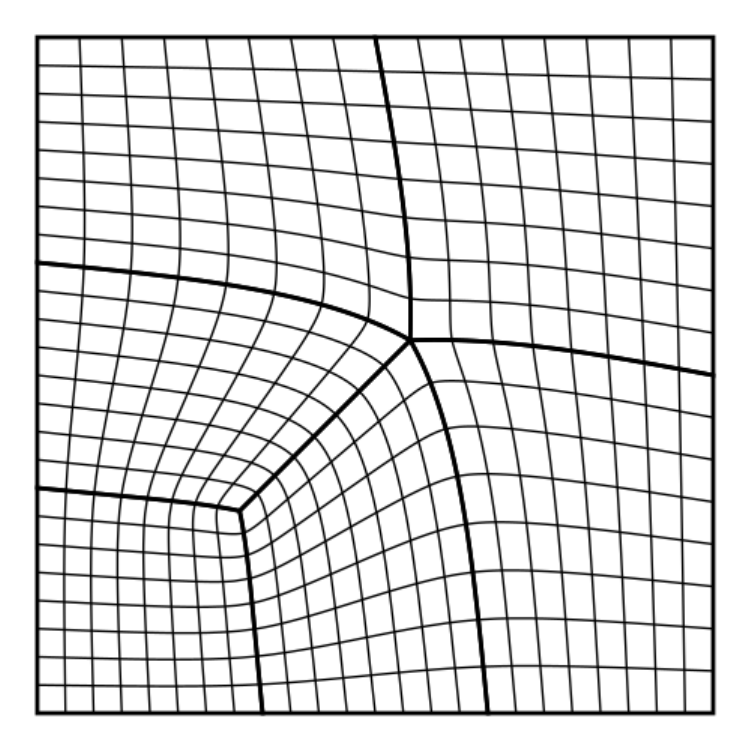} \\ 
\includegraphics[width=0.99\textwidth]{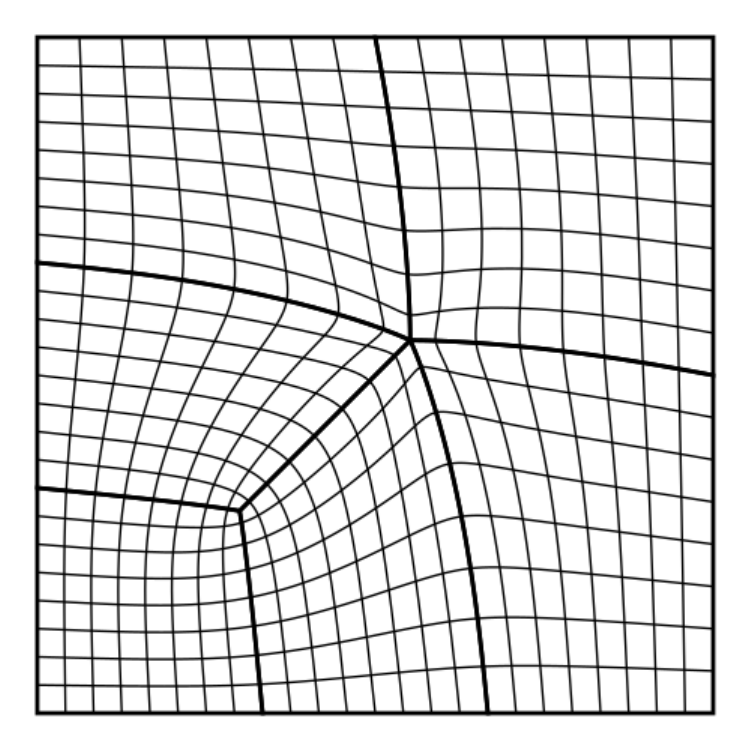} \\
\includegraphics[width=0.99\textwidth]{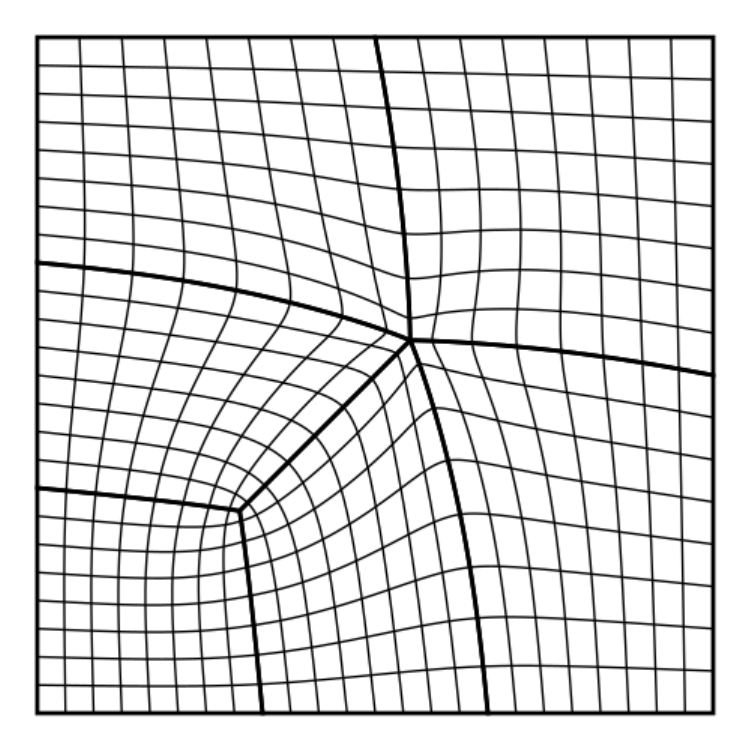} 
\caption{$p=4$, coplanar}
\end{subfigure}
\begin{subfigure}[b]{0.16\textwidth}
\includegraphics[width=0.99\textwidth]{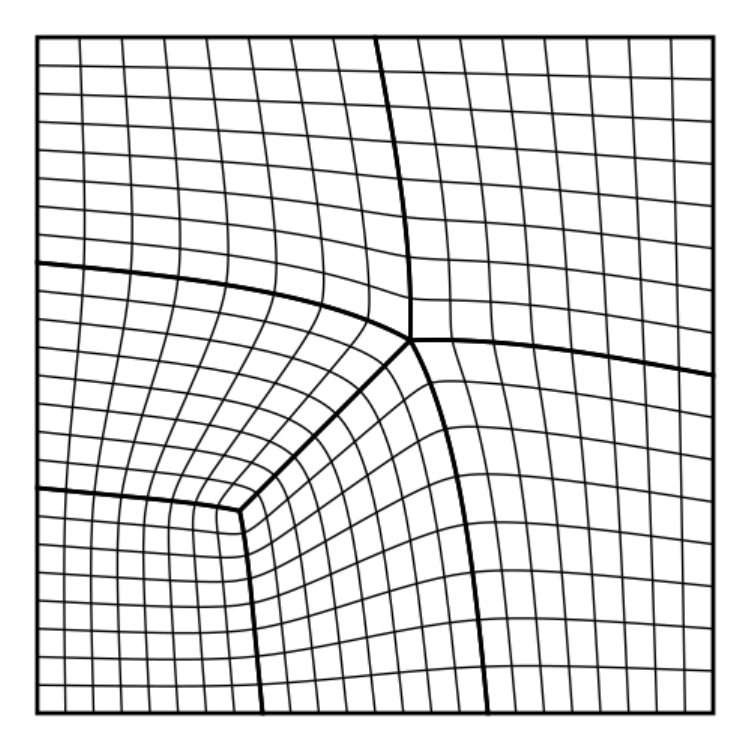} \\ 
\includegraphics[width=0.99\textwidth]{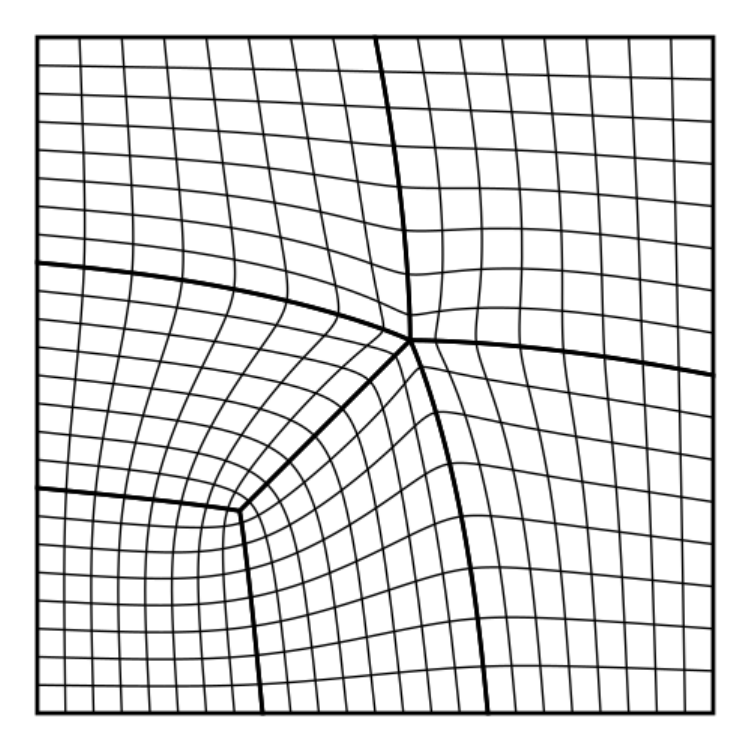} \\
\includegraphics[width=0.99\textwidth]{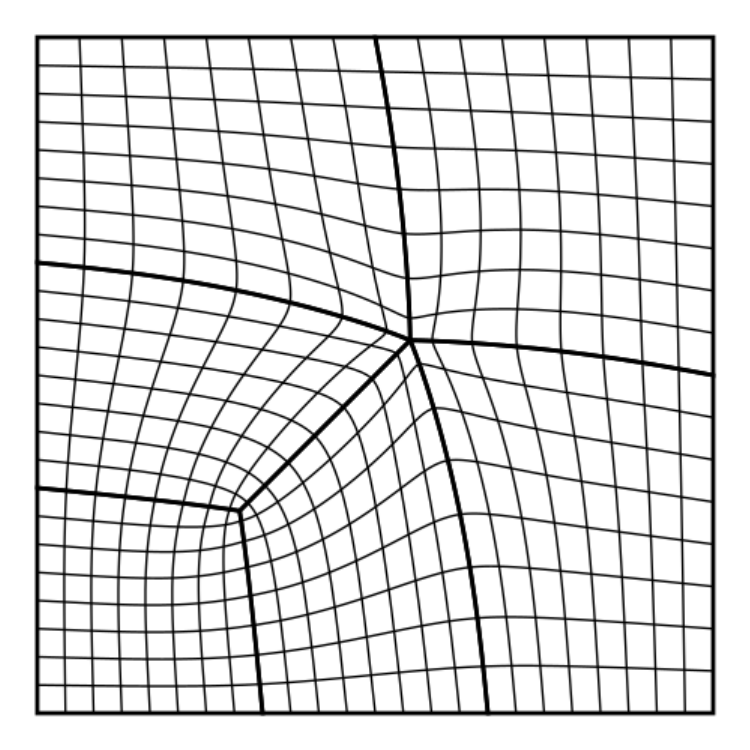} 
\caption{$p=4$, simple}
\end{subfigure}
\caption{Mesh on the unit square domain for varying degree $p=2,3,4$ with maximum smoothness $r=p-1$ and $0.26 \leq \lambda \leq 0.5$ using coplanar and simple averaging. In the top row we have $\lambda=0.5$, in the middle row $\lambda=0.35$ ($p=2$), $\lambda=0.397$ ($p=3$) and $\lambda=0.315$ ($p=4$), and in the bottom row $\lambda=0.26$.}\label{fig:domains-coplanar}
\end{figure}

For all examples, we compute the $L^\infty$-error of the interpolation by evaluating in a fine grid on the parameter domain of each patch. The theory from~\cite{takacs2025approximation} implies that the $L^\infty$-error converges at most with the rate $\max\{\lambda^{\kappa+1},1/2^{p+1}\}$, where $\kappa$ is the polynomial reproduction degree in physical coordinates and $1/2^{p+1}$ is the best possible rate. For our spline construction, we have $\kappa=1$, so we desire $\lambda_E^{2}=1/2^{p+1}$, or equivalently, $\lambda_E=2^{-(p+1)/2}$ for all extraordinary vertices $E$. For $p=2$, the optimal choice is therefore $\lambda_E=0.353553$. For $p \geq 3$, the optimal choice would be $\lambda_E \leq \frac14$, which is not a feasible choice. Thus, we test with $\lambda_E=0.26$ for higher degrees.

%for p=2
\begin{figure}[h!]
    \centering
\begin{tikzpicture}
\begin{axis}[
    width=.42\textwidth,
    height=.41\textwidth,
    xlabel={refinement level},
    ylabel={$L^\infty$-error },
    xtick={0,1,2,3,4,5},
    ymode=log,
    xmin=-0.2, xmax=5.2,
    ymax=0.2, ymin=0.0000005,
    ytick={0.1,0.01,0.001,0.0001,0.00001,0.000001},
    mark options={solid},
    grid=both,
    tick label style={font=\small},
    label style={},
    legend style={font=\small},
   legend pos=outer north east,
    legend style={font=\small},
    title={}
]
\addplot[
    color=blue!40,
    % p=2, r=1, coplanar, lambda=0.5
    mark=*,
    mark size=3pt,
    thick,
] coordinates {
    (0, 0.025219042757692933)
    (1, 0.004054283276559385)
    (2, 0.0006671908150575057)
    (3, 0.00017467439397444995)
    (4,0.000049614245172217505)
    (5, 0.000013254703408358015)

};
\addplot[
    color=blue,
     % p=2, r=1, simple, lambda=0.5
    mark=*,
    mark size=3pt,
    dashed,
] coordinates {
    
   (0, 0.02544366889792149)
    (1, 0.0036809025482960334)
    (2, 0.0011069043304047371)
    (3, 0.0003056666656973822)
    (4, 0.00008160548731417448)
    (5, 0.000020039328326980094)
};
\addplot[
    color=green,
     % p=2, r=1, coplanar, lambda=0.35355
    mark=*,
    mark size=3pt,
    thick,
] coordinates {
    (0, 0.03422473105157038)
    (1, 0.010810502074417264)
    (2, 0.0020301601920843376)
    (3, 0.00031537166999965477)
    (4, 0.00004702906631297559)
    (5, 0.000006311701615984225)
};
\addplot[
    color=teal,
     % p=2, r=1, simple, lambda=0.35355
    mark=*,
    mark size=3pt,
    dashed,
] coordinates {
   (0, 0.0373440426036076)
    (1, 0.01155335298873697)
    (2, 0.002009113624439323)
    (3, 0.0003054157399812324)
    (4, 0.00004457758171359837)
    (5, 0.000005999308386407144)
};
\addplot[
    color=orange!50,
     % p=2, r=1, coplanar, lambda=0.26
    mark=*,
    mark size=3pt,
    thick,
] coordinates {
    (0,0.05821121080813385)
    (1, 0.016315295015170858)
    (2, 0.0018837673755804016)
    (3, 0.00019997007080299134)
    (4, 0.000020372856694551508)
    (5, 0.000002255205856283027)
};
\addplot[
    color=orange,
    % p=2, r=1, simple, lambda=0.26
    mark=*,
    mark size=3pt,
    dashed,
] coordinates {
    (0, 0.058326944908411005)
    (1, 0.016356772382893542)
    (2, 0.001872714724303285)
    (3, 0.00019794610913977757)
    (4, 0.000020159223427847033)
    (5, 0.0000022543727909891953)
};
\addplot[
    color=gray,
    thick,
    dotted,
] coordinates {
    (3, 0.001)
    (4, 0.001/4)
    (5, 0.001/16)
};
\addplot[
    color=gray,
    thick,
    dashed,
] coordinates {
    (3, 0.005/64)
    (4, 0.005/512)
    (5, 0.005/4096)
};
\legend{$\lambda=0.5$ cop., $\lambda=0.5$ sim., $\lambda=0.35$ cop., $\lambda=0.35$ sim., $\lambda=0.26$  cop., $\lambda=0.26$ sim., $h^{2}\sim1/2^{2\ell}$, $h^{3}\sim1/2^{3\ell}$}
\end{axis}
\begin{axis}[
    xshift=.574\textwidth,
    width=.42\textwidth,
    height=.41\textwidth,
    xlabel={Degrees of Freedom (DOFs)},
    xmode=log,
    xtick={10, 100, 1000, 10000},
    xmin=10, xmax=10000,
    ymax=0.2, ymin=0.0000005,
    ymode=log,
    ytick={0.1,0.01,0.001,0.0001,0.00001,0.000001,0.0000001},
    yticklabel pos=right,
    mark options={solid},,
    grid=both,
    tick label style={font=\small},
    label style={},
    legend style={font=\small,at={(0.97,0.97)},anchor=north east},
    legend style={font=\small},
    legend pos=outer north east,
    title={}
]
\addplot[
    color=blue!40,
    % p=2, r=1, coplanar, lambda=0.5
    mark=*,
    mark size=3pt,
    thick,
] coordinates {
    (20, 0.025219042757692933)
    (48, 0.004054283276559385)
    (140, 0.0006671908150575057)
    (468, 0.00017467439397444995)
    (1700, 0.000049614245172217505)
    (6468, 0.000013254703408358015)

};
\addplot[
    color=blue,
     % p=2, r=1, simple, lambda=0.5
    mark=*,
    mark size=3pt,
    dashed,
] coordinates {
    
   (20, 0.02544366889792149)
    (48, 0.0036809025482960334)
    (140, 0.0011069043304047371)
    (468, 0.0003056666656973822)
    (1700, 0.00008160548731417448)
    (6468, 0.000020039328326980094)
};
\addplot[
    color=green,
     % p=2, r=1, coplanar, lambda=0.35355
    mark=*,
    mark size=3pt,
    thick,
] coordinates {
    (20, 0.03422473105157038)
    (48, 0.010810502074417264)
    (140, 0.0020301601920843376)
    (468, 0.00031537166999965477)
    (1700, 0.00004702906631297559)
    (6468, 0.000006311701615984225)
};
\addplot[
    color=teal,
     % p=2, r=1, simple, lambda=0.35355
    mark=*,
    mark size=3pt,
    dashed,
] coordinates {
   (20, 0.0373440426036076)
    (48, 0.01155335298873697)
    (140, 0.002009113624439323)
    (468, 0.0003054157399812324)
    (1700, 0.00004457758171359837)
    (6468, 0.000005999308386407144)
};
\addplot[
    color=orange!50,
     % p=2, r=1, coplanar, lambda=0.26
    mark=*,
    mark size=3pt,
    thick,
] coordinates {
    (20,0.05821121080813385)
    (48, 0.016315295015170858)
    (140, 0.0018837673755804016)
    (468, 0.00019997007080299134)
    (1700, 0.000020372856694551508)
    (6468, 0.000002255205856283027)
};
\addplot[
    color=orange,
    % p=2, r=1, simple, lambda=0.26
    mark=*,
    mark size=3pt,
    dashed,
] coordinates {
    (20, 0.058326944908411005)
    (48, 0.016356772382893542)
    (140, 0.001872714724303285)
    (468, 0.00019794610913977757)
    (1700, 0.000020159223427847033)
    (6468, 0.0000022543727909891953)
};
\end{axis}
\end{tikzpicture}
    \caption{$L^\infty$-error convergence for the interpolation problem with $p=2$.}
    \label{fig:interpolation-2}
\end{figure}
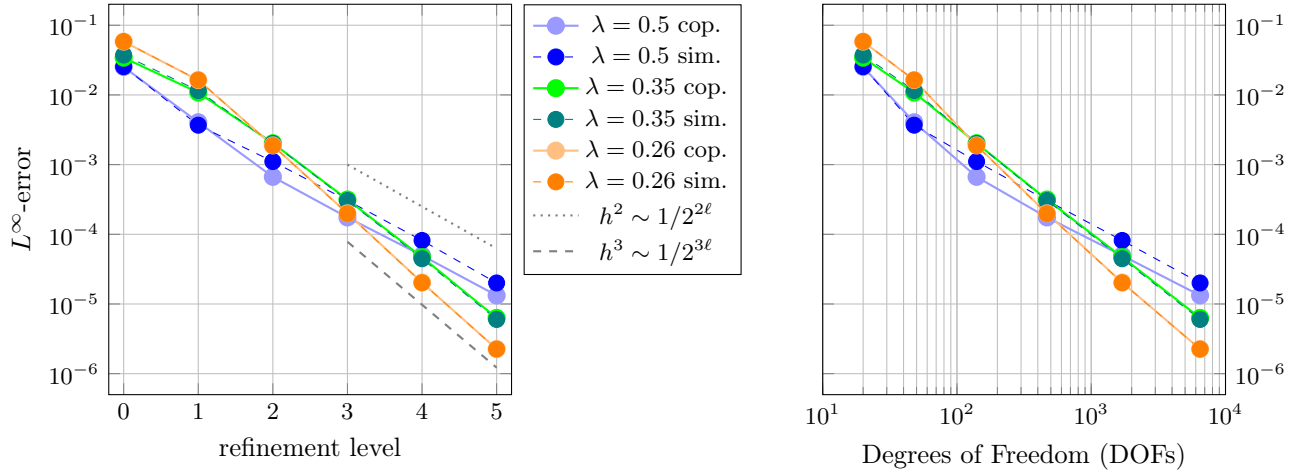

%Dofs plot for p=3
\begin{figure}[h!]
    \centering
\begin{tikzpicture}
\begin{axis}[
    width=.42\textwidth,
    height=.41\textwidth,
    xlabel={refinement level},
    ylabel={$L^\infty$-error },
    xtick={0,1,2,3,4,5},
     xmin=-0.2, xmax=5.2,
    ymode=log,
    ymin=0.00000003,
    ytick={1,0.1,0.01,0.001,0.0001,0.00001,0.000001,0.0000001,0.00000001,0.000000001},
    mark options={solid},
    grid=both,
    tick label style={font=\small},
    label style={},
    legend style={font=\small},
   legend pos=outer north east,
    title={}
]
\addplot[
    color=blue!40,
    % p=3, r=1, coplanar, lambda=0.5
    mark=*,
    mark size=3pt,
    thick,
] coordinates {
    (0, 0.004745925686219142)
    (1, 0.0010809924668855747)
    (2, 0.00025845785944694777)
    (3, 0.00006326915681817424)
    (4, 0.000015404757825965287)
    (5, 0.000003835804647384089)

};
\addplot[
    color=blue,
     % p=3, r=1, simple, lambda=0.5
    mark=*,
    mark size=3pt,
    dashed,
] coordinates {
    
   (0, 0.003943648234133346)
    (1, 0.0011509634238906254)
    (2, 0.00031038353081931386)
    (3, 0.0000795191415758522)
    (4, 0.000019984214296161346)
    (5, 0.000004600683976438424)
};
\addplot[
    color=green,
     % p=3, r=2, coplanar, lambda=0.5
    mark=*,
    mark size=3pt,
    thick,
] coordinates {
    (0, 0.004745925686219142)
    (1, 0.0017211738224215678)
    (2, 0.0004274517914710785)
    (3, 0.00010396137485646495)
    (4, 0.000025596479744594003)
    (5, 0.0000061213278988170305)
};
\addplot[
    color=teal,
     % p=3, r=2, simple, lambda=0.5
    mark=*,
    mark size=3pt,
    dashed,
] coordinates {
   (0, 0.003943648234133346)
    (1, 0.001743940387830982)
    (2, 0.0004517283017511087)
    (3, 0.00011429042456347378)
    (4, 0.000028865649017889225)
    (5, 0.0000072641311075743314)
};
\addplot[
    color=orange!50,
     % p=3, r=1, coplanar, lambda=0.26
    mark=*,
    mark size=3pt,
    thick,
] coordinates {
    (0,0.014843059656796963)
    (1, 0.001685368733857637)
    (2, 0.00013951765324187598)
    (3, 0.000009768414196992953)
    (4, 0.0000007659382918571989)
    (5, 0.0000000796077960218966)
};
\addplot[
    color=orange,
    % p=3, r=1, simple, lambda=0.26
    mark=*,
    mark size=3pt,
    dashed,
] coordinates {
    (0, 0.01477862008052789)
    (1, 0.001669447554754803)
    (2, 0.00013849440850433742)
    (3, 0.000009714097424961404)
    (4, 0.0000007651998099461599)
    (5, 0.00000007928429305095741)
};
\addplot[
    color=violet!50,
     % p=3, r=2, coplanar, lambda=0.26
    mark=*,
    mark size=3pt,
    thick,
] coordinates {
    (0,0.014843059656796963)
    (1, 0.0035111440188994106)
    (2, 0.0002282508195103998)
    (3, 0.000020702992776932758)
    (4, 0.000001979094308632995)
    (5, 0.0000001664706111535741)
};
\addplot[
    color=violet,
    % p=3, r=2, simple, lambda=0.26
    mark=*,
    mark size=3pt,
    dashed,
] coordinates {
    (0, 0.01477862008052789)
    (1, 0.003487588126466483)
    (2, 0.00022694160953355344)
    (3, 0.00002068274078890174)
    (4, 0.000001978220841294606)
    (5, 0.00000016639995828460208)
};
\addplot[
    color=gray,
    thick,
    dotted,
] coordinates {
    (3, 0.0003)
    (4, 0.0003/4)
    (5, 0.0003/16)
};
\addplot[
    color=gray,
    thick,
    dashed,
] coordinates {
    (3, 0.0025/64)
    (4, 0.0025/512)
    (5, 0.0025/4096)
};
\addplot[
    color=gray,
    solid,
] coordinates {
    (3, 0.0004/64)
    (4, 0.0004/1024)
    (5, 0.0004/16384)
};
\legend{{$\lambda=0.5$, $r=1$} cop., {$\lambda=0.5$, $r=1$} sim., {$\lambda=0.5$, $r=2$} cop., {$\lambda=0.5$, $r=2$} sim., {$\lambda=0.26$, $r=1$}  cop., {$\lambda=0.26$, $r=1$} sim., {$\lambda=0.26$, $r=2$}  cop., {$\lambda=0.26$, $r=2$} sim., $h^{2}\sim1/2^{2\ell}$, $h^{3}\sim1/2^{3\ell}$, $h^{4}\sim1/2^{4\ell}$}
\end{axis}
\begin{axis}[
    xshift=.574\textwidth,
    width=.42\textwidth,
    height=.41\textwidth,
    xlabel={Degrees of Freedom (DOFs)},
    xmode=log,
    xtick={10, 100, 1000, 10000, 100000},
    xmin=30, xmax=40000,
    ymin=0.00000003, 
    ymode=log,
    ytick={0.1,0.01,0.001,0.0001,0.00001,0.000001,0.0000001},
    yticklabel pos=right,
    mark options={solid},,
    grid=both,
    tick label style={font=\small},
    label style={},
    legend style={font=\small,at={(0.97,0.97)},anchor=north east},
    legend style={font=\small},
    legend pos=outer north east,
    title={}
]
\addplot[
    color=blue!40,
    % p=3, r=1, coplanar, lambda=0.5
    mark=*,
    mark size=3pt,
    thick,
] coordinates {
    (48, 0.004745925686219142)
    (140, 0.0010809924668855747)
    (468, 0.00025845785944694777)
 (1700,0.00006326915681817424)
 (6468,0.000015404757825965287)
 (25220,0.000003835804647384089)

};
\addplot[
    color=blue,
     % p=3, r=1, simple, lambda=0.5
    mark=*,
    mark size=3pt,
    dashed,
] coordinates {
    
   (48, 0.003943648234133346)
    (140, 0.0011509634238906254)
    (468, 0.00031038353081931386)
    (1700, 0.0000795191415758522)
(6468, 0.000019984214296161346)
(25220, 0.000004600683976438424)
};
\addplot[
    color=green,
     % p=3, r=2, coplanar, lambda=0.5
    mark=*,
    mark size=3pt,
    thick,
] coordinates {
    (48, 0.004745925686219142)
    (88, 0.0017211738224215678)
    (204, 0.0004274517914710785)
    (580, 0.00010396137485646495)
    (1908, 0.000025596479744594003)
    (6868, 0.0000061213278988170305)
};
\addplot[
    color=teal,
     % p=3, r=2, simple, lambda=0.5
    mark=*,
    mark size=3pt,
    dashed,
] coordinates {
   (48, 0.003943648234133346)
    (88, 0.001743940387830982)
    (204, 0.0004517283017511087)
    (580, 0.00011429042456347378)
    (1908, 0.000028865649017889225)
    (6868, 0.0000072641311075743314)
};
\addplot[
    color=orange!50,
     % p=3, r=1, coplanar, lambda=0.26
    mark=*,
    mark size=3pt,
    thick,
] coordinates {
    (48,0.014843059656796963)
    (140, 0.001685368733857637)
    (468, 0.00013951765324187598)
    (1700, 0.000009768414196992953)
    (6468, 0.0000007659382918571989)
    (25220, 0.0000000796077960218966)
};
\addplot[
    color=orange,
    % p=3, r=1, simple, lambda=0.26
    mark=*,
    mark size=3pt,
    dashed,
] coordinates {
    (48, 0.01477862008052789)
    (140, 0.001669447554754803)
    (468, 0.00013849440850433742)
    (1700, 0.000009714097424961404)
    (6468, 0.0000007651998099461599)
    (25220, 0.00000007928429305095741)
};
\addplot[
    color=violet!50,
     % p=3, r=2, coplanar, lambda=0.26
    mark=*,
    mark size=3pt,
    thick,
] coordinates {
    (48,0.014843059656796963)
    (88, 0.0035111440188994106)
    (204, 0.0002282508195103998)
 (580, 0.000020702992776932758)
 (1908, 0.000001979094308632995)
    (6868, 0.0000001664706111535741)
};
\addplot[
    color=violet,
    % p=3, r=2, simple, lambda=0.26
    mark=*,
    mark size=3pt,
    dashed,
] coordinates {
    (48, 0.01477862008052789)
    (88, 0.003487588126466483)
    (204, 0.00022694160953355344)
    (580, 0.00002068274078890174)
    (1908, 0.000001978220841294606)
    (6868, 0.00000016639995828460208)
};
\end{axis}
\end{tikzpicture}
    \caption{$L^\infty$-error convergence for the interpolation problem with $p=3$.}
    \label{fig:interpolation-3}
\end{figure}
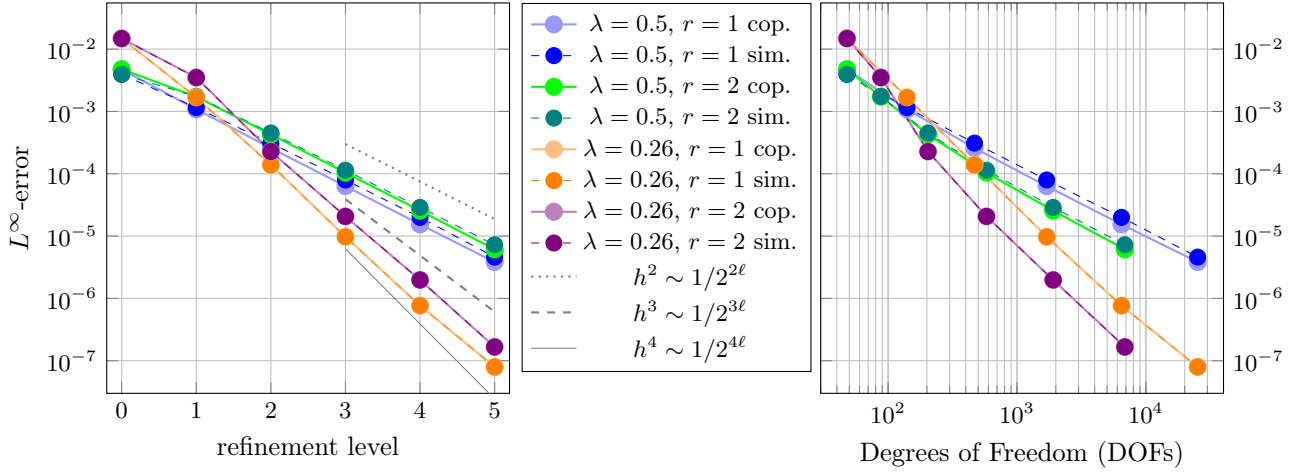

% for p=4
\begin{figure}[h!]
    \centering
\begin{tikzpicture}
\begin{axis}[
    width=.42\textwidth,
    height=.42\textwidth,
    xlabel={refinement level},
    ylabel={$L^\infty$-error },
    xtick={0,1,2,3,4,5},
    xmin=-0.2, xmax=5.2,
    ymode=log,
    ymin=0.00000001,
    ytick={0.1,0.01,0.001,0.0001,0.00001,0.000001,0.0000001},
    mark options={solid},
    grid=both,
    tick label style={font=\small},
    label style={},
    legend style={font=\small},
    legend pos=outer north east,
    title={}
]
\addplot[
    color=blue!40,
    % p=4, r=1, coplanar, lambda=0.5
    mark=*,
    mark size=3pt,
    thick,
] coordinates {
    (0, 0.0018655547566925335)
    (1, 0.00046725072621983665)
    (2, 0.00011270823276060549)
    (3, 0.000027350279409310563)
    (4, 0.000006691630298122364)
    (5,0.0000016643702869245658)

};
\addplot[
    color=blue,
     % p=4, r=1, simple, lambda=0.5
    mark=*,
    mark size=3pt,
    dashed,
] coordinates {
    
   (0, 0.0023833032877768404)
    (1, 0.0006108254316884165)
    (2, 0.0001559756186981457)
    (3, 0.000039479233808360076)
    (4, 0.000009944660702945057)
    (5, 0.000002497241680481793)
};
\addplot[
    color=green,
     % p=4, r=3, coplanar, lambda=0.5
    mark=*,
    mark size=3pt,
    thick,
] coordinates {
    (0, 0.0018655547566925335)
    (1, 0.0008154969598991174)
    (2, 0.00020765083125742767)
    (3, 0.00005118915396359891)
    (4, 0.000012305120081101961)
    (5, 0.000003055517767969196)
};
\addplot[
    color=teal,
     % p=4, r=3, simple, lambda=0.5
    mark=*,
    mark size=3pt,
    dashed,
] coordinates {
   (0, 0.0023833032877768404)
    (1, 0.0009692497750615377)
    (2, 0.0002489317657344836)
    (3, 0.00006294925951687465)
    (4, 0.000015874889398161812)
    (5, 0.0000035430893445022316)
};
\addplot[
    color=orange!50,
     % p=4, r=1, coplanar, lambda=0.26
    mark=*,
    mark size=3pt,
    thick,
] coordinates {
    (0,0.0054757178276895364)
    (1, 0.0005919544501945193)
    (2, 0.0000544487619644372)
    (3, 0.000004516200764057704)
    (4, 0.0000003460012401382273)
    (5, 0.000000025096811902963756)
};
\addplot[
    color=orange,
    % p=4, r=1, simple, lambda=0.26
    mark=*,
    mark size=3pt,
    dashed,
] coordinates {
    (0, 0.00544385152037697)
    (1, 0.0005886805455887698)
    (2, 0.0000543669983344092)
    (3, 0.000004518729310887953)
    (4, 0.0000003467846737340019)
    (5, 0.00000002512916980404456)
};
\addplot[
    color=violet!50,
     % p=4, r=3, coplanar, lambda=0.26
    mark=*,
    mark size=3pt,
    thick,
] coordinates {
    (0, 0.0054757178276895364)
    (1, 0.0014816847088731586)
    (2, 0.00012047389692542744)
    (3, 0.00001141726156653547)
    (4, 0.0000009840209199210581)
    (5, 0.00000007614973384847956)
};
\addplot[
    color=violet,
    % p=4, r=3, simple, lambda=0.26
    mark=*,
    mark size=3pt,
    dashed,
] coordinates {
    (0, 0.00544385152037697)
    (1, 0.001474517504324431)
    (2, 0.00012034249432609249)
    (3, 0.000011419606912251717)
    (4, 0.0000009845238585007299)
    (5, 0.00000007615518861706083)
};
\addplot[
    color=gray,
    thick,
    dotted,
] coordinates {
    (3, 0.0002)
    (4, 0.0002/4)
    (5, 0.0002/16)
};
\addplot[
    color=gray,
    thick,
    dashed,
] coordinates {
    (3, 0.001/64)
    (4, 0.001/512)
    (5, 0.001/4096)
};
\addplot[
    color=gray,
    solid,
] coordinates {
    (3, 0.0002/64)
    (4, 0.0002/1024)
    (5, 0.0002/16384)
};
\legend{{$\lambda=0.5$, $r=1$} cop., {$\lambda=0.5$, $r=1$} sim., {$\lambda=0.5$, $r=3$} cop., {$\lambda=0.5$, $r=3$} sim., {$\lambda=0.26$, $r=1$}  cop., {$\lambda=0.26$, $r=1$} sim., {$\lambda=0.26$, $r=3$}  cop., {$\lambda=0.26$, $r=3$} sim., $h^{2}\sim1/2^{2\ell}$, $h^{3}\sim1/2^{3\ell}$,$h^{4}\sim1/2^{4\ell}$}
\end{axis}
\begin{axis}[
    xshift=.574\textwidth,
    width=.42\textwidth,
    height=.42\textwidth,
    xlabel={Degrees of Freedom (DOFs)},
     xmode=log,
    xtick={10, 100, 1000, 10000, 100000,1000000},
    yticklabel pos= right,
     xmin=50, xmax=70000,
    ymin=0.00000001, 
    ymode=log,
    ytick={0.1,0.01,0.001,0.0001,0.00001,0.000001,0.0000001},
    mark options={solid},
    grid=both,
    tick label style={font=\small},
    label style={},
    legend style={font=\small,at={(0.01,0.01)},anchor=south west},
    legend style={font=\small},
    legend pos=outer north east,
    title={}
]
\addplot[
    color=blue!40,
    % p=4, r=1, coplanar, lambda=0.5
    mark=*,
    mark size=3pt,
    thick,
] coordinates {
    (88, 0.0018655547566925335)
    (280, 0.00046725072621983665)
    (988, 0.00011270823276060549)
    (3700, 0.000027350279409310563)
    (14308, 0.000006691630298122364)
    (56260,0.0000016643702869245658)

};
\addplot[
    color=blue,
     % p=4, r=1, simple, lambda=0.5
    mark=*,
    mark size=3pt,
    dashed,
] coordinates {
    
   (88, 0.0023833032877768404)
    (280, 0.0006108254316884165)
    (988, 0.0001559756186981457)
    (3700, 0.000039479233808360076)
    (14308, 0.000009944660702945057)
    (56260, 0.000002497241680481793)
};
\addplot[
    color=green,
     % p=4, r=3, coplanar, lambda=0.5
    mark=*,
    mark size=3pt,
    thick,
] coordinates {
    (88, 0.0018655547566925335)
    (140, 0.0008154969598991174)
    (280, 0.00020765083125742767)
    (704, 0.00005118915396359891)
    (2128, 0.000012305120081101961)
    (7280, 0.000003055517767969196)
};
\addplot[
    color=teal,
     % p=4, r=3, simple, lambda=0.5
    mark=*,
    mark size=3pt,
    dashed,
] coordinates {
   (88, 0.0023833032877768404)
    (140, 0.0009692497750615377)
    (280, 0.0002489317657344836)
    (704, 0.00006294925951687465)
    (2128, 0.000015874889398161812)
    (7280, 0.0000035430893445022316)
};
\addplot[
    color=orange!50,
     % p=4, r=1, coplanar, lambda=0.26
    mark=*,
    mark size=3pt,
    thick,
] coordinates {
    (88,0.0054757178276895364)
    (280, 0.0005919544501945193)
    (988, 0.0000544487619644372)
    (3700, 0.000004516200764057704)
    (14308, 0.0000003460012401382273)
    (56260, 0.000000025096811902963756)
};
\addplot[
    color=orange,
    % p=4, r=1, simple, lambda=0.26
    mark=*,
    mark size=3pt,
    dashed,
] coordinates {
    (88, 0.00544385152037697)
    (280, 0.0005886805455887698)
    (988, 0.0000543669983344092)
    (3700, 0.000004518729310887953)
    (14308, 0.0000003467846737340019)
    (56260, 0.00000002512916980404456)
};
\addplot[
    color=violet!50,
     % p=4, r=3, coplanar, lambda=0.26
    mark=*,
    mark size=3pt,
    thick,
] coordinates {
    (88, 0.0054757178276895364)
    (140, 0.0014816847088731586)
    (280, 0.00012047389692542744)
    (704, 0.00001141726156653547)
    (2128, 0.0000009840209199210581)
    (7280, 0.00000007614973384847956)
};
\addplot[
    color=violet,
    % p=4, r=3, simple, lambda=0.26
    mark=*,
    mark size=3pt,
    dashed,
] coordinates {
    (88, 0.00544385152037697)
    (140, 0.001474517504324431)
    (280, 0.00012034249432609249)
    (704, 0.000011419606912251717)
    (2128, 0.0000009845238585007299)
    (7280, 0.00000007615518861706083)
};
\end{axis}
\end{tikzpicture}
    \caption{$L^\infty$-error convergence for the interpolation problem with $p=4$.}
    \label{fig:interpolation-4}
\end{figure}
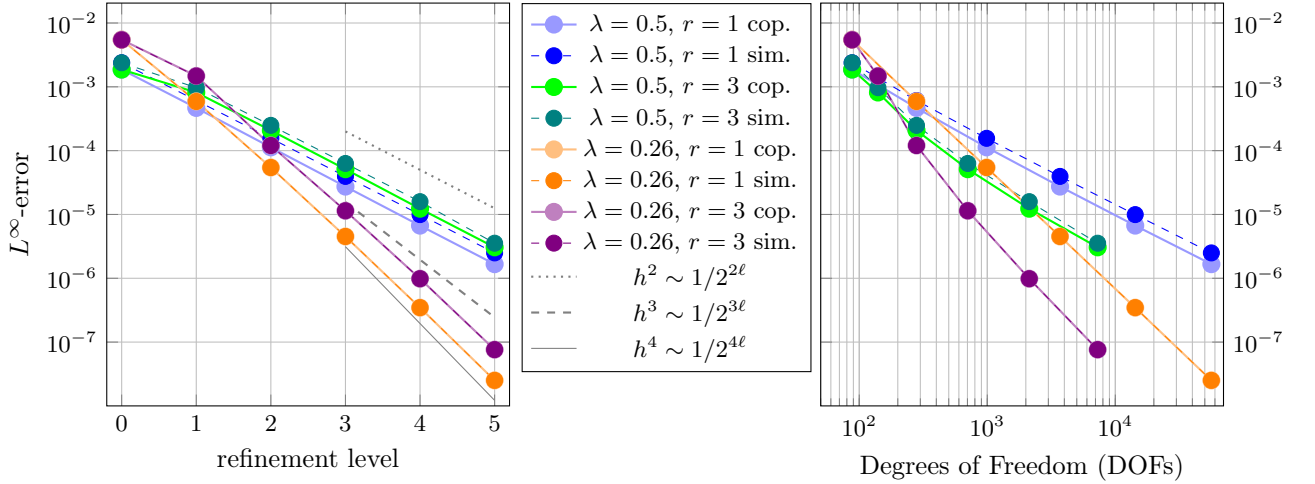

\FloatBarrier

Results of the interpolation problem are reported in Figures~\ref{fig:interpolation-2}--\ref{fig:interpolation-4}. We observe suboptimal convergence rates for $\lambda=0.5$ for all degrees $p$. For $p=2$ and $\lambda=2^{-3/2} \simeq 0.3535534$ we observe the optimal convergence rate $h^3$ for the $L^\infty$-error, in accordance with~\cite{takacs2025approximation}. Similarly, for $p=3$ and $\lambda= 0.26$, which is close to the required scaling $\lambda_{opt.}=2^{-2}$, we observe an almost optimal rate of $h^4$. For $p=4$, the optimal rate $h^5$ is not possible. As indicated by the theory, the best possible rate of $h^4$ is observed.

\subsection{$L^2$-approximation tests on a square domain}\label{sec:L2-approximation}

In the following we present the results of solving an $L^2$-approximation problem for the function given in~\eqref{eq:test-function} on the unit square. The geometry is parameterized as in Section~\ref{sec:interpolation}, cf. Figure~\ref{fig:domains-coplanar}. The results are reported in Figures~\ref{L2-fitting-p2}--\ref{L2-fitting-p4}. Following the reasoning in~\cite[Remark 13(b)]{takacs2025approximation}, the $L^2$-error converges at most with the rate $\sqrt{\ell+1}{2}^{-\ell(p+1)}$, if $\lambda=2^{-\frac{p+1}{3}}$. This rate is slightly suboptimal but asymptotically equal to the optimal rate $h^{p+1}$, as the factor $\sqrt{\ell+1}$ becomes negligible for large $\ell$. In our experiments we observe that this best possible rate is actually achieved for $p=2$ with $\lambda =0.5$, for $p=3$ with $\lambda= 2^{-\frac43}\simeq 0.397$, and for $p=4$ with $\lambda= 2^{-\frac53}\simeq 0.315$. For $p=2$, the choice $\lambda=0.35$, which yields optimal $L^\infty$-rates also yields an optimal rate $h^3$ for the $L^2$-error.

%for p=2
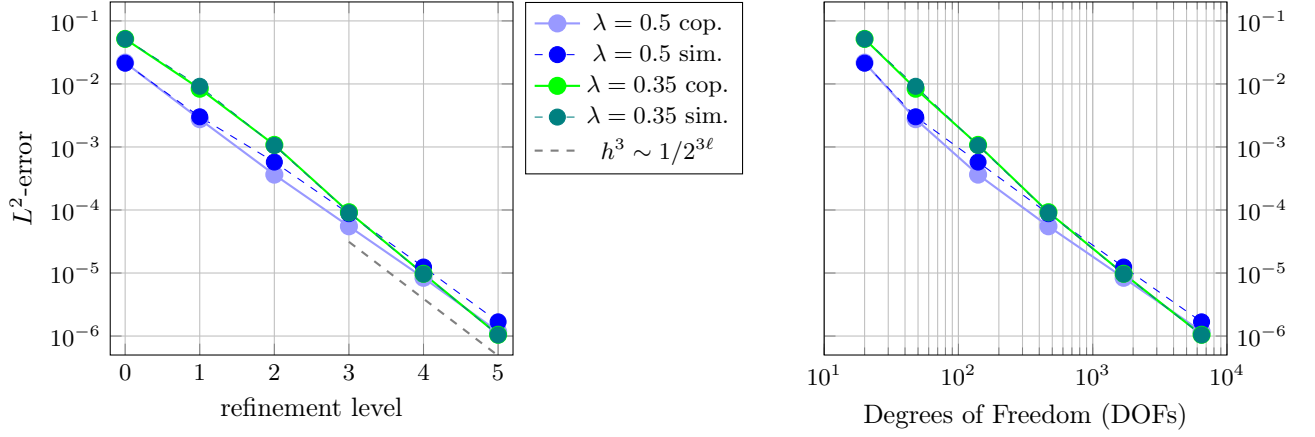
\begin{figure}[h!]
    \centering
\begin{tikzpicture}
\begin{axis}[
    width=.42\textwidth,
    height=.38\textwidth,
    xlabel={refinement level},
    ylabel={$L^2$-error},
    xtick={0,1,2,3,4,5},
     xmin=-0.2, xmax=5.2,
     ymax=0.2, ymin=0.0000005,
    ymode=log,
    ytick={0.1,0.01,0.001,0.0001,0.00001,0.000001},
    mark options={solid},
    grid=both,
    tick label style={font=\small},
    label style={},
    legend style={font=\small},
    legend pos=outer north east,
    title={}
]
\addplot[
    color=blue!40,
    % p=2, r=1, coplanar, lambda=0.5
    mark=*,
    mark size=3pt,
    thick,
] coordinates {
    (0, 0.02178872173470842)
    (1, 0.002759558723805004)
    (2, 0.0003641452417274291)
    (3, 0.00005514837088287701)
    (4, 0.00000839843295487665)
    (5, 0.0000011752542499445876)

};
\addplot[
    color=blue,
     % p=2, r=1, simple, lambda=0.5
    mark=*,
    mark size=3pt,
    dashed,
] coordinates {
    
   (0, 0.021378359384353725)
    (1, 0.003009447338693315)
    (2, 0.0005758766632020123)
    (3, 0.00008732118500184656)
    (4, 0.000012442208261589384)
    (5, 0.0000016800411799934144)
};
\addplot[
    color=green,
    % p=2, r=1, coplanar, lambda=0.35
    mark=*,
    mark size=3pt,
    thick,
] coordinates {
    (0, 0.05179264037548273)
    (1, 0.008396953876813808)
    (2, 0.0010791762810126292)
    (3, 0.00009221756443485846)
    (4, 0.000009727803033312701)
    (5, 0.0000010471725054671455)

};
\addplot[
    color=teal,
     % p=2, r=1, simple, lambda=0.35
    mark=*,
    mark size=3pt,
    dashed,
] coordinates {
    
   (0, 0.05170987719785482)
    (1, 0.009180280872743305)
    (2, 0.0010561729535408055)
    (3, 0.00009042990792463836)
    (4, 0.000009627005212468267)
    (5, 0.0000010399111294356115)
};
\addplot[
    color=gray,
    thick,
    dashed,
] coordinates {
    (3, 0.002/64)
    (4, 0.002/512)
    (5, 0.002/4096)
};
\legend{$\lambda=0.5$ cop., $\lambda=0.5$ sim., $\lambda=0.35$ cop., $\lambda=0.35$ sim., $h^{3}\sim1/2^{3\ell}$}
\end{axis}
\begin{axis}[
    xshift=.574\textwidth,
    width=.42\textwidth,
    height=.38\textwidth,
    xlabel={Degrees of Freedom (DOFs)},
    xmode=log,
    xtick={10, 100, 1000, 10000},
    xmin=10, xmax=10000,
    ymax=0.2, ymin=0.0000005,
    ymode=log,
    ytick={0.1,0.01,0.001,0.0001,0.00001,0.000001,0.0000001},
    yticklabel pos=right,
    mark options={solid},,
    grid=both,
    tick label style={font=\small},
    label style={},
    legend style={font=\small,at={(0.97,0.97)},anchor=north east},
    legend style={font=\small},
    legend pos=outer north east,
    title={}
]
\addplot[
    color=blue!40,
    % p=2, r=1, coplanar, lambda=0.5
    mark=*,
    mark size=3pt,
    thick,
] coordinates {
    (20, 0.02178872173470842)
    (48, 0.002759558723805004)
    (140, 0.0003641452417274291)
    (468, 0.00005514837088287701)
    (1700, 0.00000839843295487665)
    (6468, 0.0000011752542499445876)
};
\addplot[
    color=blue,
     % p=2, r=1, simple, lambda=0.5
    mark=*,
    mark size=3pt,
    dashed,
] coordinates {
    
   (20, 0.021378359384353725)
    (48, 0.003009447338693315)
    (140, 0.0005758766632020123)
    (468, 0.00008732118500184656)
    (1700, 0.000012442208261589384)
    (6468, 0.0000016800411799934144)
};
\addplot[
    color=green,
     % p=2, r=1, coplanar, lambda=0.35355
    mark=*,
    mark size=3pt,
    thick,
] coordinates {
    (20, 0.05179264037548273)
    (48, 0.008396953876813808)
    (140, 0.0010791762810126292)
    (468, 0.00009221756443485846)
    (1700, 0.000009727803033312701)
    (6468, 0.0000010471725054671455)
};
\addplot[
    color=teal,
     % p=2, r=1, simple, lambda=0.35355
    mark=*,
    mark size=3pt,
    dashed,
] coordinates {
   (20, 0.05170987719785482)
    (48, 0.009180280872743305)
    (140, 0.0010561729535408055)
    (468, 0.00009042990792463836)
    (1700, 0.000009627005212468267)
    (6468, 0.0000010399111294356115)
};
\end{axis}
\end{tikzpicture}
    \caption{$L^2$-error convergence for the $L^2$-fitting problem with $p=2$.}
    \label{L2-fitting-p2}
\end{figure}

% l2 fitting for p=3
\begin{figure}[h!]
    \centering
\begin{tikzpicture}
\begin{axis}[
    width=.42\textwidth,
    height=.38\textwidth,
    xlabel={refinement level},
    ylabel={$L^2$-error },
    xtick={0,1,2,3,4},
     xmin=-0.2, xmax=4.2,
    ymode=log,
    ymax=0.01,ymin=0.0000002,
    ytick={0.1,0.01,0.001,0.0001,0.00001,0.000001,0.0000001},
    mark options={solid},
    grid=both,
    tick label style={font=\small},
    legend style={font=\small},
    legend pos=outer north east,
    title={}
]
\addplot[
    color=blue!40,
    % p=3, r=1, coplanar, lambda=0.5
    mark=*,
    mark size=3pt,
    thick,
] coordinates {
    (0, 0.0017951121320426667)
    (1, 0.0003607378720976136)
    (2, 0.00006461826457090689)
    (3, 0.000009399756044758763)
    (4, 0.0000012548628157631394)
};
\addplot[
    color=blue,
     % p=3, r=1, simple, lambda=0.5
    mark=*,
    mark size=3pt,
    dashed,
] coordinates {
    
   (0, 0.0025916873163228)
    (1, 0.000584778017959811)
    (2, 0.00009019726191376172)
    (3, 0.000012396107447561653)
    (4, 0.0000016218095706136437)    
};
\addplot[
    color=green,
     % p=3, r=2, coplanar, lambda=0.5
    mark=*,
    mark size=3pt,
    thick,
] coordinates {
    (0, 0.0017951121320426667)
    (1, 0.0007084533468707583)
    (2, 0.00012654414449693403)
    (3, 0.000016916705412759918)
    (4, 0.0000021672928346713338)
};
\addplot[
    color=teal,
     % p=3, r=2, simple, lambda=0.5
    mark=*,
    mark size=3pt,
    dashed,
] coordinates {
   (0, 0.0025916873163228)
    (1, 0.0010590378904386196)
    (2, 0.00017572927593370374)
    (3, 0.000023160546416200767)
    (4, 0.000002962162557699128)
};
\addplot[
    color=orange!50,
     % p=3, r=1, coplanar, lambda=0.397
    mark=*,
    mark size=3pt,
    thick,
] coordinates {
    (0,0.003929307228266964)
    (1, 0.0007860462027518315)
    (2, 0.00008554559088954404)
    (3, 0.0000074658647583489685)
    (4, 0.0000005847889408826181)
    
};
\addplot[
    color=orange,
    % p=3, r=1, simple, lambda=0.397
    mark=*,
    mark size=3pt,
    dashed,
] coordinates {
    (0, 0.004377903141258081)
    (1, 0.0007697075058841086)
    (2, 0.00008245946870659236)
    (3, 0.0000072170790890251626)
    (4, 0.0000005685638509426571)
    
};
\addplot[
    color=violet!50,
     % p=3, r=2, coplanar, lambda=0.397
    mark=*,
    mark size=3pt,
    thick,
] coordinates {
    (0,0.003929307228266964)
    (1, 0.0014991503679001688)
    (2, 0.00014116227021637178)
    (3, 0.000012200462067098959)
    (4, 0.0000009894540922730503)
};
\addplot[
    color=violet,
    % p=3, r=2, simple, lambda=0.397
    mark=*,
    mark size=3pt,
    dashed,
] coordinates {
    (0, 0.004377903141258081)
    (1, 0.0015032547686040397)
    (2, 0.00013872571715849407)
    (3, 0.000012087173680175539)
    (4, 0.0000009861855863234507)
};
\addplot[
    color=gray,
    thick,
    dotted,
] coordinates {
    (3, 0.00005)
    (4, 0.00005/8)
};
\addplot[
    color=gray,
    thick,
    dashed,
] coordinates {
    (3, 0.000004)
    (4, 0.000004/16)
};
\legend{{$\lambda=0.5$, $r=1$} cop., {$\lambda=0.5$, $r=1$} sim., {$\lambda=0.5$, $r=2$} cop., {$\lambda=0.5$, $r=2$} sim., {$\lambda=0.397$, $r=1$}  cop., {$\lambda=0.397$, $r=1$} sim., {$\lambda=0.397$, $r=2$}  cop., {$\lambda=0.397$, $r=2$} sim., $h^{3}\sim1/2^{3\ell}$, $h^{4}\sim1/2^{4\ell}$}
\end{axis}
\begin{axis}[
   xshift=0.584\textwidth,
    width=.42\textwidth,
    height=.38\textwidth,
    xlabel={Degrees of Freedom (DOFs)},
    xmode=log,
    xtick={10, 100, 1000, 10000},
     xmin=30, xmax=10000,
    ymode=log,
    yticklabel pos=right,
    ymax=0.01,ymin=0.0000002,
    ytick={0.1,0.01,0.001,0.0001,0.00001,0.000001,0.0000001},
    mark options={solid},
    grid=both,
    tick label style={font=\small},
    label style={},
    legend style={font=\small},
    legend pos=outer north east,
    title={}
]
\addplot[
    color=blue!40,
    % p=3, r=1, coplanar, lambda=0.5
    mark=*,
    mark size=3pt,
    thick,
] coordinates {
    (48, 0.0017951121320426667)
    (140, 0.0003607378720976136)
    (468, 0.00006461826457090689)
    (1700, 0.000009399756044758763)
    (6468, 0.0000012548628157631394)

};
\addplot[
    color=blue,
     % p=3, r=1, simple, lambda=0.5
    mark=*,
    mark size=3pt,
    dashed,
] coordinates {
    
   (48, 0.0025916873163228)
    (140, 0.000584778017959811)
    (468, 0.00009019726191376172)
    (1700, 0.000012396107447561653)
    (6468, 0.0000016218095706136437)
    
};
\addplot[
    color=green,
     % p=3, r=2, coplanar, lambda=0.5
    mark=*,
    mark size=3pt,
    thick,
] coordinates {
    (48, 0.0017951121320426667)
    (88, 0.0007084533468707583)
    (204, 0.00012654414449693403)
    (580, 0.000016916705412759918)
    (1908, 0.0000021672928346713338)
    
};
\addplot[
    color=teal,
     % p=3, r=2, simple, lambda=0.5
    mark=*,
    mark size=3pt,
    dashed,
] coordinates {
   (48, 0.0025916873163228)
    (88, 0.0010590378904386196)
    (204, 0.00017572927593370374)
    (580, 0.000023160546416200767)
    (1908, 0.000002962162557699128)
};
\addplot[
    color=orange!50,
     % p=3, r=1, coplanar, lambda=0.397
    mark=*,
    mark size=3pt,
    thick,
] coordinates {
    (48,0.003929307228266964)
    (140, 0.0007860462027518315)
    (468, 0.00008554559088954404)
    (1700, 0.0000074658647583489685)
    (6468, 0.0000005847889408826181)
    
};
\addplot[
    color=orange,
    % p=3, r=1, simple, lambda=0.397
    mark=*,
    mark size=3pt,
    dashed,
] coordinates {
    (48, 0.004377903141258081)
    (140, 0.0007697075058841086)
    (468, 0.00008245946870659236)
    (1700, 0.0000072170790890251626)
    (6468, 0.0000005685638509426571)
    
};
\addplot[
    color=violet!50,
     % p=3, r=2, coplanar, lambda=0.397
    mark=*,
    mark size=3pt,
    thick,
] coordinates {
    (48,0.003929307228266964)
    (88, 0.0014991503679001688)
    (204, 0.00014116227021637178)
    (580, 0.000012200462067098959)
    (1908, 0.0000009894540922730503)
};
\addplot[
    color=violet,
    % p=3, r=2, simple, lambda=0.397
    mark=*,
    mark size=3pt,
    dashed,
] coordinates {
    (48, 0.004377903141258081)
    (88, 0.0015032547686040397)
    (204, 0.00013872571715849407)
    (580, 0.000012087173680175539)
    (1908, 0.0000009861855863234507)
};
\end{axis}
\end{tikzpicture}
    \caption{$L^2$-error convergence for the $L^2$-fitting problem with $p=3$.}
    \label{fig:L2-fitting-p3}
\end{figure}
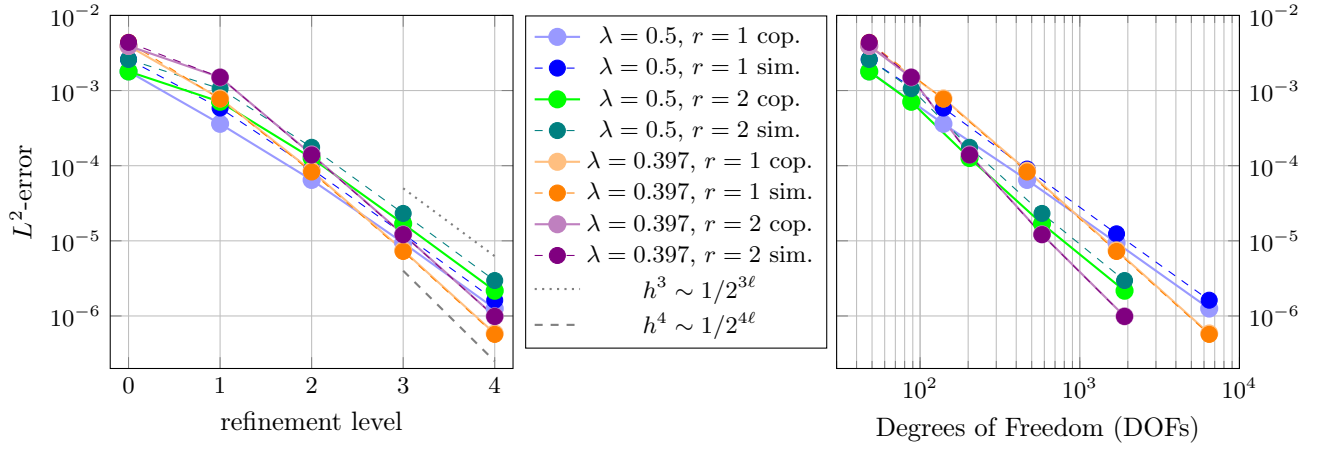

\FloatBarrier

% l2 for fitting p=4
\begin{figure}[h!]
    \centering
\begin{tikzpicture}
\begin{axis}[
    width=.42\textwidth,
    height=.38\textwidth,
    xlabel={refinement level},
    ylabel={$L^2$-error },
    xtick={0,1,2,3},
     xmin=-0.2, xmax=3.2,
    ymode=log,
    ymax=0.005,
    ymin=0.0000001,
    ytick={0.1,0.01,0.001,0.0001,0.00001,0.000001,0.0000001,0.00000001},
    mark options={solid},
    grid=both,
    tick label style={font=\small},
    legend style={font=\small},
    legend pos=outer north east,
    title={}
]
\addplot[
    color=blue!40,
    % p=4, r=1, coplanar, lambda=0.5
    mark=*,
    mark size=3pt,
    thick,
] coordinates {
    (0, 0.0006735254259603027)
    (1, 0.00011607656098523116)
    (2, 0.0000161619025476345)
    (3, 0.000002118774153871677)
};
\addplot[
    color=blue,
     % p=4, r=1, simple, lambda=0.5
    mark=*,
    mark size=3pt,
    dashed,
] coordinates {
    
   (0, 0.001180296517058389)
    (1, 0.0001883081534677559)
    (2, 0.000025098442993193116)
    (3, 0.000003237274886476038)  
};
\addplot[
    color=green,
     % p=4, r=3, coplanar, lambda=0.5
    mark=*,
    mark size=3pt,
    thick,
] coordinates {
    (0, 0.0006735254259603027)
    (1, 0.0002712953447631932)
    (2, 0.000049616098836687255)
    (3, 0.000006817684406662678)
};
\addplot[
    color=teal,
     % p=4, r=3, simple, lambda=0.5
    mark=*,
    mark size=3pt,
    dashed,
] coordinates {
   (0, 0.001180296517058389)
    (1, 0.00042052585632607997)
    (2, 0.0000711433175667463)
    (3, 0.000009588530862056817)
};
\addplot[
    color=orange!50,
     % p=4, r=1, coplanar, lambda=0.315
    mark=*,
    mark size=3pt,
    thick,
] coordinates {
    (0, 0.0020201439305922108)
    (1, 0.00018508603060820835)
    (2, 0.000008891469038136523)
    (3, 0.0000003860661923654516)
};
\addplot[
    color=orange,
    % p=4, r=1, simple, lambda=0.315
    mark=*,
    mark size=3pt,
    dashed,
] coordinates {
    (0, 0.002043709611127006)
    (1, 0.00018216805262939904)
    (2, 0.000008840284435260935)
    (3, 0.00000038691862482822115)
};
\addplot[
    color=violet!50,
     % p=4, r=3, coplanar, lambda=0.315
    mark=*,
    mark size=3pt,
    thick,
] coordinates {
    (0, 0.0020201439305922108)
    (1, 0.000615075666991626)
    (2, 0.000037479975148435455)
    (3, 0.000003696997543507491)
};
\addplot[
    color=violet,
    % p=4, r=3, simple, lambda=0.315
    mark=*,
    mark size=3pt,
    dashed,
] coordinates {
    (0, 0.002043709611127006)
    (1, 0.0006090132629108179)
    (2, 0.00003723825823118342)
    (3, 0.0000036903792717315994)
};
\addplot[
    color=gray,
    thick,
    dotted,
] coordinates {
    (2, 0.00015)
    (3, 0.00015/8)
};
\addplot[
    color=gray,
    thick,
    dashed,
] coordinates {
    (2, 0.000005)
    (3, 0.000005/32)
};
\legend{{$\lambda=0.5$, $r=1$} cop., {$\lambda=0.5$, $r=1$} sim., {$\lambda=0.5$, $r=3$} cop., {$\lambda=0.5$, $r=3$} sim., {$\lambda=0.315$, $r=1$}  cop., {$\lambda=0.315$, $r=1$} sim., {$\lambda=0.315$, $r=3$}  cop., {$\lambda=0.315$, $r=3$} sim., $h^{3}\sim1/2^{3\ell}$, $h^{5}\sim1/2^{5\ell}$}
\end{axis}
\begin{axis}[
    xshift=.585\textwidth,
   width=.42\textwidth,
    height=.38\textwidth,
    xlabel={Degrees of Freedom (DOFs)},
    xmode=log,
    xtick={10, 100, 1000, 10000},
    xmin=60, xmax=5000,
    ymode=log,
    yticklabel pos=right,
    ymax=0.005,
    ymin=0.0000001,
    ytick={0.1,0.01,0.001,0.0001,0.00001,0.000001,0.0000001},
    mark options={solid},
    grid=both,
    tick label style={font=\small},
    label style={},
    legend style={font=\small},
    legend pos=outer north east,
    title={}
]
\addplot[
    color=blue!40,
    % p=4, r=1, coplanar, lambda=0.5
    mark=*,
    mark size=3pt,
    thick,
] coordinates {
    (88, 0.0006735254259603027)
    (280, 0.00011607656098523116)
    (988, 0.0000161619025476345)
    (3700, 0.000002118774153871677)
};
\addplot[
    color=blue,
     % p=4, r=1, simple, lambda=0.5
    mark=*,
    mark size=3pt,
    dashed,
] coordinates {
    
   (88, 0.001180296517058389)
    (280, 0.0001883081534677559)
    (988, 0.000025098442993193116)
    (3700, 0.000003237274886476038)
};
\addplot[
    color=green,
     % p=4, r=3, coplanar, lambda=0.5
    mark=*,
    mark size=3pt,
    thick,
] coordinates {
    (88, 0.0006735254259603027)
    (140, 0.0002712953447631932)
    (280, 0.000049616098836687255)
    (704, 0.000006817684406662678)    
};
\addplot[
    color=teal,
     % p=4, r=3, simple, lambda=0.5
    mark=*,
    mark size=3pt,
    dashed,
] coordinates {
   (88, 0.001180296517058389)
    (140, 0.00042052585632607997)
    (280, 0.0000711433175667463)
    (704, 0.000009588530862056817)    
};
\addplot[
    color=orange!50,
     % p=4, r=1, coplanar, lambda=0.315
    mark=*,
    mark size=3pt,
    thick,
] coordinates {
    (88, 0.0020201439305922108)
    (280, 0.00018508603060820835)
    (988, 0.000008891469038136523)
    (3700, 0.0000003860661923654516)    
};
\addplot[
    color=orange,
    % p=4, r=1, simple, lambda=0.315
    mark=*,
    mark size=3pt,
    dashed,
] coordinates {
    (88, 0.002043709611127006)
    (280, 0.00018216805262939904)
    (988, 0.000008840284435260935)
    (3700, 0.00000038691862482822115)   
};
\addplot[
    color=violet!50,
     % p=4, r=3, coplanar, lambda=0.315
    mark=*,
    mark size=3pt,
    thick,
] coordinates {
    (88, 0.0020201439305922108)
    (140, 0.000615075666991626)
    (280, 0.000037479975148435455)
    (704, 0.000003696997543507491)    
};
\addplot[
    color=violet,
    % p=3, r=2, simple, lambda=0.397
    mark=*,
    mark size=3pt,
    dashed,
] coordinates {
    (88, 0.002043709611127006)
    (140, 0.0006090132629108179)
    (280, 0.00003723825823118342)
    (704, 0.0000036903792717315994) 
};
\end{axis}
\end{tikzpicture}
    \caption{$L^2$-error convergence for the $L^2$-fitting problem with $p=4$.}
    \label{L2-fitting-p4}
\end{figure}
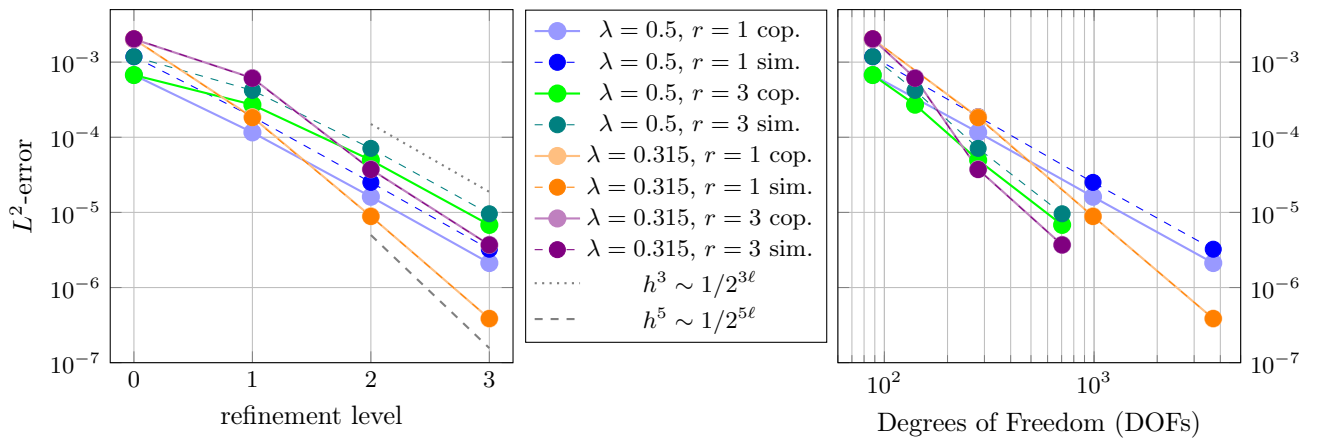

\subsection{Approximation of a sphere}\label{sec:sphere-fitting}

We consider a spherical surface consisting of ten quadrilateral patches. On the initial mesh, the patches are bicubic B\'ezier surfaces. The DOFs are determined by solving an $L^2$-fitting problem onto the unit sphere. The geometry is refined according to the algorithm using coplanar averaging with $\lambda = 0.5$, $\lambda = 0.397$ and $\lambda = 0.26$ for $p=3$ as well as with $\lambda=0.5$ for $p=2$ and then fitted to the sphere again by an $L^2$-projection. In this example, we refine up to the $3$-rd level, as illustrated in Figure~\ref{fig:sphere}. 
We show the resulting surface and patch structure, the surface together with contour lines $x=c$ and a close-up view of the north pole.

\begin{figure}[h!]
\centering
\begin{subfigure}[b]{0.22\textwidth}
\includegraphics[width=0.9\textwidth]{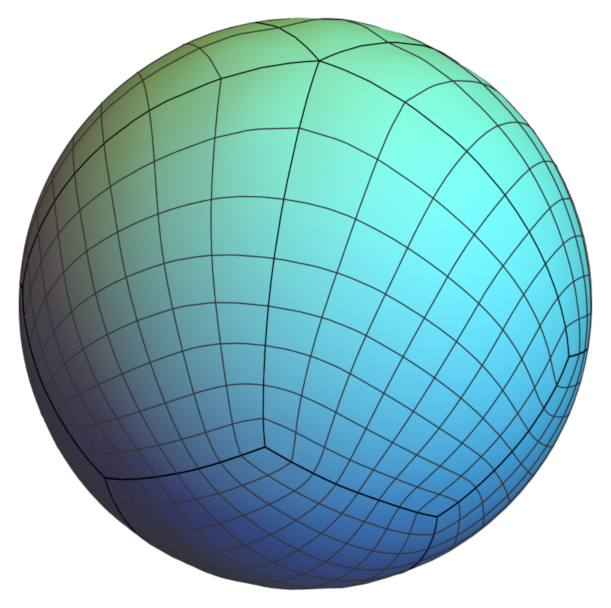} \\
\includegraphics[width=0.9\textwidth]{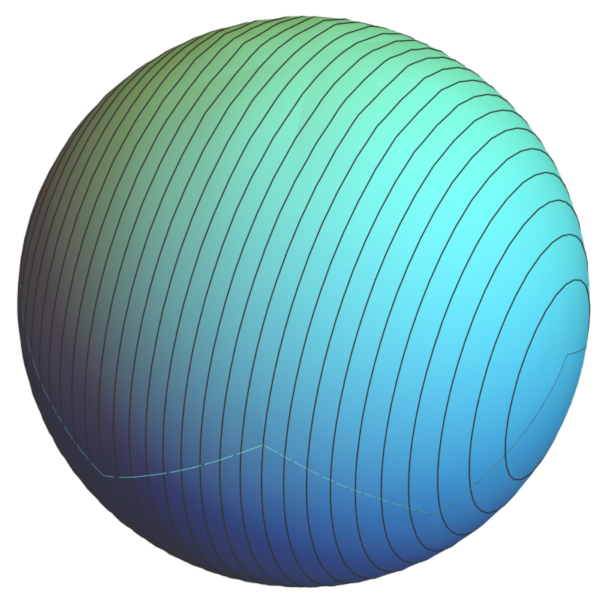} \\
\includegraphics[width=0.9\textwidth]{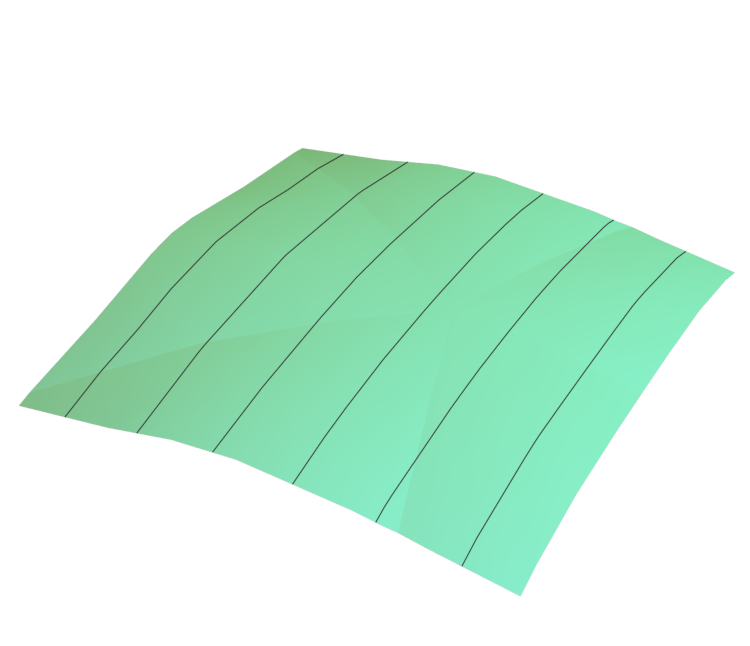} 
\caption{$p=2$, $r=1$, $\lambda=0.5$}\label{fig:sphere-a}
\end{subfigure}
\begin{subfigure}[b]{0.22\textwidth}
\includegraphics[width=0.9\textwidth]{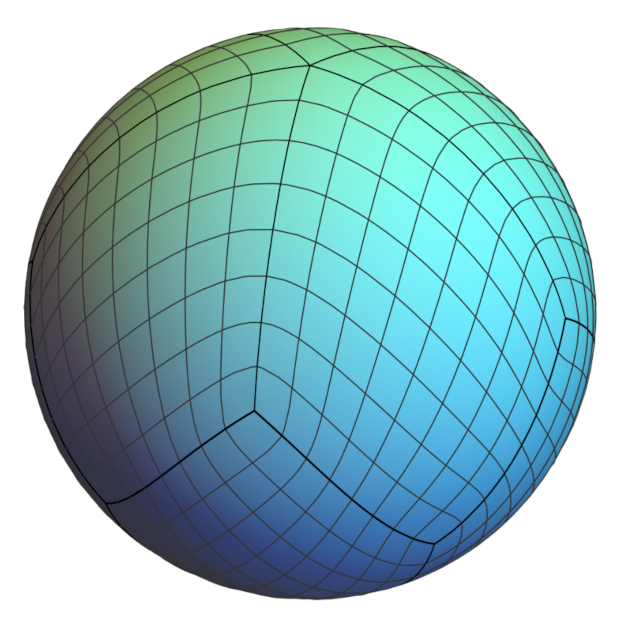} \\
\includegraphics[width=0.9\textwidth]{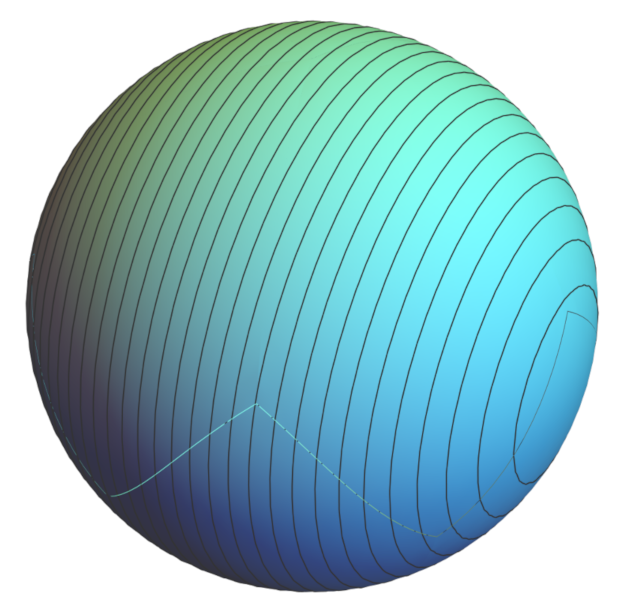} \\
\includegraphics[width=0.9\textwidth]{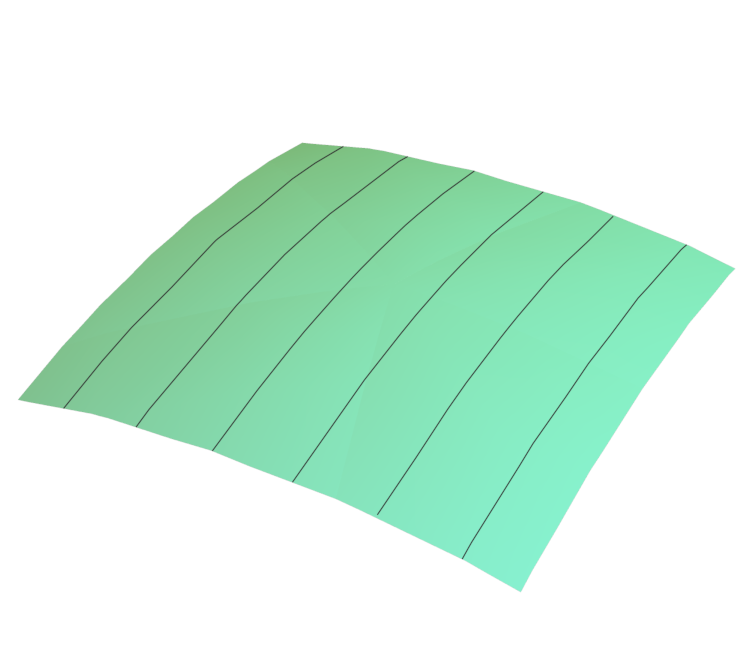} 
\caption{$p=3$, $r=2$, $\lambda=0.5$}\label{fig:sphere-b}
\end{subfigure}
\begin{subfigure}[b]{0.22\textwidth}
\includegraphics[width=0.9\textwidth]{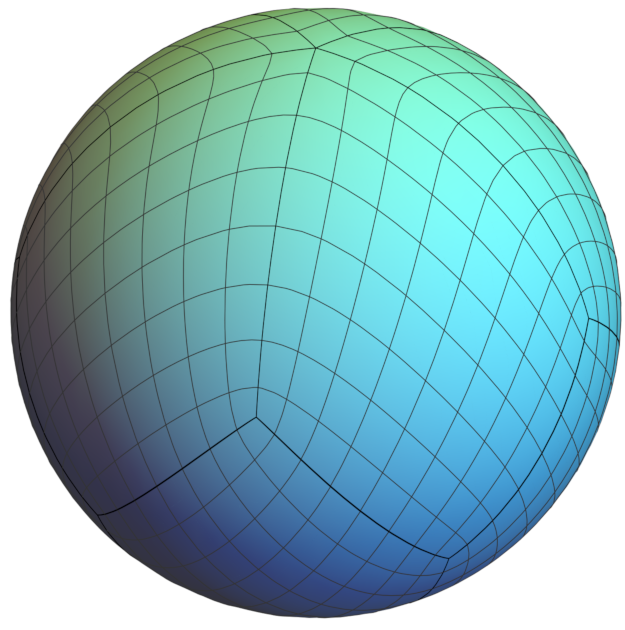} \\
\includegraphics[width=0.9\textwidth]{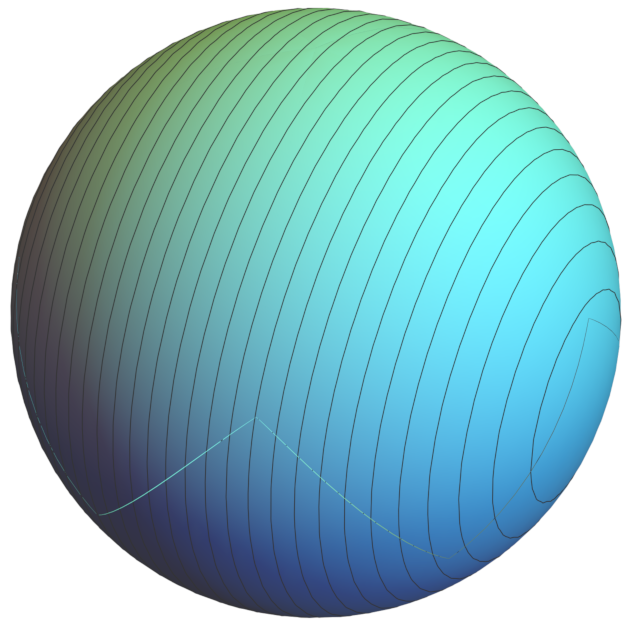} \\
\includegraphics[width=0.9\textwidth]{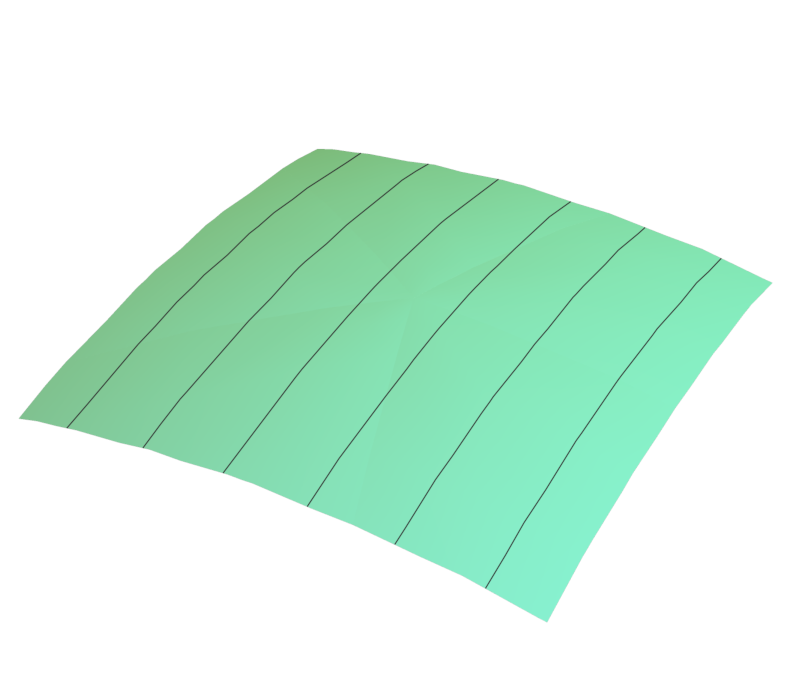} 
\caption{$p=3$, $r=2$, $\lambda=0.397$}\label{fig:sphere-c}
\end{subfigure}
\begin{subfigure}[b]{0.22\textwidth}
\includegraphics[width=0.9\textwidth]{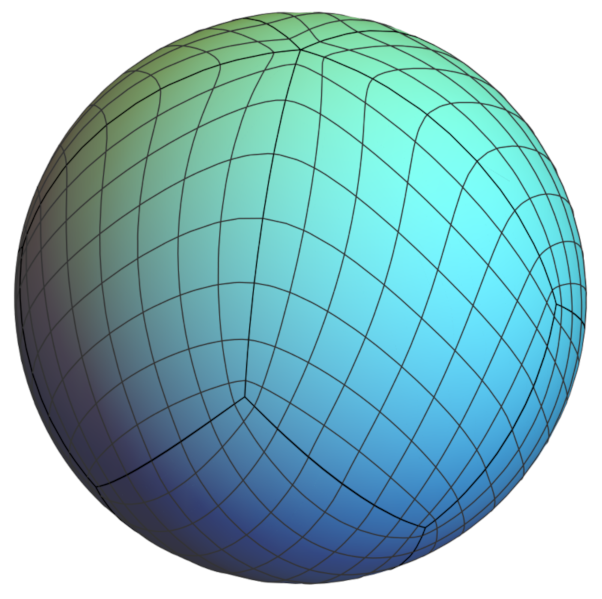} \\
\includegraphics[width=0.9\textwidth]{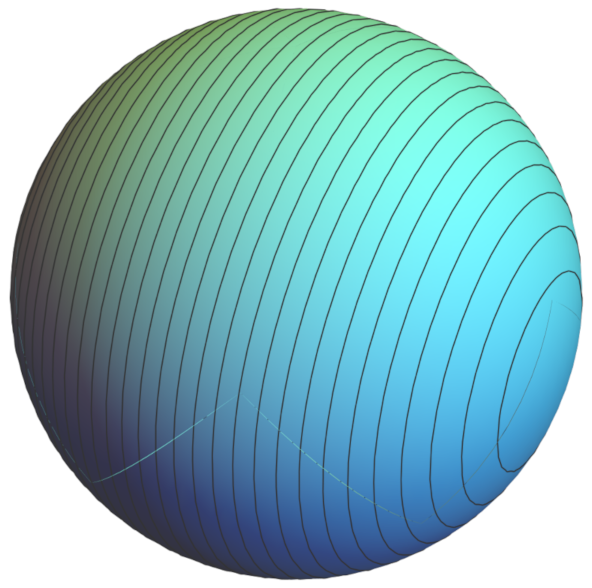} \\
\includegraphics[width=0.9\textwidth]{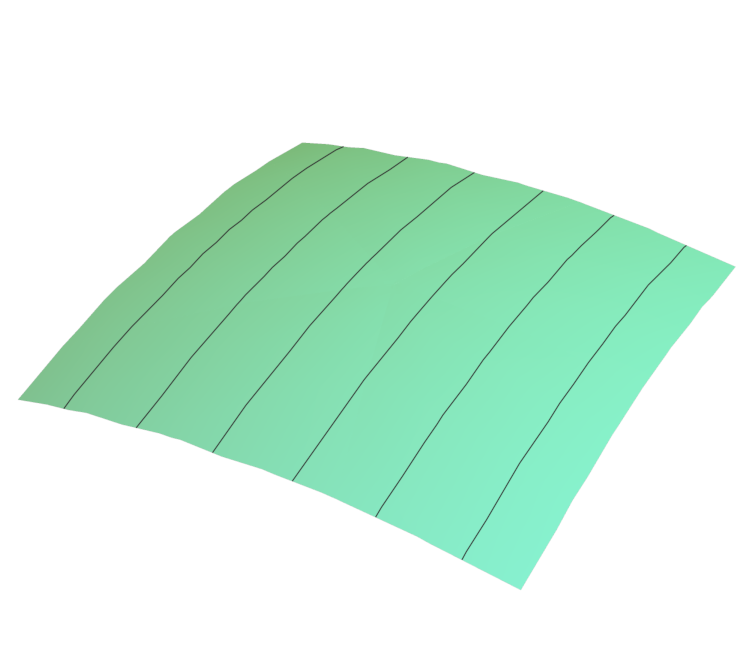} 
\caption{$p=3$, $r=2$, $\lambda=0.26$}\label{fig:sphere-d}
\end{subfigure}
\caption{Approximation of a sphere using $L^2$-fitting. Comparison for varying degree and $\lambda$. Meshes, contour lines $x=c$ and a close-up of the north pole are shown in the top, middle and bottom row, respectively.
}\label{fig:sphere}
\end{figure}

\FloatBarrier

\begin{figure}[h!]
    \centering
\begin{tikzpicture}
\begin{axis}[
    width=.42\textwidth,
    height=.38\textwidth,
    xlabel={refinement level},
    ylabel={$L^2$-error},
    xtick={0,1,2,3},
    ymode=log,
    ytick={0.1,0.01,0.001,0.0001},
    ymax=0.3,ymin=0.00003,
    mark options={solid},
    grid=both,
    tick label style={font=\small},
    label style={},
    legend style={font=\small},
    legend pos=outer north east,
    title={}
]
\addplot[
    color=teal,
     % p=2, r=1, coplanar, lambda=0.5
    mark=*,
    mark size=3pt,
    thick
] coordinates {
    
   (0, 0.2069)
    (1, 0.114854)
    (2, 0.0278763)
    (3, 0.00405671)
};
\addplot[
    color=blue,
    % p=3, r=2, coplanar, lambda=0.5
    mark=*,
    mark size=3pt,
    thick,
] coordinates {
    (0, 0.026112)
    (1, 0.0126498)
    (2, 0.00199528)
    (3, 0.000237102)

};
\addplot[
    color=green,
     % p=3, r=2, coplanar, lambda=0.397
    mark=*,
    mark size=3pt,
    thick,
] coordinates {
    
   (0, 0.0263432)
    (1, 0.0115327)
    (2, 0.00110185)
    (3, 0.0000838388)
};
\addplot[
    color=orange,
    % p=3, r=2, coplanar, lambda=0.26
    mark=*,
    mark size=3pt,
    thick,
] coordinates {
    (0, 0.0598916)
    (1, 0.0226009)
    (2, 0.00172021)
    (3, 0.000129808)

};
\addplot[
    color=gray,
    thick,
    dotted,
] coordinates {
    (2, 0.15/8)
    (3, 0.15/64)
};
\addplot[
    color=gray,
    thick,
    dashed,
] coordinates {
    (2, 0.01/16)
    (3, 0.01/256)
};
\addplot[
    color=gray,
    thick,
    dotted,
    ] coordinates {
    (2, 0.03/8)
    (3, 0.03/64)
};
\legend{{$\lambda=0.5$ cop. $p=2$, $r=1$}, {$\lambda=0.5$ cop. $p=3$, $r=2$}, {$\lambda=0.397$ cop. $p=3$, $r=2$}, {$\lambda=0.26$ cop. $p=3$, $r=2$},  $h^{3}\sim1/2^{3\ell}$, $h^{4}\sim1/2^{4\ell}$}
\end{axis}
\end{tikzpicture}
    \caption{$L^2$-error convergence for the approximation of a sphere.}
    \label{fig:sphere-rates}
\end{figure}
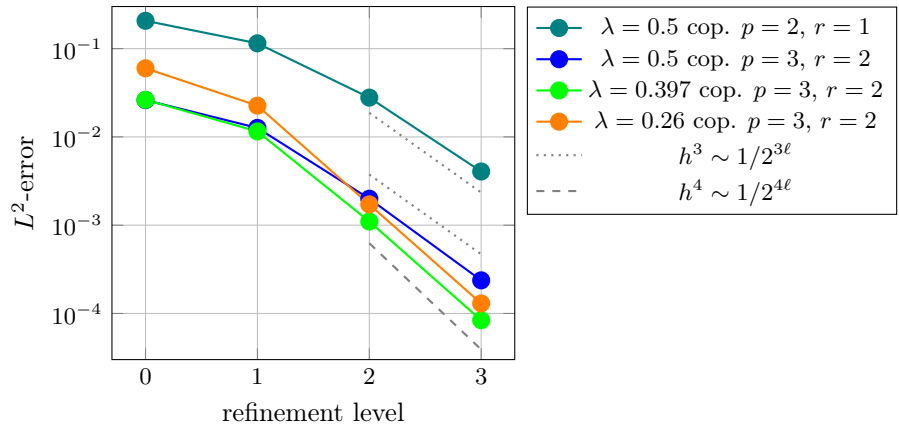

The convergence behavior of the refinement is illustrated in Figure~\ref{fig:sphere-rates}, where the  $L^2$-error is evaluated for successive refinement levels. The plot shows a reduction in error with each level of refinement, and the observed convergence rate for $p=2$, $\lambda = 0.5$ and for $p=3$, $\lambda = 0.5$ is $h^3$. For $p=3$, $\lambda =0.397$ or $\lambda =0.26$ it is $h^4$. In this setting, strong shrinking with $\lambda =0.26$ does not pay off, as the error is larger than with $\lambda =0.397$ throughout all refinement levels.

\section{Conclusions}\label{sec:conclusions}

We developed two families of spline constructions for general degree $p$ over unstructured meshes that is defined through averaging and refinement. The process results in subdivision schemes that are, for $p=2$, similar to almost-$C^1$ splines and Doo--Sabin subdivision. There exists a feasible choice of parameters, such that the resulting surfaces are $C^1$ in the limit and the corresponding basis functions form a local partition of unity. Moreover, for suitably chosen parameters, the basis forms a non-negative partition of unity, thus the splines satisfy a local convex hull property. Both families can be tuned such that the subdominant eigenvalue is $\frac{1}{2}$. While the construction based on simple averaging always yields a non-negative partition of unity, its spectral properties might be limiting, since some sub-sub-dominant eigenvalues are larger than $\frac{1}{4}$. The construction based on coplanar averaging yields nicer spectral properties. However, one needs to select a subdominant eigenvalue $<\frac{1}{2}$ for odd valences to obtain a non-negative partition of unity. 
Our numerical tests indicate that it is possible to recover optimal rates for $L^2$-errors for degrees $p=2,3,4$ by suitable scaling of the subdominant eigenvalue. Optimal rates for the $L^\infty$-errors are achievable for $p=2,3$.

In the future we want to investigate modifications of the construction that allow almost-$C^k$, for $k>1$, that is, higher order continuity, for arbitrary spline degrees. In particular, localized parameter adjustment near extraordinary vertices may enable improved smoothness without compromising locality or convexity. We also want to test wether almost-$C^k$ constructions can further reduce the $C^1$-jump, i.e., the angle between normals near extraordinary vertices. This is of particular interest for solving fourth order PDEs, such as Kirchhoff--Love thin shell problems, using an isogeometric discretization. In the future we want to compare our new spline construction with other existing constructions that can be used to solve fourth order PDEs, as summarized in~\cite{hughes2021smooth,verhelst2024comparison}.

Moreover, for desired $\lambda=\frac{1}{2}$, the approximation power near extraordinary vertices is reduced. We intend to develop further modifications that recover high order approximation without the need to shrink the EV-neighborhoods excessively.

\section*{Acknowledgments}
This research was funded in whole or in part by the Austrian Science Fund (FWF) 10.55776/P37177. This support is gratefully acknowledged.

\bibliography{mybibfile}

\end{document}